\documentclass[11pt]{article}

\usepackage[margin=1in]{geometry}
\usepackage{amsmath,amsthm,amssymb,amsfonts,float}

\makeatletter
\newcommand{\subjclass}[2][2020]{
  \let\@oldtitle\@title
  \gdef\@title{\@oldtitle\footnotetext{#1 \emph{Mathematics subject classification.} #2}}
}
\makeatother

\usepackage[colorlinks=true, pdfstartview=FitV, linkcolor=blue,citecolor=blue, urlcolor=blue]{hyperref}

\usepackage[abbrev,lite,nobysame]{amsrefs}
\usepackage{times}
\usepackage[usenames,dvipsnames]{color}
\usepackage{orcidlink}
\usepackage{esint}

\usepackage{mathtools,enumitem,mathrsfs}

\usepackage[compact]{titlesec}

\usepackage[title]{appendix}

\usepackage{comment}

\usepackage[normalem]{ulem}

\mathtoolsset{showonlyrefs=true}

\usepackage{todonotes}

\colorlet{darkblue}{blue!90!black}
\colorlet{darkred}{red!90!black}
\colorlet{dr}{red!90!black}

\newcommand{\ep}{\epsilon}
\newcommand{\eps}{\epsilon}

\newcommand{\grad}{\nabla}

\newcommand{\Lc}{\mathcal{L}}

\newcommand{\Hc}{{\mathcal{H}}}

\newcommand{\esssup}{\operatorname{EssSup}}

\newcommand{\norm}[1]{\left|\left| #1 \right|\right|}
\newcommand{\abs}[1]{\left| #1 \right|}
\newcommand{\set}[1]{\left\{ #1 \right\}}
\newcommand{\brak}[1]{\left\langle #1 \right\rangle}

\newcommand{\R}{\mathbb{R}}

\newcommand{\Z}{\mathbb{Z}}
\newcommand{\N}{\mathbb{N}}
\newcommand{\T}{\mathbb{T}}

\renewcommand{\S}{\mathbb{S}}

\newcommand{\cM}{\mathcal{M}}

\newcommand{\dee}{\mathrm{d}}
\newcommand{\ds}{\dee s}
\newcommand{\dt}{\dee t}

\DeclareMathOperator{\Div}{\mathrm{div}}
\DeclareMathOperator{\Id}{\mathrm{Id}}

\usepackage{bbm}

\renewcommand{\P}{\mathbf{P}}

\newcommand{\E}{\mathbf{E}}
\newcommand{\EE}{\mathbf E}
\newcommand{\PP}{\mathbf P}

\newcommand{\dist}{\operatorname{dist}}

\newtheorem{theorem}{Theorem}[section]
\newtheorem{proposition}[theorem]{Proposition}
\newtheorem{corollary}[theorem]{Corollary}
\newtheorem{lemma}[theorem]{Lemma}
\newtheorem*{lemma*}{Lemma}

\newtheorem{assumption}{Assumption}

\theoremstyle{definition}
\newtheorem{definition}[theorem]{Definition}
\newtheorem{remark}[theorem]{Remark}
\newtheorem{example}[theorem]{Example}

\numberwithin{equation}{section}
    
\begin{document}

\title{Non-uniqueness of stationary measures for stochastic systems with almost surely invariant manifolds}
\subjclass{Primary: 37H15, 60H10. Secondary: 37D25, 35R60, 37A50}
\author{Jacob Bedrossian\thanks{\footnotesize Department of Mathematics, University of California, Los Angeles, CA 90095, USA \href{mailto:jacob@math.ucla.edu}{\texttt{jacob@math.ucla.edu}}. J.B. was supported by National Science Foundation grant DMS-2108633} \and Alex Blumenthal\thanks{\footnotesize School of Mathematics, Georgia Institute of Technology, Atlanta, GA 30332, USA \href{mailto:ablumenthal6@gatech.edu}{\texttt{ablumenthal6@gatech.edu}}. A.B. was supported by National Science Foundation grants DMS-2009431 and DMS-2237360.}\orcidlink{0000-0002-6777-7848} \and Sam Punshon-Smith\thanks{\footnotesize Department of Mathematics, Tulane University, New Orleans, LA 70118, USA \href{mailto:spunshonsmith@tulane.edu}{\texttt{spunshonsmith@tulane.edu}}. S. P. was  supported by the National Science Foundation under Award No. DMS-1803481 and a Sloan fellowship } \orcidlink{0000-0003-1827-220X}}

\maketitle

\begin{abstract}
  We develop a general framework for establishing non-uniqueness of stationary measures for stochastically forced dynamical systems possessing an almost surely invariant submanifold. Our main abstract result provides sufficient conditions for the existence of multiple stationary measures on compact manifolds, though the underlying methodology extends to non-compact settings. The key insight is to construct additional stationary measures by exploiting the linear instability of the invariant submanifold, as quantified by a positive \emph{transverse Lyapunov exponent}.
  
  To demonstrate the practical applicability of our framework, we apply it to the Lorenz 96 model with degenerate stochastic forcing, which serves as an example of both non-compact and high-dimensional dynamics. We prove that as the damping parameter becomes sufficiently small, the unique stationary measure bifurcates, giving rise to exactly two distinct stationary measures. The proof combines our general theory with computer-assisted verification of certain Lie algebra generation properties that ensure the required hypoellipticity and irreducibility conditions.
\end{abstract}

\setcounter{tocdepth}{2}
{\small\tableofcontents}

\section{Introduction}\label{sec:Intro}

The purpose of this paper is to put forward a general methodology for evaluating the (in)stability of almost-surely invariant submanifolds in systems with random forcing, and for deducing from this instability the existence of stationary statistics supported off the invariant submanifold.

To demonstrate the method on a reasonably complicated concrete example, we study the stochastically-forced Lorenz-96 (L96) system \cite{lorenz1996predictability} with degenerate forcing.
Recall L96 consists of a periodic array of $N >0$ unknowns $u = \{u^{j}\}_{j\in \Z_N}$, solving the stochastic differential equation (SDE)
\begin{align}\label{eq:lorenz96intro}
\dee u^j_t = (u^{j+1}_t - u^{j-2}_t)u^{j-1}_t \dt- \eps u^j_t\dt + \sqrt{\eps} \sigma_j dW^j_t,
\end{align}
note that $u^j$ is indexed by the discrete torus $\Z_N = \Z /N\Z$, so that all indices are interpreted modulo $N$. This ensures terms like $u^{j+1}$ and $u^{j-2}$ are well-defined within the cyclic array. Here, $\eps > 0$\footnote{The \emph{fluctuation-dissipation} scaling in $\epsilon$ presented above is chosen so that the damping term $- \epsilon u^j_t$ and the forcing term $\sqrt{\eps} \sigma_j dW^j_t$ are balanced, ensuring statistically stationary solutions $(u_t)$ exist and are controlled as $\epsilon \to 0$. It is straightforward to check that, on rescaling in time and in $u$, one can recover our main result (Theorem \ref{thm:mainL96}) when the $\sqrt{\eps}$ term in \eqref{eq:lorenz96intro} is omitted.}  is a small parameter, $\sigma_j \in \mathbb R_{\geq 0}$, and the $\{W_t^{j}\}_{j\in \Z_N}$ are iid Brownian motions on the canonical stochastic basis $(\Omega,\mathcal{F},\mathbb P,(\mathcal{F}_t))$. Let $H = \R^{\Z_N} \simeq \R^N$ be the state space of \eqref{eq:lorenz96intro}.

It is known that if two consecutive modes are forced (i.e. $\sigma_j \neq 0$ and $\sigma_{j+1} \neq 0$ for some $j$) then the system is hypoelliptic and has a unique stationary measure for the associated Markov process. Here however we will assume the following \emph{degenerate forcing}, where the forcing is only on every third mode: for some integer $K \geq 3$, we have
\begin{gather}\label{eq:degenforcingIntro}
    N = 3 K \qquad \text{ and } \qquad \sigma_j \neq 0 \quad \text{ iff } \quad j = 3 k,\quad k \in \mathbb{Z} \, .
\end{gather}
Let $I = 3\Z/ N \Z = \{j \in \Z_N\,:\, j \mod 3 = 0\}$ denote the set of forced indices (again interpreted modulo $N$). It is easy to see that the subspace $H_I = \mathrm{span}\{e_j\,:\, j\in I\}$ is almost-surely invariant, on which the dynamics reduces to the independent Ornstein-Uhlenbeck processes
\begin{align}
    \dee u^{j}_t & = - \epsilon u^{j}_t \dt + \sqrt{\eps} \sigma_{j} \dee W_t^{k} \, , \quad j \in I \, .
\end{align}
Here and elsewhere, $\set{e_i}$ denotes the standard basis of $H$.

The subspace $H_I$ admits a unique stationary probability $\mu^I$: this measure is Gaussian on $H_I$, independent of $\eps$, and admits the density
\begin{equation}\label{eq:Gaussian_Density}
\rho_{I}(u_3, u_6, \dots, u_{3K}) =  \prod_{k = 1}^K \frac{1}{(2 \pi \sigma_{3k})^{1/2}} e^{-\frac{ u_{3k}^2}{2\sigma_{3k}}}
\end{equation}
with respect to Lebesgue on $H_I$.

When $\epsilon$ is sufficiently large, the damping is active enough that trajectories of the stochastic flow for \eqref{eq:lorenz96intro} will merge almost-surely\footnote{We say that trajectories of \eqref{eq:lorenz96intro} \emph{synchronize} if for any two initial data $u_0, v_0$ driven by the same noise realizations $(W_t^j)$, one has that $| u_t - v_t| \to 0$ as $t \to \infty$. It is not hard to check that, under mild conditions, strong synchronization implies unique existence of stationary probabilities.} as time advances, hence $\mu_{I}$ is the unique stationary measure. This is standard and can be achieved by an asymptotic coupling argument -- see, e.g., \cite{mattingly1999ergodicity}. However, as $\epsilon$ is taken smaller, the invariant subspace $H_I$ becomes unstable, leading to the emergence of a new stationary measure supported off of $H_I$.

\begin{theorem} \label{thm:mainL96}
Assume the degenerate forcing of \eqref{eq:degenforcingIntro} and assume\footnote{See Remark \ref{rmk:introCAP1} for discussion of the constraint $N \geq 9$. } $N \geq 9$. Then, there exists $\eps_C > 0$ such that for every $\eps < \eps_C$ there are \emph{exactly two} ergodic stationary measures, $\mu_{I}$ as above and a second measure $\mu$ equivalent to Lebesgue measure in $H$, with a smooth density in $H \setminus H_I$. Moreover, $\mu$ is \emph{geometrically ergodic}, in that there exists a function $\mathcal{V} : H \setminus H_I \to [1,\infty)$ and an exponent $\gamma > 0$ such that for all bounded measurable $\varphi : H \to \R$ and initial $u_0 \in H \setminus H_I$,
\begin{align} \label{eq:geomErgodicityL96intro}
    \left|\E_{u_0} \varphi(u_t) - \int \varphi \dee \mu \right| \leq \mathcal{V}(u_0) e^{- \gamma t} \| \varphi\|_{L^\infty} \, .
\end{align}
Above, $\E_{u_0}$ is the expectation conditioned on the initial data $u_0$.
\end{theorem}

By the pointwise ergodic theorem and the absolute continuity of $\mu$, it follows that the long-time statistics of Lebesgue-generic initial data is governed by $\mu$ (not $\mu_{I}$) and hence in this sense, $\mu$ is the \emph{physical} stationary measure in the dynamical sense of \cite{eckmann1985ergodic}. This is true even of initial data supported arbitrarily close to $H_I$, hinting at a strong instability of $H_I$ for the random dynamics generated by \eqref{eq:lorenz96intro}. We note however that necessarily $\mathcal{V}(u_0) \to \infty$ as $\dist(u_0, H_I) \to 0$, which in view of \eqref{eq:geomErgodicityL96intro} captures the transient time $u_t$ remains near $H_I$ when $\dist(u_0, H_I) \ll 1$.

\subsection*{A heuristic instability mechanism for Theorem \ref{thm:mainL96}}

As suggested already, the primary obstruction to the existence of $\mu$ is the asymptotic stability of $H_I$.
The basic idea we use is to assess \emph{instability} of $H_I$ by measuring the \emph{transverse Lyapunov exponent}
\begin{align}\label{eq:defnTverseLEIntro}
    \lambda^\perp = \lim_{n \to \infty} \frac1n \log \abs{\Pi^\perp D \varphi^t(u_0)}, \quad u_0\in H_I.
\end{align}
Here, $\varphi^t : H \to H$ is the stochastic flow of diffeomorphisms corresponding to solutions to \eqref{eq:lorenz96intro}, $D\varphi^t(u_0)$ is the derivative evaluated at a point $u_0$ in the invariant subspace $H_I$, and $\Pi^\perp$ is the projection onto the orthogonal complement to $H_I$, i.e., ``transverse'' to $H_I$.
Heuristically, $\lambda^\perp > 0$ suggests that a small displacement of the initial condition perpendicular to $H_I$ should grow under the nonlinear dynamics of \eqref{eq:lorenz96intro}.

To make this mechanism more precise, we will show that positivity of $\lambda^\perp$ allows us to construct a Lyapunov function $V : H \setminus H_I \to [1,\infty)$,  with $V \to \infty$ near $H_I$, satisfying a Lyapunov-Foster drift condition. Roughly speaking, this is a way of quantifying recurrence to the sublevel sets $\{ V \leq C\}$ and implies existence of stationary probability measures on $H \setminus H_I$. Uniqueness of $\mu$ and geometric ergodicity follows from standard techniques from the theory of Markov chains (see e.g. \cites{meyn2012markov, hairer2011yet}) if one can verify irreducibility and hypoellipticity conditions of the process in $H \setminus H_I$.
 
\subsection*{Plan for the paper}

The remainder of Section \ref{sec:Intro} discusses our results and their relationship to existing literature. Section \ref{sec:Abstract} establishes a general framework for random dynamical systems that admit an almost-surely invariant submanifold. Within this framework, we develop conditions connecting the positivity of the transverse Lyapunov exponent \eqref{eq:defnTverseLEIntro} to the existence of stationary statistics off the invariant submanifold. This analysis is initially conducted in the simpler setting of a compact phase space. In Section \ref{sec:OutlineL96}, we outline the application of this instability mechanism to the Lorenz 96 system described in Theorem \ref{thm:mainL96}, emphasizing additional technical steps required to overcome challenges posed by the noncompact state space.
The remainder of the paper -- Sections \ref{sec:Hypoellipticity}, \ref{sec:FI} and \ref{sec:Drift} -- implement this program. See Section \ref{subsec:agenda3} at the end of Section \ref{sec:OutlineL96} for a more detailed summary of this later material.

\subsection*{Discussion}

Classical linearization theory provides tools to assess stability or instability of relatively simple invariant structures in phase space, such as equilibria (via spectral theory of linearization) or periodic orbits (via Floquet exponents). However, stability problems become considerably more challenging for invariant sets with complicated interior dynamics. A prime example is the 3D Navier-Stokes equations on a periodic box, which admits an invariant subspace $H_I$ of velocity fields constant along the $z$-axis. Under certain degenerate forcing conditions, $H_I$ is preserved, with dynamics equivalent to those of 2D Navier-Stokes. At high Reynolds numbers, it is predicted that $H_I$ becomes strongly unstable, with generic initial velocity fields in $H$ being repelled from $H_I$.

In such situations, transverse Lyapunov exponents analogous to \eqref{eq:defnTverseLEIntro} offer a natural approach. While these limits can be shown to exist for initial conditions typical with respect to invariant probability measures on $H_I$, severe practical limitations arise. Lyapunov exponents are notoriously difficult to bound from below, even for simple models with convincing numerical evidence. Additionally, the potential presence of multiple ergodic invariant measures, each with a distinct transverse exponent, further complicates the analysis.

The random setting provides a more tractable framework for addressing these stability problems, as stochastic driving introduces a regularizing effect on asymptotic statistics. In this context, stationary measures—invariant when averaged over noise realizations—replace invariant measures, and established criteria can demonstrate uniqueness of the stationary measure. The transverse Lyapunov exponent associated with this unique stationary measure is guaranteed to converge by the multiplicative ergodic theorem, and estimating such exponents from below is significantly more feasible than in deterministic systems, as demonstrated in, e.g., \cite{BBPS20,blumenthal2017lyapunov, lian2012positive, chemnitz2023positive}.

A cornerstone of our analysis for the Lorenz-96 system is the rigorous verification of various forms of H\"ormander's condition, which is essential for establishing the hypoellipticity and irreducibility of the dynamics. These properties, in turn, are fundamental for proving the existence and uniqueness of the stationary measures and their geometric ergodicity as well as establishing quantitative estimates for showing positivity of the transverse Lyapunov exponent. Often the most challenging step in this process is verifying the algebraic bracket-spanning requirement of projective lifts (see for instance \cite{BedrossianPunshon-Smith-Chaos-2024y}). In this paper, we establish the fundamental algebraic generation condition on certain collection of traceless matrices $M_k = DB(e_k)|_{H_I^\perp}$, $k\in I$, namely that $\mathrm{Lie}(\{M_k\}) = \mathfrak{sl}(H_I^\perp)$ (Proposition \ref{prop:slT_generation} in Section \ref{sec:Hypoellipticity}). This is achieved through a computer-assisted proof, detailed in Appendix \ref{app:CAP}, which combines symbolic computation for a base case with an argument based on the system's shift-invariance to extend the result to all big enough $N$. The code for this verification is publicly available \cite{L96CAPGithub}. The novelty of our computer-assisted approach lies in its exploitation of the system's sparsity and shift-invariance to verify the bracket condition. This technique, distinct from methods like algebraic variety computations employed for systems with less sparse interaction matrices (e.g., \cite{BedrossianPunshon-Smith-Chaos-2024y}), is particularly well-suited for analyzing other high-dimensional SDEs with local-in-frequency interactions, such as certain shell models of turbulence (e.g., GOY, SABRA).

\begin{remark}\label{rmk:introCAP1}

The constraint $N \ge 9$ in Theorem \ref{thm:mainL96} arises from the algebraic bracket-spanning condition (Proposition \ref{prop:slT_generation}). While our detailed computer-assisted proof in Appendix \ref{app:CAP} focuses on $N \ge 15$ due to its reliance on a sufficiently large local block of indices for the shift-invariance argument, the result can be extended to $N=9$ and $N=12$ by direct computation. For smaller system sizes, specifically $N=3$ and $N=6$, the structure of the transverse space $H_I^\perp$ and the generating matrices $M_k$ becomes significantly more degenerate and the Lie algebra $\mathfrak{sl}(H_I^\perp)$ is in fact not generated.
\end{remark}

\subsection*{Relation to prior work}

Our methodology, using transverse Lyapunov exponents, was inspired by approaches to the instability of the diagonal in two-point motions associated with chaotic stochastic flows, where the true Lyapunov exponent plays the role of $\lambda^\perp$. {To the authors' best knowledge, this approach to analysis of the diagonal of the two-point process originates in} the works \cite{baxendale1988large,dolgopyat2004sample} (see also the excellent related survey \cite{baxendale1991statistical}) and has been extended in various ways in subsequent studies, including \cite{ayyer2007exponential} and \cite{bedrossian2022almost}. These works collectively demonstrate the power of Lyapunov exponents (and the associated Feynman-Kac semigroup) in analyzing stability properties of invariant structures in stochastic dynamical systems, providing the foundation upon which our current analysis builds.

{Closely related works include \cite{coti2021noise} and \cite{hani2025non}, both of which study SDE with almost-sure invariant subsets, using a method parallel to our approach based on the dominant eigenfunction of an appropriately-chosen  Feynman-Kac semigroup to build a Lyapunov function. The work \cite{coti2021noise} studies a degenerately-forced version of the classical Lorenz '63 ODE on $\R^3$, while \cite{hani2025non} studies a degenerately-forced system of three coupled oscillators. Our work proposes a general framework for answering these kinds of non-uniqueness questions.}

We also acknowledge the method of \emph{average Lyapunov functions} and H-exponents, notably developed for population ecology models where invariant subsets often represent species extinction \cite{benaim2018stochastic, hening2018coexistence}. This framework provides general criteria for fundamental questions of extinction or \emph{persistence} (long-term survival of species), often employing Lyapunov functions with, for example, logarithmic growth near the boundary \cite{benaim2018stochastic}. While the underlying concept of quantifying transverse growth (via H-exponents or ``invasion rates'') is related to our use of transverse Lyapunov exponents, our work focuses on the subsequent challenge of identifying and characterizing new statistical states emerging from such instabilities. Furthermore, in many ecological applications, the transverse dynamics effectively simplify to one-dimensional dynamics, alleviating the need for the systematic treatment of multi-dimensional projective cocycles, the construction of Lyapunov functions with stronger (e.g., algebraic) repulsion via Feynman-Kac theory, and non trivial use of advanced regularity tools (like H\"ormander regularity theory) on projective space that are central to our approach, particularly in complex, high-dimensional systems.

For additional related work on the use of Lyapunov exponents to study (in)stability of almost-sure fixed points, see, e.g., \cite{baxendale2006invariant}, and for more from the perspective of bifurcations for almost-sure fixed points, see, e.g., \cite{baxendale1994stochastic} and citations therein.

\section{Abstract result} \label{sec:Abstract}

Our aim in Section \ref{sec:Abstract} is to present, in a simplified setting, an abstract criterion for the existence of stationary measures off of an almost-surely invariant submanifold. This setting, that of IID random diffeomorphisms of a compact, boundaryless manifold, avoids many technical complications to be dealt with in applications to unbounded systems like Lorenz 96 (Theorem \ref{thm:mainL96}), while at the same time exhibiting some surprising subtleties to the approach of this paper.

In Section \ref{subsec:settingMainResultAbstr2} below we lay out the setting and main result, Theorem \ref{thm:absMain2}. After some discussion and a brief outline of the proof to come, Section \ref{subsec:prelimsAbs2} handles some preliminary results and Section \ref{subsec:completeAbsPf2} ties the proof together.

\subsection{Assumptions and statement of Theorem \ref{thm:absMain2}} \label{subsec:settingMainResultAbstr2}

Let $(\Omega, \mathcal{F}, \P)$ be a probability space and let $f_1, f_2, \dots$ be independent, identically distributed (IID) diffeomorphisms of a compact Riemannian manifold $M$ without boundary\footnote{In this section, the `ambient' state space $M$ plays the role of the space $H = \R^{3K}$ in other sections, while the almost-surely invariant submanifold $N$ plays the role of the invariant subspace $H_I$. These notational choices are made to reinforce that the instability implications of transverse Lyapunov exponents apply in nonlinear state spaces.}.

\begin{assumption}\label{ass:fiC2bounded}
   \[\esssup \| f_i\|_{C^2}, ~\esssup \| f_i^{-1}\|_{C^2} < \infty\]
\end{assumption}

\begin{assumption}\label{ass:invariantSubmanifold2}
    There is a nonempty, compact, boundaryless manifold $N \subset M$ for which
    \begin{align}
        f_i(N) \subset N \quad \text{ with probability 1 \, .}
    \end{align}
\end{assumption}

We will consider the dynamics of the random compositions \[f^n := f_n \circ \dots \circ f_1 \,. \]
Given a fixed initial $x_0 \in M$, let $(x_n)$ denote the Markov chain on $M$ generated by $(f^n)$ given by \[x_n = f^n(x_0) \, .\]
Theorem \ref{thm:absMain2} below provides a sufficient condition for the existence of a stationary measure $\mu$ on $M$ for the Markov chain $(x_n)$ supported off of the almost-surely invariant submanifold $N$. These conditions are stated in terms of the \emph{transverse Lyapunov exponent}, defined precisely below.

\begin{definition}\label{defn:transverseStuff2} \
    \begin{itemize}
        \item[(a)] The \emph{transverse bundle} $T^\perp N \subset TM$ is the subbundle consisting of pairs $(y, w)$ for $y \in N$ and $w \in T_{y} M$ such that $w$ is orthogonal to $T_{y} N$.
        \item[(b)] The \emph{transverse process} $(y_n, w_n)$ on $T^\perp N$ is defined, for fixed initial $(y_0, w_0) \in T^\perp N$, by
        \begin{align}
            y_n = f_n(y_{n-1}) \, , \qquad w_n = \Pi^\perp_{y_n} D_{y_{n-1}} f_n (w_{n-1})
        \end{align}
        \item[(c)] The \emph{transverse projective process} $(y_n, v_n)$ on $\S^\perp N \subset T^\perp N$, the subbundle of unit vectors in $T^\perp N$, is defined\footnote{Throughout, when it is clear from context we write $|\cdot|$ for the norm on $T_x M$ coming from the Riemannian metric at a given $x \in M$.} by
        \begin{align}
            v_n = \frac{w_n}{|w_n|}.
        \end{align}
    \end{itemize}
\end{definition}

Note that since $f_i(N) \subset N$ with probability 1, it holds that $D_y f_i(T_y N) = T_{f_i(y)} N$ holds  almost-surely for all $y \in N$, hence
\[\Pi^\perp_{f^n(y_0)} D_{y_0} f^n = \Pi^\perp_{f^n (y_0)} D_{f^{n-1} y_0} f_n \circ \dots \circ \Pi^\perp_{f^1(y_0)} D_{y_0} f_1 \, .  \]

For exponential growth rates of compositions of $(\Pi^\perp D f_i)$ we have the following.
\begin{proposition}\label{prop:LEandMET2} \
    \begin{itemize}
        \item[(a)] Assume that $(y_n)$ admits a unique stationary measure $\mu_N$ on $N$. Then, the limit \[\lambda^\perp = \lim_{n \to \infty} \frac1n \log \| \Pi^\perp_{y_n} D_{y_0} f^n\|  \] exists and is constant $\P \times \mu_N$ almost-surely.
        \item[(b)] Assume that $(y_n, v_n)$ admits a unique\footnote{Note that the marginal of $\nu^\perp$ on $N$ is a stationary measure for $(y_n)$.  In particular, it is not hard to check in this compact setting that unique existence of a stationary measure $\nu^\perp$ for $(y_n, v_n)$ implies unique existence of a stationary measure for $(y_n)$ itself.
        } stationary measure $\nu^\perp$ on $\S^\perp N$. Then, for $\mu_N$-a.e. $y_0 \in N$ and for any $w_0 \in T^\perp_{y_0} N$, we have that
        \[\lambda^\perp = \lim_{n \to \infty} \frac1n \log |w_n| \qquad \text{ with probability 1}\]
    \end{itemize}
\end{proposition}
Item (a) is a standard consequence of the multiplicative ergodic theorem applied to the random compositions $(f_i)$ (see, e.g., \cite[Theorem III.1.1]{kifer2012ergodic}), while (b) follows on realizing Lyapunov exponents as additive observables of the corresponding projective process-- see, e.g., \cite[Theorem III.1.2]{kifer2012ergodic}.

The following is our main result, containing some terms that have yet to be defined.
\begin{theorem} \label{thm:absMain2}
    Let Assumptions \ref{ass:fiC2bounded} and \ref{ass:invariantSubmanifold2} hold, and moreover, assume
    \begin{itemize}
        \item[(i)] $(y_n, v_n)$ is uniformly geometrically ergodic; and
        \item[(ii)] $\lambda^\perp > 0$.
    \end{itemize}
    Then, there exists a stationary measure $\mu$ for the original chain $(x_n)$ on $M$ for which \[\mu(N) = 0 \,.  \]
    In particular, there are at least two stationary measures for $(x_n)$.
\end{theorem}

\begin{definition}\label{defn:unifGeoErgodic2}
    A Markov chain $(z_n)$ on a compact metric space $Z$ is called \emph{uniformly geometrically ergodic} if it admits a unique stationary measure $\eta$ with the property that there exists $C > 0, r \in (0,1)$ such that
    \begin{align}
        \left| \E_{z_0} [\varphi(z_n)] - \int \varphi \,\dee \eta \right| \leq C r^n
    \end{align}
    for all $z_0 \in Z$ and $\varphi : Z \to \R$ continuous.
\end{definition}
Here, for a Markov chain $(z_n)$ we write $\P_{z_0}, \E_{z_0}$ for the probability and expectation conditioned on the specified value of $z_0$.
Note that uniform geometric ergodicity implies uniqueness of the stationary measure, so hypothesis (i) and Proposition \ref{prop:LEandMET2} imply the existence of $\lambda^\perp$ as in hypothesis (ii). For more on methods for checking geometric ergodicity in concrete systems, see \cite{meyn2012markov}.

\subsubsection*{Summary of the proof of Theorem \ref{thm:absMain2}}

We will show that in our setting, $\lambda^\perp > 0$ implies a \emph{drift condition} for the Markov process $(x_n)$ on $M \setminus N$.

\begin{definition}\label{defn:driftCond}
    Let $(z_n)$ be a Markov chain on a complete, separable metric space $Z$. We say that a function $\mathcal{V} : Z \to [1,\infty)$ satisfies a \emph{drift condition} for $(z_n)$ if there exists $\alpha \in (0,1), \beta > 0$ and a compact $K \subset Z$ such that
    \[\E_{z_0} \mathcal{V}(z_1) \leq \alpha \mathcal{V}(z_0) + \beta {\bf 1}_K(z_0)\]
    for all $z_0 \in Z$, where ${\bf 1}_K$ is the indicator function of $K$. The function $\mathcal{V}$ is sometimes referred to as a \emph{Lyapunov function}.
\end{definition}

The following is standard-- see, e.g., \cite{meyn2012markov}.
\begin{theorem} \label{thm:driftCondAbs2}
    Suppose that
    \begin{itemize}
        \item[(i)] $(z_n)$ has the \emph{Feller}\footnote{We say that $(z_n)$ is \emph{Feller} if $z_0 \mapsto \E_{z_0} \varphi(z_1)$ is bounded and continuous for any bounded and continuous $\varphi : Z \to \R$.} property;
        \item[(ii)] there exists a function $\mathcal{V} : Z \to [1,\infty)$ satisfying a drift condition for $(z_n)$; and
        \item[(iii)] the function $\mathcal{V}$ has compact sublevel sets $\{ \mathcal{V} \leq C\}, C \geq 1$.
    \end{itemize}
    Then, $(z_n)$ admits a stationary probability measure.
\end{theorem}

Roughly speaking, Markov chains on noncompact spaces can drift off indefinitely, breaking the recurrence-type behavior necessary for the existence of a stationary measure. A drift condition ensures a positive asymptotic frequency of returns to sufficiently large sublevel sets of $\mathcal{V}$. If these sublevel sets are compact, then it follows that for fixed initial $z_0 \in Z$ that the sequence $(\eta_n)$ is tight, where $\eta_n$ is the law of $z_n$.
Existence now follows from the Krylov-Bogoliubov argument, which obtains a stationary measure as a weak$^*$ limit of the empirical averages
\[\frac1n \sum_0^{n-1} \eta_i \, . \]
That such a weak$^*$ limit is stationary follows from the Feller property, which we assume here. For further details see \cite[]{meyn2012markov}.

In our case, we seek to apply Theorem \ref{thm:driftCondAbs2} to the Markov chain $(x_n)$ on the noncompact space $Z := M \setminus N$. Here we view $N$ as being ``at infinity'' for the purposes of the drift condition, which will now require that $\mathcal{V}: M \setminus N \to [1,\infty)$ have the property that $\mathcal{V}(x) \to \infty$ as $x \to N$.

To this end, and in view of the repulsion mechanism indicated earlier, we will construct $\mathcal{V}$ near $N$ to be of the form
\begin{align}
    \mathcal{V}(x) = \frac{1}{\dist(x, N)^p} \psi(y(x), v(x))
\end{align}
for some $p > 0$, where $y(x) \in N$ is a suitably chosen point in $N$, and $v(x) = w(x) / |w(x)| \in \S^\perp_{y(x)} N$ where $w(x)$ is (approximately) the displacement between $x$ and $y(x)$.

{The function $\psi : \S^\perp N \to \R_{\geq 0}$ itself will be constructed as an eigenfunction of a Feynman-Kac semigroup built from $(y_n, v_n)$. This construction is an adaptation to our setting of a known technique for drift conditions for repulsion from the diagonal for two-point processes under analogous conditions-- see, e.g., \cites{baxendale1988large,dolgopyat2004sample, bedrossian2022almost, ayyer2007exponential}}.

\subsection{Subtleties of the transverse stability condition}\label{subsec:transverseStabilitySubtleties}

Before proceeding with the technical preliminaries, we clarify an important subtlety regarding the assumption that $\lambda^\perp > 0$. The following example illustrates how positive transverse Lyapunov exponents can coexist with regions of transverse attraction.

\begin{example}\label{ex:mixedStabilityTorus}
    Let $M = \T^2$ and let $N$ be an embedded circle in $M$. Let us assume that the random dynamics $(f_i)$ leave a point $p \in N$ almost-surely invariant. Relative to $N$, we will assume that the point $p$ is a sink, but transversal to $N$ we will assume that $p$ is unstable. Here, the relevant stationary measure $\mu_N$ on $N$ is merely the Dirac mass supported at $p$.
    
    Theorem \ref{thm:absMain2} implies the existence of a stationary measure supported off of $N$ (i.e., $\mu(N) = 0$). Somewhat surprisingly, \emph{this state of affairs is compatible with compression onto $N$ away from $p$}. Key here is the fact that the dynamics on $N$ is geometrically ergodic: while trajectories in $M \setminus N$ may temporarily collapse onto $N$, entrainment to the dynamics on $N$ forces them to enter a vicinity of $p$, where the trajectory now experiences repulsion from $N$. This mechanism is illustrated in Figure \ref{fig:torus-example}.

   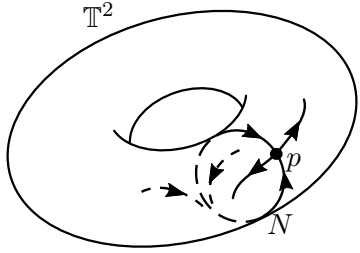
\begin{figure}[H]
   \centering
\begingroup%
  \makeatletter%
  \providecommand\color[2][]{%
    \errmessage{(Inkscape) Color is used for the text in Inkscape, but the package 'color.sty' is not loaded}%
    \renewcommand\color[2][]{}%
  }%
  \providecommand\transparent[1]{%
    \errmessage{(Inkscape) Transparency is used (non-zero) for the text in Inkscape, but the package 'transparent.sty' is not loaded}%
    \renewcommand\transparent[1]{}%
  }%
  \providecommand\rotatebox[2]{#2}%
  \newcommand*\fsize{\dimexpr\f@size pt\relax}%
  \newcommand*\lineheight[1]{\fontsize{\fsize}{#1\fsize}\selectfont}%
  \ifx\svgwidth\undefined%
    \setlength{\unitlength}{132.14285946bp}%
    \ifx\svgscale\undefined%
      \relax%
    \else%
      \setlength{\unitlength}{\unitlength * \real{\svgscale}}%
    \fi%
  \else%
    \setlength{\unitlength}{\svgwidth}%
  \fi%
  \global\let\svgwidth\undefined%
  \global\let\svgscale\undefined%
  \makeatother%
  \begin{picture}(1,0.68840802)%
    \lineheight{1}%
    \setlength\tabcolsep{0pt}%
    \put(0,0){\includegraphics[width=\unitlength,page=1]{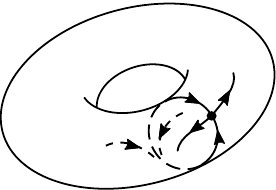}}%
    \put(0.74258745,0.03930017){\makebox(0,0)[lt]{\lineheight{1.25}\smash{\begin{tabular}[t]{l}$N$\end{tabular}}}}%
    \put(0.79797163,0.23190329){\makebox(0,0)[lt]{\lineheight{1.25}\smash{\begin{tabular}[t]{l}$p$\end{tabular}}}}%
    \put(0.21687719,0.63091236){\makebox(0,0)[lt]{\lineheight{1.25}\smash{\begin{tabular}[t]{l}$\mathbb{T}^2$\end{tabular}}}}%
  \end{picture}%
\endgroup%

   \caption{The submanifold $N$ (shown as a circle) contains a point $p$ that is simultaneously a sink along $N$ and a source in the transverse direction. Trajectories starting off $N$ may initially be attracted to $N$ away from $p$, but once near $N$, they flow toward $p$ and then are repelled transversely, preventing permanent collapse onto $N$.}
   \label{fig:torus-example}
   \end{figure}
\end{example}

Further comments on Example \ref{ex:mixedStabilityTorus} are given in Remark \ref{rmk:finalExCOmments2} at the end of Section \ref{sec:Abstract}.

\subsection{Preliminaries} \label{subsec:prelimsAbs2}

\subsubsection{Geometry}

For $\epsilon > 0$, let $N_\epsilon$ denote the tubular neighborhood
\[N_\eps = \{ x \in M : \dist(x, N) < \eps \} \,. \]
Below, let $\exp_x : T_x M \to M$ denote the exponential at $x \in M$. The following standard result yields a useful coordinate system for $N_\epsilon$.
\begin{proposition}\label{prop:tubeCoords2}
    For $\eps > 0$ sufficiently small, $N_\eps$ is diffeomorphic to
    \[T_\eps^\perp N := \{ (y, w) \in T^\perp N : |w| < \eps \}\]
    under the mapping $(y, w) \mapsto \exp_{y}(w)$.
\end{proposition}
Given $x \in N_\eps$, let\footnote{Note that for $\epsilon$ sufficiently small as in Proposition \ref{prop:tubeCoords2}, $y = y(x)$ is the unique element of $N$ such that $\dist(x, N) = \dist(x, y)$. } $y = y(x) \in N, w = w(x) \in T_{y}^\perp N$
be such that $\exp_y(w) = x$.

We use repeatedly the following basic estimates, the proofs of which are omitted. Here, $\dist_N$ refers to the distance along $N$, and $\| f \|_{C^2}$ refers to any chart-defined $C^2$ norm on mappings $M \to M$.
\begin{proposition}\label{prop:geomEstimatesAppendix}
    Assume the setting of Proposition \ref{prop:tubeCoords2} and let $f : M \to M, f(N) \subset N$. There exists a constant $C > 0$, depending only on $M, N$ and $\| f \|_{C^2}$, with the following properties.
    Let $x \in N_\eps$ with $(y, w) = (y(x), w(x))$. Then,
    \begin{itemize}
        \item[(a)] $\dist(f(x), \exp_{f(y)}(D_{y} f (w))) \leq C |w|^2$, and
        \item[(b)] $\dist_N(f(y), y(f(x))) \leq C |w|$.
    \end{itemize}
\end{proposition}

\subsubsection{Semigroups}

Our construction of the function $V$ in the drift condition involves spectral theory for the various semigroups associated to the Markov chains we consider. Below, the \emph{semigroup} $Q$ associated to a Markov chain $(z_n)$ on a metric space $Z$ is the operator taking bounded measurable $\varphi : Z \to \R$ to
\[Q \varphi(z) = \E_z \varphi(z_1) \,. \]
Throughout, we write $P$ for the semigroup associated to $(x_n)$ on $M \setminus N$; $T P^\perp$ for the semigroup of the transverse process $(y_n, w_n)$, and $\widehat P^\perp$ for the semigroup of the projective transverse process $(y_n, v_n)$.

The following, the \emph{tilted} or \emph{Feynman-Kac semigroup}, generalizes the moment generating function of a single random variable to the setting of a Markov chain, and is a crucial ingredient in many approaches to large deviations estimates for Markov chains \cite{ellis2007entropy} and Lyapunov exponents in particular \cite{arnold1984formula}.  We apply this construction to the semigroup $\widehat P^\perp$ as follows.
\begin{definition}
    For $q \in \R$ the \emph{tilted} or \emph{Feynman-Kac Semigroup} $\widehat P^\perp_q$ is defined, for bounded measurable $\psi : \S^\perp N \to \R$, by
    \[\widehat P^\perp_q \psi (y, v)  = \E \left[|D_{y} f v|^{-q} \psi( y_1, v_1)\right]  \]
    Above, ``$f$'' refers to a random diffeomorphism distributed according to the law of the IID sequence $(f_i)$.
\end{definition}

Note that $\widehat P^\perp_0 = \widehat P^\perp$. For all purposes below, we will consider $\widehat P^\perp$ and $\widehat P^\perp_q$ as operators on $C^0 := C^0(\S^\perp N)$, the space of continuous functions with the uniform norm $\| \cdot\|_{C^0}$.

\begin{definition}\label{defn:spectralGap}
    Below, we say that a semigroup $Q$ on $C^0$ admits a \emph{spectral gap} if it admits a simple positive eigenvalue $r$, and if the spectrum $\sigma(Q) \setminus \{r\}$ away from $\{r\}$ is contained in the closed ball of radius $\leq r - \delta$ for some small $\delta > 0$.
\end{definition}
\begin{proposition}\label{prop:tilted2} \
    \begin{itemize}
        \item[(a)] Suppose $(y_n, v_n)$ is uniformly geometrically ergodic. Then, $\widehat P_q^\perp$ admits a spectral gap for all $|q|$ sufficiently small.
        \item[(b)] If in addition $\lambda^\perp > 0$, then the dominant simple eigenvalue $r(q) = e^{-\Lambda(q)}$ of $\widehat P^\perp_q$ satisfies
        \[\Lambda(q) > 0\]
        for all $q > 0$ sufficiently small.
    \end{itemize}
\end{proposition}

\begin{proof}
    That $\widehat P^\perp = \widehat P^\perp_0$ admits a spectral gap is immediate from uniform geometric ergodicity of $(y_n, v_n)$. It is not hard to check that under Assumption \ref{ass:fiC2bounded},
    \[\widehat P^\perp_q \to \widehat P^\perp \quad \text{ in operator norm}\]as $q \to 0$. Standard spectral perturbation theory now implies $\widehat P^\perp_q$ admits a spectral gap for all $|q|$ sufficiently small. This completes the proof of part (a).

    For (b), it follows from standard arguments (see, e.g., \cite{arnold1984formula}) that $q \mapsto \Lambda(q)$ is convex, analytic, and that
    \[\Lambda'(0) = \lambda^\perp \,. \]
    It now follows that if $\lambda_\perp > 0$, then $\Lambda(q) > 0$ for all $q > 0$ sufficiently small.
\end{proof}

For $|q|$ sufficiently small as in Proposition \ref{prop:tilted2}(a), let
\begin{align} \label{eq:defnPsiq2}\psi_q = \lim_{n \to \infty} r(q)^{-n} (\widehat P_q^\perp)^n {\bf 1} \, , \end{align}
where ${\bf 1}$ stands for the constant function of $\S^\perp N$ identically equal to 1. The
right-hand limit exists in $C^0$ by the spectral gap property, and the resulting function $\psi_q \in C^0$ is an eigenfunction of $\widehat P^\perp_q$ associated to the dominant eigenvalue $r(q)$.

\begin{lemma}\label{lem:psiqPropertiesAPP} Let $|q|$ be sufficiently small as in Proposition \ref{prop:tilted2}(a) and let $\psi_q$ be as in \eqref{eq:defnPsiq2}. Then,
    \begin{itemize}
        \item[(a)] $\psi_q > 0$; and
        \item[(b)] for any $\delta > 0$ there exists a constant $C_\delta > 0$ and a decomposition $\psi_q = \psi_q^{C^1} + \psi_q^{C^0}$ with the property that
        \[\| \psi^{C^1}_q\|_{C^1} \leq C_\delta \, , \quad \| \psi^{C^0}_q \|_{C^0} \leq \delta \,. \]
    \end{itemize}
\end{lemma}
\begin{proof}
    For (a), since $\widehat P^\perp_q \to \widehat P^\perp_0 = \widehat P^\perp$ as $q \to 0$, standard spectral theory implies that the spectral projector associated to the dominant eigenvalue $r(q)$ of $\widehat P^\perp_q$ converges to that of $\widehat P^\perp$ in norm as $q \to 0$. The latter projector has range spanned by identically constant functions, and so it follows that $\psi_q \to c {\bf 1}$ as $q \to 0$, where $c > 0$ is a constant\footnote{Indeed, by our choice of normalization \eqref{eq:defnPsiq2} for the dominant eigenfunction of $\widehat P^\perp_q$, it can be shown that in fact $c = 1$. Further details are omitted. }; it follows that $\psi_q$ is strictly positive.
    
    Part (b) follows from the proof of \cite[Corollary 4.3]{BCZG}, which we briefly recall here. By equation \eqref{eq:defnPsiq2}, one has that
    \begin{align}
        \psi_q = r(q)^{-n} (\widehat P^\perp_q)^n {\bf 1} + \mathcal{E}_n \, ,
    \end{align}
    where $\mathcal{E}_n$ is an error term converging to 0 in the $C^0$ norm as $n \to \infty$. With $\delta > 0$ fixed, let $n$ be such that $\| \mathcal{E}_n\|_{C^0} < \delta$. We now set
    \begin{align}
        \psi_q^{C^1} = r(q)^{-n} (\widehat P^\perp_q)^n {\bf 1} \, , \qquad \psi_q^{C^0} = \mathcal{E}_n \, ,
    \end{align}
    noting that $\psi_q^{C^1} \in C^1$ is automatic from the smoothness of the functions $(f_i)$.
\end{proof}

\subsection{Lyapunov function construction} \label{subsec:completeAbsPf2}

Let $\eps > 0$ be sufficiently small as in Proposition \ref{prop:tubeCoords2}, to be adjusted smaller as we go.
For the rest of this section we will freely use the coordinate representation $N_\eps \cong T^\perp_\eps N$, intentionally confusing functions defined on $T^\perp_\eps N$ with those on $N_\eps$.  We will similarly confuse functions $\psi : \S^\perp N \to \R$ with those defined on $T^\perp N$ with the assignment
\[(y,w) \mapsto \psi(y, w / |w|) \,. \]

Let $h : T^\perp N \to \R$ be given by $h(y, w) = |w|$, which as above will be confused with $h : N_\epsilon \to \R$ given by $h(x) = h(y(x), w(x))$.  Below, $q > 0$ is a fixed small parameter as in Proposition \ref{prop:tilted2}(a).

\begin{lemma}\label{lem:localApproxEstAPP}
    For $\epsilon > 0$ sufficiently small there exists $C > 0$ such that the following holds for all $x \in N_\eps$: for any Lipschitz-continuous $\psi : \S^\perp N \to \R$ we have that
    \begin{align}
        | T^\perp P [h^{-q} \psi](y, w) - P[h^{-q} \psi] (x) | \leq C [\psi]_{\rm Lip} \dist(x, N)^{1 - q}
    \end{align}
    where $y = y(x), w = w(x), v = v(x) = w / |w|$.
\end{lemma}

\begin{proof}
    Let $f$ stand here for a typical diffeomorphism $M \to M$ distributed according to our IID law. Unwinding the definitions, our desired estimate amounts to estimating
    \begin{align}
        |\bar w |^{-q} \psi(f y, \bar w / |\bar w| ) - |\dist(f x, N)|^{-q} \psi(y(f x), w(f x) / |w(f x)| )
    \end{align}
    where $\bar w = \Pi_{y}^\perp D_{y} f (w)$. Differencing in the three arguments present, it is straightforward to compute
    \begin{gather}
        \left| |\bar w|^{-q} - |\dist (f x, N)|^{-q} \right| \lesssim |w|^{1 - q} \, , \\
        \left| \frac{\bar w}{|\bar w|} - \frac{w(f x)}{|w(f x)|} \right| \lesssim |w| \, , \\
        \dist_N(f y, y(f x)) \lesssim |w| \, ,
    \end{gather}
    where in the first and second lines we use Proposition \ref{prop:geomEstimatesAppendix}(a), and in the third we use Proposition \ref{prop:geomEstimatesAppendix}(b). Here, $\lesssim$ means less than or equal to up to a multiplicative constant independent of $\psi$ or $x$, but perhaps depending on $q$ and $\| f \|_{C^2}$. Combining these and using the Lipshitz estimate for $\psi$ completes the proof. For further details, see the parallel development in \cite[Lemma 4.4]{BCZG}.
\end{proof}

We are now ready to define the Lyapunov function. Let $\chi_\eps : M \to [0,1]$ be a $C^\infty$ bump function for which $\chi|_{N_\eps} \equiv 1$ and $\chi|_{N_{2 \eps}} \equiv 0$. Let $q > 0$ small be fixed and $\psi_q, \Lambda(q)$ the corresponding eigenpair for $\widehat P^\perp_q$. Define
\[\mathcal{V}(x) = \max\left\{\chi(x) |w(x)|^{-q} \psi_q(y(x), v(x)), 1 \right\}\,.  \]
\begin{proposition}\label{prop:driftCondHolds2}
It holds that
    \[P \mathcal{V}(x) \leq \alpha \mathcal{V}(x)\]
    for $x \in N_\eps$, where $\alpha \in (0,1)$ is a constant.
\end{proposition}

From here, continuity of $P\mathcal{V}(x)$ and compactness of $N_\epsilon^c \subset M$ allow to conclude the drift condition for $\mathcal{V}(x)$ (Definition \ref{defn:driftCond}), which in view of Theorem \ref{thm:driftCondAbs2} completes the proof of Theorem \ref{thm:absMain2}.

\begin{proof}[Proof of Proposition \ref{prop:driftCondHolds2}]
    Apply Lemma \ref{lem:psiqPropertiesAPP}(b) for a value $\delta > 0$ to be taken sufficiently small at the end. By Lemma \ref{lem:localApproxEstAPP}, we have that $x \in N_\eps$, we have that
    \[| P \mathcal{V}(x) - T^\perp P \mathcal{V}(y(x), w(x)) | \lesssim C_\delta d_M(x, N)^{1 -q} + \delta d_M(x, N)^{-q} \, . \]
    
    Since
    \[T^\perp P \mathcal{V}(y(x), w(x)) = r(q) |w(x)|^{-q} \psi_q(y(x), w(x) / |w(x)|) = r(q) \mathcal{V}(x) \, , \]
    it follows that
    \begin{align}
        P \mathcal{V}(x) & \leq r(q) \mathcal{V}(x) + C_\delta d_M(x, N)^{1 - q} + \delta d_M(x, N)^{-q} \\
        & \leq r(q) \mathcal{V}(x) + C_\delta \eps^{1 - q} + \delta d_M(x, N)^{-q}
    \end{align}
    Fix $\delta \ll r(q) \inf |\psi_q|$ and fix $\epsilon \ll C_\delta^{1 / (1 - q)}$. Our desired condition follows on taking $\alpha \in (r(q), 1)$ and using the lower bound on $\psi_q$ from Lemma \ref{lem:psiqPropertiesAPP}(a) to absorb the second and third terms in the above display formula.
\end{proof}

\begin{remark}\label{rmk:finalExCOmments2}
    We conclude Section \ref{sec:Abstract} with some remarks on Example \ref{ex:mixedStabilityTorus}, in view of the proof of Theorem \ref{thm:absMain2}. In this example, the unique stationary measure for the process $(y_n)$ on $N$ is the Dirac mass at the almost-sure fixed-point $p$, where the transversal dynamics are repelling. It is  counterintuitive that this stationary measure somehow governs statistics of repulsion from \emph{any} part of $N$, not just the vicinity of $p$. This underscores the importance of the assumption of uniform geometric ergodicity in Theorem \ref{thm:absMain2}, which connects the statistics of trajectories initiated \emph{away} from $p$ with the random dynamics \emph{at} $p$.
\end{remark}

\section{Preliminaries and outline for application to L96} \label{sec:OutlineL96}

In this section, we outline the proof of Theorem \ref{thm:mainL96} using the series of ideas related to transverse Lyapunov exponents laid out in Section \ref{sec:Abstract}. In addition to checking relevant hypotheses for the Lorenz 96 model, we will have to address two key technical differences between the L96 and the setting of Section \ref{sec:Abstract}: (1) that L96 is a stochastic differential equation, posed in continuous time; and (2) noncompactness of the `ambient' state space $H $, playing the role of $M$ in Section \ref{sec:Abstract}, and the noncompact invariant subspace $H_I$, playing the role of $N$. As we will see, difference (2) is substantive, and will require significant modifications from the construction in Section \ref{sec:Abstract}.
Lastly, we will have to address here the somewhat stronger statement made for L96 in Theorem \ref{thm:mainL96}, that the second stationary measure $\mu$ supported off $H_I$ is unique and geometrically ergodic.

The plan is as follows. After some setup (Section \ref{subsec:setup3}), Section \ref{subsec:3tverseLE} discusses the proofs of existence and positivity of the transverse Lyapunov exponent in the L96 setting, and Section \ref{subsec:3driftCond} treats the proof of the full drift condition. Finally, some comments on uniqueness and geometric ergodicity of the measure $\mu$ supported off $H_I$ are given in Section \ref{sec:SecondMeasure}. Comments on the organization of the rest of the paper are given in Section \ref{subsec:agenda3}.

\subsection{Setup, notation and preliminaries}\label{subsec:setup3}
We regard the Lorenz 96 process $(u_t)$ from \eqref{eq:lorenz96intro} as the solution to the stochastic differential equation
\begin{equation}\label{eq: L96-SDE}
\dee u_t = X_0(u_t)\dt + \sqrt{\ep}\sum_{j \in I} X_j \dee W_t^j,
\end{equation}
on $H = \R^{\Z_N}$, $\Z_N = \Z / N \Z$,  where\footnote{Observe that the components of $u \in H$ are indexed cyclically, identifying component $1$ with $N  + 1$, $2$ with $N + 2$, etc.} $N = 3K$ for some $K \geq 3$ fixed throughout.
Here, $I = \{j\in \Z_N\,:\, j\mod 3 = 0\}$, and the vector fields $X_j$ are given by

\[
X_0(u) = B(u,u) - \ep u, \quad X_j(u) = \sigma_{j} e_{j},\quad \text{for}\quad j\in I \,.
\]
Here, $B(u,u)$ is the bilinear nonlinearity given by
\[
 B(u,v) = \sum_{j\in \Z_N} (u_{j+1} -u_{j-2})v_{j-1} e_j,\quad \text{for}\quad u,v \in H \, .
\]
Here and throughout, $\{e_j\}_{j \in \Z_N}$ is the standard basis of $H$.
We write $(\Omega, \mathcal{F}, \P)$ for the canonical space of the Brownian motions $W_t^j$.  Lastly, we write $\mathcal{L}$ for the generator of the $(u_t)$ process \eqref{eq: L96-SDE} in H\"ormander form, and $\varphi^t = \varphi^t_\omega$ for the stochastic flow on $H$ generated by \eqref{eq: L96-SDE}.

Before continuing, we record the following enhanced \emph{Lyapunov-Foster drift condition} for the $(u_t)$ process. We define the Lyapunov function family
\begin{equation}\label{eq:LyapunovFunctionFamily}
V_\eta(u) := e^{\eta \abs{u}^2}
\end{equation}
for $\eta > 0$.

\begin{lemma}\label{lem:twisty}
For $\eta_\ast = \frac{1}{8 \max_{j\in I} |\sigma_{j}|^2}$, and $\forall \eps > 0$, $\exists \gamma_\ast > 0$ such that $\forall c>0,T>0$ and $\gamma,\eta$ such that $0 < \gamma < \gamma_\ast$ and  $0<\eta e^{\gamma T} < \eta_\ast $, the following estimate holds:
\begin{equation}\label{ineq:LVeta}
\EE_u \left[e^{c \int_0^T \abs{u_s} \dee s} \sup_{0 < t < T} V_{\eta e^{\gamma t}} (u_t)\right] \lesssim_{c,T,\gamma} V_\eta(u).
\end{equation}
Moreover, due to the fluctuation dissipation scaling, we have the uniform in $\eps$ estimate for all $0<\eta<\eta_*$.
\begin{equation}\label{eq:superLF2}
  \sup_{\eps \in (0,1]} \EE_u \sup_{0<t<T} V_{\eta}(u_t) \lesssim_{T} V_\eta(u).
\end{equation}
\end{lemma}

{ A proof of Lemma \ref{lem:twisty} is given in Appendix \ref{app:SuperLyap}.
Since it will be used elsewhere, we record a somewhat weaker consequence of the arguments in Appendix \ref{app:SuperLyap}. That is, $V_\eta$ satisfies a \emph{super-Lyapunov} drift estimate: for any $\gamma > 0$ there exists $C_\gamma > 0$ such that
\begin{align}\label{eq:superLF3}\mathcal{L} V_\eta \leq - \gamma V_\eta + C_\gamma \, .\end{align}
}

\subsubsection*{The invariant subspace and the base process}

Note that for an initial point $y \in H_I = \{u \in H\,:\, u_j=0, \quad j\neq I\}$, the nonlinearity $B(y,y)$ vanishes and $H_I$ is almost surely invariant-- this decouples the coordinates of $y$ from each other , reducing \eqref{eq: L96-SDE} to the Ornstein-Uhlenbeck (OU) process

\begin{equation}\label{eq:OU-process}
\dee y_t = -\ep y_t \dt + \sqrt{\eps}\sum_{j\in I} X_j \dee W_t^{j} \,.
\end{equation}

As mentioned in the introduction, this implies that there exists a unique stationary (hence ergodic) Gaussian measure $\mu^I$ for the $(y_t)$ process on $H_I$, given by \eqref{eq:Gaussian_Density}. We occasionally refer to $(y_t)$ as the \emph{base process}.

We observe that since $H_I$ is invariant, Lemma \ref{lem:twisty} and \eqref{eq:superLF3} hold equally well for the base process $(y_t)$, with the generator $\Lc^y$ of the $(y_t)$ process replacing $\Lc$ in \eqref{eq:superLF3}.

\subsubsection*{Transverse linearization}

We now analyze the dynamics linearized around the invariant manifold $H_I$. Recall that $H_I = \mathrm{span}\{e_j : j \in I\}$ where $I = \{j \in \Z_N : j \equiv 0 \pmod 3\}$. The orthogonal complement is $H_I^\perp = \mathrm{span}\{e_j : j \in T\}$, where \[T = \Z_N \setminus I = \{ j \in \Z_N : j \equiv 1, 2 \text{ mod 3}\}  \]
 is the set of \emph{transverse} modes.
Let $\Pi: H \to H_I$ and $\Pi^\perp: H \to H_I^\perp$ be the corresponding orthogonal projections. Consequently the transverse bundle is given by
\[
T^\perp H_I \simeq H_I \times H_I^\perp.
\]

The full linearized dynamics around a solution $u_t = \varphi^t_\omega(u)$ of \eqref{eq: L96-SDE} are governed by the operator $A^t_{u,\omega} = D\varphi^t_{\omega}(u) \in \mathrm{GL}(H)$, which solves the linear random ODE:
\begin{equation}\label{eq:LinearizedODE}
\frac{\dee}{\dee t} A^t_{u,\omega} = (DB(u_t) - \eps \Id) A^t_{u,\omega}, \quad A^0_{u,\omega} = \Id.
\end{equation}
Here $DB(u)$ is the derivative of the bilinear term $B(u,u)$ evaluated at $u$.

A computationally useful simplification, specific to L96, is that the full linearization $A^t_{y, \omega}$ at a point $y \in H_I$ actually preserves the splitting $H = H_I \oplus H_I^\perp$ with  probability one.

\begin{lemma}[Invariance of Transverse Subspace]\label{lem:TransverseInvariance}
Let $y \in H_I$. Then, the linearized operator $DB(y): H \to H$ maps the subspace $H_I^\perp$ into itself. Consequently, for any trajectory $y_t$ of the Ornstein-Uhlenbeck process \eqref{eq:OU-process} starting from $y \in H_I$, the solution $A^t_{y,\omega}$ to the linearized equation \eqref{eq:LinearizedODE} with $u_t = y_t$ maps $H_I^\perp$ into itself for all $t \ge 0$.
\end{lemma}
\begin{proof}
Let $y \in H_I$ and $w \in H_I^\perp$. The $j$-th component of $DB(y)w$ is \[(DB(y)w)_j = (w_{j+1} - w_{j-2})y_{j-1} + (y_{j+1} - y_{j-2})w_{j-1} \,. \] If $j \in I$ (i.e., $j \equiv 0 \pmod 3$), then $j-1 \equiv 2$, $j+1 \equiv 1$, and $j-2 \equiv 1 \pmod 3$. Since $y \in H_I$, the components $y_{j-1}$, $y_{j+1}$, and $y_{j-2}$ are all zero. Thus, $(DB(y)w)_j = 0$ for $j \in I$, which means $DB(y)w \in H_I^\perp$. As $-\eps \Id$ also preserves $H_I^\perp$, the operator $DB(y_t) - \eps \Id$ preserves $H_I^\perp$ for any $y_t \in H_I$. The invariance of $H_I^\perp$ under the flow $A^t_{y,\omega}$ follows directly from the ODE \eqref{eq:LinearizedODE} and the fact that $A^0_{y,\omega} = \Id$.
\end{proof}

By Lemma \ref{lem:TransverseInvariance}, we can conveniently define the transverse cocycle $A_{\perp, y, \omega}^t, y \in H_I$ acting on the transverse bundle $T^\perp H_I$ by the restriction
\[A_{\perp, y,\omega}^{t} = A^t_{y,\omega}|_{H_I^{\perp}} \in GL(H_I^\perp) \, , \]
with no additional projection required, in contrast with the construction of Section \ref{sec:Abstract}.

In what follows, we will often suppress the dependence on the starting point $y$ and the noise path $\omega$, writing $A^t_{\perp, y, \omega} = A_{\perp}^t$.

\begin{definition}[Transverse Linear and Projective Processes] \label{def:TransverseLinearProcess} \
\begin{enumerate}
\item The \emph{transverse linear process} $\xi_t = (y_t, w_t) \in T^\perp H$, where $w_t \in H_I^\perp$ starting from $w_0 \in H_I^\perp$ is defined by $w_t = A_{\perp, y,\omega}^{t} w_0$, where $y_t$ is the solution to \eqref{eq:OU-process} with $y_0=y$. The process $w_t$ satisfies the linear ODE (driven by $y_t$):
\[
\frac{\dee}{\dt}w_t = (DB(y_t) - \eps \Id) w_t.
\]
We denote the Markov semigroup associated with the process $\xi_t$ by $TP_t^\perp$.
\item The \emph{transverse projective process} $z_t = (y_t,v_t) \in \S^\perp H_I \simeq H_I\times \S_I^\perp$, where $v_t$ on the unit sphere $\mathbb{S}^\perp_I = \{v \in H_I^\perp : |v|=1\}$ is defined by $v_t = w_t / |w_t|$ for any initial condition $v_0 \in \S^\perp_I$. The process $v_t$ satisfies the ODE:
\[
\frac{\dee}{\dt}v_t = DB(y_t)v_t - \eps v_t - v_t\langle v_t, DB(y_t)v_t\rangle.
\]
We denote the Markov semigroup associated with the process $z_t$ by $\widehat{P}_t^\perp$.
\end{enumerate}
\end{definition}

\subsection{Transverse Lyapunov exponents} \label{subsec:3tverseLE}

As in Section \ref{sec:Abstract}, instability of the invariant space $H_I$ is measured by the \emph{transverse Lyapunov exponent} $\lambda_\eps^\perp$, which from now on will be written explicitly $\epsilon$-dependence to emphasize the role of the small parameter $\eps$.

\subsubsection*{Existence of the transverse Lyapunov exponent $\lambda^\perp_\eps$}
The following is the analogue of Proposition \ref{prop:LEandMET2} in the setting of L96.

 \begin{proposition}[Existence of Transverse Lyapunov Exponent]\label{prop:LE_L96} \
\begin{itemize}
    \item[(a)] The limit
    \[ \lambda_\eps^\perp = \lim_{t \to \infty} \frac{1}{t} \log \| A_{\perp, y, \omega}^t \| \]
    exists and is constant $\mu^I \times \P$-almost surely.
    \item[(b)]
    For any $(y_0, w_0) \in T^\perp H_I$, $w_0 \neq 0$,
    the solution $w_t = A_{\perp, y_0, \omega}^t w_0$ satisfies
    \[ \lambda_\eps^\perp = \lim_{t \to \infty} \frac{1}{t} \log |w_t| \quad \P\text{-a.s.} \]
\end{itemize}
\end{proposition}
\begin{proof}[Proof Sketch]
For (a), the existence of the Lyapunov exponent $\lambda_\eps^\perp$ and its representation follow from the multiplicative ergodic theorem in the continuous-time setting; see, e.g.,
\cite{arnold1998random}. A key requirement for applying the multiplicative ergodic theorem in this continuous-time setting is the log-integrability of the cocycle norm and its inverse
with respect to the stationary measure $\mu^I$:  \[\E\int_{H_I} \sup_{0<t<1} \log^+\|A_{\perp,y,\omega}^t\| \, \dee\mu^I(y) < \infty \,, \qquad
\E\int_{H_I} \sup_{0<t<1} \log^+\|(A_{\perp,y,\omega}^t)^{-1}\| \, \dee\mu^I(y) < \infty \,.
\]   This is verified in Appendix \ref{app:SuperLyap} as Corollary \ref{cor:twistyFC}, relying on the properties of the super-Lyapunov function $V_\eta$ as in Lemma \ref{lem:twisty}.
Item (b) will follow from geometric ergodicity of the projective process $(y_t, v_t)$, stated below as Lemma \ref{lem:GeoErgoPP}.
\end{proof}

As in Section \ref{sec:Abstract}, we used geometric ergodicity of the transverse projective process $z_t = (y_t, v_t)$ on $\S^\perp H_I$. Here however we confront a difference due to noncompactness of $H_I$, which prevents us from invoking \emph{uniform} geometric ergodicity as in Definition \ref{defn:unifGeoErgodic2}. The following analogue, \emph{$V$-geometric ergodicity} will be necessary.

For this we will use the weighted spaces $C_{V_\eta}$ of functions on $\mathbb{S}^\perp_I$, equipped with the norm:
\begin{align}\label{eq:C_V}
  \|\varphi\|_{C_{V_\eta}} := \sup_{(y,v) \in H_I \times \mathbb{S}_I^\perp} \frac{|\varphi(y,v)|}{V_\eta(y)}.
  \end{align}
  Here $\eta > 0$, to be taken sufficiently small in what follows.

\begin{lemma} \label{lem:GeoErgoPP}
For all $\eps > 0$, it holds that (i) the process $z_t = (y_t, v_t)$ admits a unique stationary measure $\nu^\eps$ on $\S^\perp H_I$; and (ii) for all $\eta \in (0, \eta_\ast)$, $\eta_\ast$ as in Lemma \ref{lem:twisty}, it holds that
the transverse projective process is geometrically ergodic in $C_{V_\eta}$. That is, there exist $C, r > 0$ such that
 that for all bounded measurable $\varphi:\S^\perp H_I \to \R$, there holds
\begin{align*}
\left|\widehat{P}_t^\perp \varphi(z_0) - \int \varphi \dee \nu^\eps \right| \leq C V_\eta(z_0) e^{- r t} \|\varphi\|_{C_{V_\eta}} \, .
\end{align*}
\end{lemma}
The proof of Lemma \ref{lem:GeoErgoPP} is given in Section \ref{sec:GeoErgodicity_consolidated}. It relies on verifying Hörmander's condition for the generator  of the projective Markov processs (using computer assistance, see Sections \ref{sec:hypo_linearization_new}, \ref{sec:hypo_projective} and Appendix \ref{app:CAP}) to establish the strong Feller property and irreducibility, which, combined with the drift condition \eqref{ineq:LVeta} via Harris' theorem (see e.g. \cite{meyn2012markov}), yields geometric ergodicity in the weighted space $C_{V_\eta}$. A corollary of these arguments is that $\nu_\eps$ is absolutely continuous with respect to Lebesgue measure on $\S^\perp H_I$, with a strictly positive $C^\infty$ density $f_\eps$.

A key consequence of geometric ergodicity in $C_{V_\eta}$ is that the semigroup ${\widehat{P}}_t^\perp$ has a spectral gap on $C_{V_\eta}$, meaning $1$ is a simple, isolated eigenvalue and the rest of the spectrum is contained in a disk of radius $< 1$.

\subsubsection*{Positivity of $\lambda_\eps^\perp$}

Theorem \ref{prop:LE_L96} above addressed \emph{existence} of the transverse Lyapunov exponent $\lambda_\eps^\perp$, the first primary requirement in the construction of the drift condition off the invariant subspace $H_I$. We now turn to the second requirement, \emph{positivity} of $\lambda_\eps^\perp$.

\begin{lemma}\label{lem:PosLyapL96}
In the setting of Section \ref{subsec:setup3}, assume $\sigma_j \neq 0$ for all $j \in I$. Then,
\begin{align*}
  \lim_{\eps \to 0} \frac{\lambda^\perp_\eps}{\eps} = +\infty.
  \end{align*}
In particular, there exists $\eps_0$ such that for all $0 < \eps < \eps_0$,
\[\lambda^\perp_\eps > 0 \,.\]\end{lemma}

The proof of Lemma \ref{lem:PosLyapL96}, largely following \cite{BBPS20} with some modifications, will be presented below in a series of additional preliminary lemmas, with proofs deferred to elsewhere or omitted where indicated.

We start with the following identity for $\lambda^\perp_\eps$ in terms of a certain \emph{Fisher Information} of the density $f_\eps$ for the stationary measure $\nu_\eps$ of the $z_t = (y_t, v_t)$ process.

\begin{lemma}\label{lem:FIIdentity3}
  For any $\eps > 0$, it holds that
\[    \eps FI(f_\eps) = |T| \lambda^\perp_\eps + \ep N \, ,
\]
where
\[FI(f_\eps) := \sum_{k \in I} \frac{1}{2}\int_{H_I \times \mathbb{S}^\perp_I} \frac{|\sigma_k\partial_{u_k} f_\eps|^2}{f_\eps} \dee y \dee v \,. \]
\end{lemma}

Lemma \ref{lem:FIIdentity3} is an adaptation of \cite[Proposition 3.2]{BBPS20}. A proof sketch is given in Section \ref{sec:FI}.

The next key idea from \cite{BBPS20} is the following hypoelliptic regularity estimate which relates $L^1$-type Sobolev regularity of $f_\eps$ to the degenerate Fisher information $FI$. Below, for $R > 0$ we define a smooth cut-off $\chi_R(u) = \chi(R^{-1}\abs{u})$, where we have fixed $\chi \in C^\infty_c([0,2))$ with $\chi(x) = 1$ for $x \in [0,1]$ and $x = 0$ for $x > 3/2$.

\begin{lemma}[Theorem B in \cite{BBPS20}] \label{lem:FIT}
There exists an $s > 0$ such that $\forall R \geq 1$, $\exists C_R$ such that the following holds uniformly in $\eps$:
\begin{align*}
\norm{\chi_R f_\eps}_{W^{s,1}}^2 \leq C_R(1 + FI(f_\eps)).
\end{align*}
\end{lemma}

Here $W^{s,1}$ denotes the $L^1$-based Sobolev space on $\S^\perp H_I \cong H_I \times \S^\perp_I$; for a precise definition of this Sobolev space on a manifold, see [Appendix A.1 \cite{BBPS20}] or the general reference \cite{Triebel}.
The proof of Lemma \ref{lem:FIT} is a straightforward adaptation of \cite[Section 4]{BBPS20}, and relies crucially on an $\epsilon$-uniform quantitative hypoellipticity property for the $(z_t)$ process, to be checked in Section \ref{sec:quant_hypo}.

\subsubsection*{Proof sketch of Lemma \ref{lem:PosLyapL96}: Contradiction argument}

Our contradiction hypothesis will be that
\[\liminf_{\epsilon \to 0}  {\lambda^\perp_\epsilon \over \epsilon} < \infty \, . \]
This immediately implies, in view of
Lemma \ref{lem:FIIdentity3}, that there exists a subsequence $\eps_j \to 0$ such that
\[FI(f_{\eps_j}) \leq C \]
for a constant $C > 0$ independent of $j$.
Lemma \ref{lem:FIT}, in turn, implies
\begin{align}\label{eq:posRegControl3}\| \chi_R f_{\eps_j} \|_{W^{s, 1}} \leq C'\end{align}
where $C' > 0$ is also independent of $j$.

This implies $\epsilon$-uniform $W^{s,1}$ bounds on compact sets via the hypoelliptic estimate. We now apply \cite[Lemma A.3]{BBPS20}, which in our context is a criterion for strong $L^1(\S^\perp H_I)$-precompactness of a collection of functions satisfying (i) uniform $W^{s, 1}$ control on bounded subsets of $\S^\perp H_I$; and (ii) a tightness condition. Item (i) is handled from \eqref{eq:posRegControl3}, while (ii) follows from the fact that for all $\eps > 0$, $\nu_\eps$ projects to the Gaussian measure $\mu^I$ with ($\epsilon$-independent) density given in \eqref{eq:Gaussian_Density}.

Refining to an $L^1_{\rm loc}$-convergent subsequence $f_{\eps_{j_k}}$ implies the existence of an invariant density $f_0$ for the deterministic $\eps = 0$ transversal projective ODE
\[\begin{cases}
  \dot y = 0 \\
  \dot v = DB(y) v - v \langle v, DB(y) v \rangle
\end{cases}
\]
on $\S^\perp H_I$. It is immediate that any such density $f_0$ projects to $\mu^I$ on the $H_I$ factor.

To complete the contradiction argument, we will demonstrate in Section \ref{sec:FI}
(Lemma \ref{lem:singular_limit}) that no such invariant density exists. This is a direct consequence of Lemma \ref{lem:Eigenvalue-unstable}, which checks that there is a positive Lebesgue-measure set of $y \in H_I$ for which the corresponding linearized flow
\[A_\perp^t = e^{t DB(y)|_{H^\perp_I}}\]
has unstable eigenvalues. Lemma \ref{lem:singular_limit} follows on noting that  at any such $y \in H_I$, invariant mass  for $(z_t)$ at $\eps = 0$ must collapse to the zero-volume subset of $\{ y \} \times \S^\perp_I$ corresponding to that unstable eigenspace. See Section \ref{subsec:nonexistDensity5} for further details.

{
\begin{remark}
  The foregoing contradiction argument parallels that given in \cite[Section 6]{BBPS20}. However, our job is far easier here due to the fact that at $\eps = 0$ the $(y_t)$ dynamics are completely suppressed, and every point of $H_I$ is fixed.
\end{remark}

\subsection{Drift Condition and Existence of a Second Stationary Measure} \label{subsec:3driftCond}

Having established the positivity of the transverse Lyapunov exponent $\lambda^\perp_\eps > 0$ for small $\eps$ (Lemma \ref{lem:PosLyapL96}), we turn to the construction of a Lyapunov function $\mathcal{V}$ repelling from $H_I$ as in Section \ref{sec:Abstract}, leading to the existence of a stationary measure $\mu$ distinct from the Gaussian measure $\mu^I$ supported on $H_I$. As we will see, modifications to the construction of Section \ref{sec:Abstract} are necessary to cope with the noncompactness of $H, H_I$. For the rest of Section \ref{sec:OutlineL96}, $\epsilon > 0$ is fixed to be sufficiently small so that $\lambda^\perp_\eps > 0$ as in Lemma \ref{lem:PosLyapL96}.

\subsubsection{Lyapunov Function Construction}\label{subsubsec:3LFconstruct}

Our Lyapunov function will be of the form
\begin{align}\label{eq:LyapFunc-candidate}
  \mathcal{V}(u) = \Hc_p(u) + V_\eta(u) \,.
\end{align}
Here, $V_\eta(u) = e^{\eta |u|^2}$, where $\eta > 0$ will be taken sufficiently small; as in Lemma \ref{lem:twisty}, this term reflects global confinement and controls drift of $(u_t)$ to infinity.

The other term, $\mathcal{H}_p$, is analogous to the construction of Section \ref{subsec:completeAbsPf2}, blowing up near $H_I$ and leveraging instability near $H_I$ to guarantee repulsion. It will be of the form
\[
\mathcal{H}_p(u) = \frac{1}{\abs{\Pi^\perp u}^p} \psi_p \left(\Pi u, \frac{\Pi^\perp u}{\abs{\Pi^\perp u}} \right) \,,
\]
where $\psi_p : H_I \times \S^\perp_I \to \R$ will be the dominant eigenfunction of the Feynman-Kac semigroup
\begin{align*}
  \widehat{P}_t^{\perp,p} \varphi(y,v) &= \EE_{z} \left[ \frac{1}{\abs{A_{\perp}^t v}^p} \varphi\left(y_t, v_t\right) \right] \\
  & = \EE_{z} \left[ \exp \left( -p \int_0^t \brak{v_s, DB(y_s) v_s - \eps v_s} \dee s \right) \varphi \left( y_t, v_t\right) \right]
  \end{align*}
  satisfying
  \[\widehat{P}_t^{\perp, p} \psi_p = e^{-t \Lambda(p)} \psi_p \,. \]

As usual, we will often intentionally confuse $\mathcal{H}_p : H \setminus H_I \to \R$ with the corresponding function $H_I \times (H_I^\perp \setminus \{ 0 \}) \to \R$ via the coordinate assignment $u \mapsto (y, w)$ with $u = y + w, y \in H_I, w \in H_I^\perp \setminus \{ 0 \}$. Under this parametrization, $\Hc_p$ takes the form
\[\Hc_p(y, w) = \frac{1}{|w|^{p}} \psi_p \left(y, \frac{w}{|w|}\right)\]

In parallel with Section \ref{subsec:completeAbsPf2}, the following will be needed in the proof of the drift condition. Below, a sufficiently small value of $p > 0$ is fixed.

\begin{itemize}
  \item[(i)] {\bf A spectral gap for $\widehat{P}_t^{\perp,p}$}: like before, we will realize $\widehat{P}^{\perp,p}_t$ as a norm-continuous perturbation of the Markov semigroup $\widehat{P}^\perp_t$ for the transverse projective process $z_t = (y_t, v_t)$. The main difference here from Section \ref{sec:Abstract} is that the perturbation will take place in the weighted space $C_{V_\eta}$, to account for noncompactness of the state space $\S^\perp H_I$.
  \item[(ii)] {\bf Dominant eigenvalue of $\widehat{P}^{\perp, p}_t$}: It will be shown that for $t > 0$, the dominant eigenvalue of $\widehat{P}^{\perp, p}_t$ is $e^{- t \Lambda(p)}$, where the moment Lyapunov exponent $\Lambda(p) = \lim_{t \to \infty} \frac1t \log \E_{(y_0, w_0)} |w_t|^{-p}$ will satisfy the asymptotic $\Lambda(p) = p \lambda^\perp_\eps + o(p)$ for $|p| \ll 1$.
  \item[(iii)] {\bf Properties of the dominant eigenfunction $\psi_p$}: for $p > 0$ small, we let $\psi_p = \lim_{t \to \infty} \widehat{P}^{\perp, p}_t {\bf 1}$, which by the $C_{V_\eta}$-spectral gap condition for $\widehat{P}^{\perp, p}_t$ is a dominant, nonnegative eigenfunction with eigenvalue $e^{t \Lambda(p)}$.
  In parallel with Section \ref{subsec:completeAbsPf2}, we will need (a) positivity, i.e., $\psi_p > 0$ pointwise, as well as (b) some control on the $C^1$ regularity of $\psi_p$, namely that
  \[\| \psi_p\|_{C^1_{V_\eta}} := \| \psi_p\|_{C_{V_\eta}} + \| D \psi_p\|_{C_{V_\eta}} < \infty \,. \]
  The argument from Lemma \ref{lem:psiqPropertiesAPP} no longer suffices to obtain this kind of quantitative control, and a separated argument must be used -- ours takes advantage of the continuous-time setting, adapting quantitative hypoelliptic estimates from \cite{bedrossian2021quantitative,bedrossian2022stationary, BBPS20,Hormander67}.
\end{itemize}
Items (i) -- (iii) are treated in Section \ref{sec:Drift}. }

\subsubsection{The Drift Condition and Its Consequences}

With the Lyapunov function $\mathcal{V}$ and the properties of $\psi_p$ established, we state the main result of this section:

\begin{lemma}[Drift Condition for L96]\label{lem:drift_L96_merged}
Let $p>0$ and $\eta>0$ be sufficiently small such that $\Lambda(p)>0$ and $\psi_p \in C^1_{V_\eta}$. There exist constants $\lambda>0$ and $C_{0}\ge0$ such that for all $t\ge0$ and $u \in H \setminus H_I$,
\[
\E_{u}\mathcal{V}(u_{t})\le e^{-\lambda t}\mathcal{V}(u)+C_{0}.
\]
\end{lemma}

As discussed in Section \ref{sec:Abstract} (see Theorem \ref{thm:driftCondAbs2} and the surrounding discussion), this drift condition guarantees the existence of a stationary measure distinct from $\mu^I$:

\begin{corollary}\label{cor:SecondMeasure_L96_merged}
Under the conditions of Lemma \ref{lem:drift_L96_merged}, the process $(u_{t})$ admits at least one stationary measure $\mu$ on $H$ satisfying $\int \mathcal{V} \, \mathrm{d}\mu < \infty$. This measure $\mu$ is distinct from $\mu^I$.
\end{corollary}
\begin{proof}[Proof Sketch]
Existence follows from the Krylov-Bogoliubov theorem applied in the weighted space $L^1(\mathcal{V})$, where tightness is provided by the drift condition. Distinctness follows because $\mathcal{V}(u) \to \infty$ as $u \to H_I$ (due to the $\Hc_p$ term), implying $\int \mathcal{V} \, \mathrm{d}\mu^I = \infty$, whereas $\int \mathcal{V} \, \mathrm{d}\mu < \infty$.
\end{proof}

In what follows, we present the full proof of Lemma \ref{lem:drift_L96_merged} assuming properties (i) -- (iii) listed in Section \ref{subsubsec:3LFconstruct}.
\begin{proof}[Proof of Lemma \ref{lem:drift_L96_merged}]

  The goal is to show that the Lyapunov function $\mathcal{V} = \mathcal{H}_p + V_\eta$ satisfies a drift condition $\mathcal{L}\mathcal{V} \leq -\lambda_0 \mathcal{V} + C'$ for some $\lambda_0 > 0$, where $\mathcal{L}$ is the generator of the L96 process \eqref{eq: L96-SDE}. The generator is given by
  \[
  \mathcal{L}\varphi = \underbrace{\frac{\eps}{2}\sum_{j \in I} \sigma_{j}^2 \partial_{u_j u_j}^2 \varphi - \eps u \cdot \nabla \varphi}_{\mathcal{L}_{OU}\varphi} - B(u,u) \cdot \nabla \varphi.
  \]
  Here $\mathcal{L}_{OU}$ is the generator of an Ornstein-Uhlenbeck (OU) process on the full space $H$. We analyze the action of $\mathcal{L}$ on $\Hc_p$ and $V_\eta$ separately.
  
  We have $\mathcal{L} V_\eta = \mathcal{L}_{OU} V_\eta - B(u,u) \cdot \nabla V_\eta$. The gradient is $\nabla V_\eta(u) = 2\eta u e^{\eta|u|^2} = 2\eta u V_\eta(u)$. Due to the conservation property $\langle B(u,u), u \rangle = 0$, the term $B(u,u) \cdot \nabla V_\eta = 2\eta V_\eta(u) \langle B(u,u), u \rangle$ vanishes identically. Thus, $\mathcal{L} V_\eta = \mathcal{L}_{OU} V_\eta$. By a straightforward analogue of \eqref{eq:superLF3} for $V_\eta$ viewed in $H_I$, the super Lyapunov property for $V_\eta : H \to \R$ holds:
   for every $\gamma>0$ there exists a constant $C_\gamma$ such that
  \[
  \mathcal{L} V_\eta \leq -\gamma V_\eta + C_\gamma.
  \]
Let $\gamma > 0$ be fixed, its value to be specified shortly.

Next, we analyze the action on $\Hc_p$. From the eigenfunction relation $\widehat{P}_t^{\perp, p} \psi_p = e^{- t\Lambda(p)}$, it follows that, in $(y, w)$-coordinates for $u$, the function $\Hc_p = \Hc_p(y, w)$ is an eigenfunction of the Markov semigroup $T P_t$ for the transverse linearized process $\xi_t = (y_t, w_t)$. By standard semigroup arugments (see Section \ref{sec:psip} for details), it follows that
  \[\mathcal{L}^\xi \Hc_p = - \Lambda(p) \Hc_p \, , \]
  where $\mathcal{L}^\xi$ is the generator for $(\xi_t)$, given by
 \[  \mathcal{L}^\xi \varphi(y, w) := \frac{\eps}{2}\sum_{j \in I} \sigma_{j}^2 \partial_{y_j y_j}^2 \varphi - \eps y \cdot \nabla_y \varphi - (DB(y) w - \eps w) \cdot \nabla_w \varphi.
\]
  
  Using $B(u,u) = B(w,w) + DB(y)w$ and that $DB(y)$ has range $H_I^\perp$ (Lemma \ref{lem:TransverseInvariance}), we can rewrite the generator $\mathcal{L}$ in $H_I \times H_I^\perp$-coordinates $u \mapsto (y, w)$ as
  \begin{equation}
  \mathcal{L} = \frac{\eps}{2}\sum_{j \in I} \sigma_{j}^2 \partial_{y_j y_j}^2 - \eps y \cdot \nabla_y + (DB(y)w - \eps w) \cdot \nabla_w + B(w,w)\cdot\nabla
  \end{equation}
  Consequently,
  \[
  \begin{aligned}
  \mathcal{L}\mathcal{H}_p(u) &= \mathcal{L}^\xi \mathcal{H}_p(y,w)  + B(w,w)\cdot\nabla \mathcal{H}_p(u)\\
  &= - \Lambda(p)\mathcal{H}_p(u) + B(w,w)\cdot\nabla \mathcal{H}_p(u).
  \end{aligned}
  \]

 We need to bound the error term $B(w,w)\cdot\nabla \mathcal{H}_p(u)$. The crucial ingredient is the $C^1_{V_\eta}$ regularity of $\psi_p$ (Lemma \ref{lem:psipC1}), which implies
  \begin{align*}
  |\psi_p(y,v)| &\lesssim V_{\eta^\prime}(y) \\
  |\nabla \psi_p(y,v)| &\lesssim V_{\eta^\prime}(y)
  \end{align*}
  for some $\eta^\prime < \eta$, where $\nabla$ includes derivatives w.r.t both $y$ and $v$.
  Calculating $\nabla \mathcal{H}_p$, we find terms involving $\psi_p$ and $\nabla \psi_p$, multiplied by powers of $|w|$. Schematically,
  \[ |\nabla \mathcal{H}_p(u)| \lesssim |w|^{-p-1} |\psi_p| + |w|^{-p} \|\nabla \psi_p\| \lesssim |w|^{-p-1} e^{\eta^\prime|y|^2}. \]
  The error terms involve $B(w, w)$, which scales as $|w|^2$. Therefore,
  \[ |B(w,w)\cdot \nabla \mathcal{H}_p| \lesssim |w|^2 |\nabla \mathcal{H}_p| \lesssim |w|^{1-p} e^{\eta^\prime|y|^2}\lesssim V_\eta(u). \]
  Thus, there exists a constant $c>0$ such that
  \[
  \mathcal{L} \mathcal{H}_p \leq -\Lambda(p) \mathcal{H}_p + c V_\eta(u).
  \]

  Combining the estimates for $\mathcal{L}\mathcal{H}_p$ and $\mathcal{L}V_\eta$ yields:
  \begin{align*}
  \mathcal{L}\mathcal{V} = \mathcal{L}\mathcal{H}_p + \mathcal{L}V_\eta &\leq (-\Lambda(p) \mathcal{H}_p + c V_\eta) + (-\gamma V_\eta + C_\gamma) \\
  &= -\Lambda(p) \mathcal{H}_p - (\gamma - c) V_\eta + C_\gamma.
  \end{align*}
  Choose $\gamma = c + \Lambda(p)$ gives
  \[
  \mathcal{L}\mathcal{V} \leq -\Lambda(p)\mathcal{V} + C_\gamma.
  \]
  The above inequality is the required differential form of the drift condition. A standard application of Dynkin's formula yields the desired time integrated form of the drift condition.
  \end{proof}

\subsection{Uniqueness and geometric ergodicity of $\mu$}\label{sec:SecondMeasure}
 The drift condition established in Lemma \ref{lem:drift_L96_merged} guarantees the existence of at least one stationary measure $\mu$ distinct from $\mu^I$ and satisfying $\int \mathcal{V} \, \mathrm{d}\mu < \infty$, as stated in Corollary \ref{cor:SecondMeasure_L96_merged}. To complete the proof of Theorem \ref{thm:mainL96}, it remains to check that the measure $\mu$ is unique and geometrically ergodic. These properties will be deduced using Harris' Theorem, which will require stronger properties of the dynamics \eqref{eq: L96-SDE} on the state space $H \setminus H_I$. Further details are deferred to Section \ref{sec:GeoErgodicity_consolidated}.

\subsection{Agenda for the rest of the paper}\label{subsec:agenda3}

We close the outline of Section \ref{sec:OutlineL96} with a brief summary of the remainder of the paper, which will fill in the technical steps needed in the proof of Theorem \ref{thm:mainL96} presented thus far.

\begin{itemize}
  \item[(1)] {\bf Existence of the transverse Lyapunov exponent $\lambda_\eps^\perp$} (Proposition \ref{prop:LE_L96}).
  
  Integrability conditions on the transverse linearizations $A_\perp^t$ needed in the proof sketch of Proposition \ref{prop:LE_L96} will be carried out in Appendix \ref{app:SuperLyap} (Corollary \ref{cor:twistyFC})
  \item[(2)] {\bf Geometric ergodicity of the transverse projective process $z_t = (y_t, v_t)$ on $\S^\perp H_I$  }  (Lemma \ref{lem:GeoErgoPP}).
  
  This is an application of Harris' Theorem, which will require that we prove some bracket-spanning and irreducibility properties of the $(z_t)$ process. These are carried out in Section \ref{sec:Hypoellipticity}, with
  supplemental Appendix \ref{app:control_theory} -- relating hypoellipticity conditions with control theory and irreducibility -- and Appendix \ref{app:CAP} -- detailing a computer-assisted step in the bracket-spanning computation.
  \item[(3)] {\bf Positivity of the transverse Lyapunov exponent $\lambda^\perp_\eps$} (Lemma \ref{lem:PosLyapL96}).
  
  In Section \ref{sec:FI} we will check the remaining ingredients in the proof of Lemma \ref{lem:PosLyapL96} outlined in Section \ref{subsec:3tverseLE}, namely, the Fisher information identity (Lemma \ref{lem:FIIdentity3}) and the nonexistence of an invariant density for $z_t = (y_t, v_t)$ at $\eps = 0$.

  \item[(4)] {\bf Properties (i) -- (iii) from Section \ref{subsubsec:3LFconstruct} regarding dominant eigendata of the twisted semigroup $\widehat{P}^{\perp, p}_t$.}
  
  These are checked in Section \ref{sec:Drift}.

  \item[(5)] {\bf Uniqueness and geometric ergodicity of full process $(u_t)$ on $H \setminus H_I$.}
  
  It is proved in Section \ref{sec:SecondMeasure} that if a stationary measure for $(u_t)$ on $H \setminus H_I$ exists, then it is unique and geometrically ergodic, completing the proof of Theorem \ref{thm:mainL96}.
\end{itemize}

\section{Hypoellipticity, Irreducibility and Geometric Ergodicity}\label{sec:Hypoellipticity}

This section consolidates the core technical arguments regarding the hypoellipticity, irreducibility and ergodicity properties of both the full Lorenz-96 process $(u_t)$ on $H \setminus H_I$ and the associated transverse projective process $(y_t, v_t)$ on $\mathbb{S}^\perp H_I$. These properties are key for establishing the existence, uniqueness, and regularity of the stationary measures.

\subsection{Hörmander Condition and Control Theory Framework}\label{sec:hormander_framework}

Let $M$ be a smooth manifold and let $\mathfrak{X}(M)$ denote the space of smooth vector fields on $M$. The Lie bracket of $X, Y \in \mathfrak{X}(M)$ is the vector field $[X, Y] \in \mathfrak{X}(M)$ defined such that for any smooth function $f: M \to \R$, $[X, Y]f = X(Yf) - Y(Xf)$. This operation endows $\mathfrak{X}(M)$ with the structure of a Lie algebra.

Given a set of vector fields $\mathcal{Y} = \{Y_1, \dots, Y_r\} \subset \mathfrak{X}(M)$, the Lie algebra generated by $\mathcal{Y}$, denoted $\mathrm{Lie}(\mathcal{Y})$, is the smallest Lie subalgebra of $\mathfrak{X}(M)$ containing $\mathcal{Y}$. Its evaluation at a point $x \in M$, denoted $\mathrm{Lie}(\mathcal{Y})(x)$, is the subspace of the tangent space $T_x M$ spanned by the vectors $\{Z(x) : Z \in \mathrm{Lie}(\mathcal{Y})\}$.

For $X, Y \in \mathfrak{X}(M)$, the adjoint operator is $\mathrm{ad}_X Y = [X, Y]$. Iterated brackets are denoted $\mathrm{ad}_X^0 Y = Y$ and $\mathrm{ad}_X^k Y = [X, \mathrm{ad}_X^{k-1} Y]$ for $k \ge 1$.

Consider a stochastic differential equation (SDE) on $M$:
\begin{equation}\label{eq:generic_SDE_hypo}
\dee x_t = Y_0(x_t) \dt + \sum_{k=1}^r Y_k(x_t) \circ \dee W_t^k,
\end{equation}
where $Y_0, \dots, Y_r \in \mathfrak{X}(M)$ and $\circ$ denotes the Stratonovich integral.

\begin{definition}[Parabolic Hörmander Condition]\label{def:parabolic_hormander}
The SDE \eqref{eq:generic_SDE_hypo} satisfies the \emph{parabolic Hörmander condition} at $x \in M$ if the Lie algebra generated by the diffusion vector fields $Y_1, \dots, Y_r$ and all their iterated Lie brackets with the drift $Y_0$ spans the tangent space $T_x M$. Formally, let $\mathcal{S} = \{ \mathrm{ad}_{Y_0}^j Y_k : 1 \le k \le r, j \ge 0 \}$. The condition is:
\[
\mathrm{Lie}(\mathcal{S})(x) = T_x M.
\]
If this condition holds for all $x$ in an open set $U \subset M$, Hörmander's theorem \cite{Hormander67} guarantees that the generator $\mathcal{L} = Y_0 + \frac{1}{2}\sum_{k=1}^r Y_k^2$ is hypoelliptic on $U$. A commonly used sufficient condition, often called the \emph{restricted Hörmander condition}, involves only the first-order brackets. Let $\mathcal{S}_1 = \{Y_k, [Y_0, Y_k] : 1 \le k \le r \}$. The restricted condition is:
\[
\mathrm{Lie}(\mathcal{S}_1)(x) = T_x M.
\]
\end{definition}

\begin{remark}
The parabolic H\"ormander condition is sometimes stated using the Lie algebra ideal $\mathcal{I}$ generated by $\{X_1,\ldots, X_r\}$ in $\mathrm{Lie}(X_0,\ldots,X_r)$, requiring $\mathcal{I}(x) = T_x M$. Here, $\mathcal{I}$ denotes the smallest ideal containing $\{X_1,\ldots, X_r\}$, which consists of all finite linear combinations of Lie brackets $[Y, X_k]$ where $Y \in \mathrm{Lie}(X_0,\ldots,X_r)$ and $k \in \{1,\ldots,r\}$. It can be seen by a straightforward induction proof that $\mathcal{I} = \mathrm{Lie}(\mathcal{S})$, where $\mathcal{S} = \{ \mathrm{ad}_{X_0}^j X_k : 1 \le k \le r, j \ge 0 \}$. Hence, this formulation is equivalent to the condition $\mathrm{Lie}(\mathcal{S})(x) = T_x M$ stated above.
\end{remark}

\medskip

The hypoellipticity and irreducibility of stochastic processes are often established via Hörmander's condition. Appendix \ref{app:control_theory} details the connection between this condition, control theory, and topological irreducibility. Specifically, Proposition \ref{prop:irreducibility-general} shows that if the fields $Y_0, \dots, Y_r$ are analytic and satisfy the cancellation property $[Y_k, [Y_k, Y_0]] = 0$ for all $k$, then the restricted parabolic H\"ormander condition implies topological irreducibility via the Stroock-Varadhan support theorem and controllability arguments.

\subsection{Hypoellipticity of the Transverse Linearization}\label{sec:hypo_linearization_new}

We first analyze the hypoellipticity of the underlying linear process $(y_t, A^t_\perp)$, where $y_t$ is the OU process \eqref{eq:OU-process} on $H_I$ and $A^t_\perp \in \mathrm{GL}(H_I^\perp)$ solves the linear random ODE (driven by $y_t$):
\begin{equation}\label{eq:LinearizedODE_perp_recap}
\frac{\dee}{\dt} A^t_\perp = (DB(y_t) - \eps \Id) A^t_\perp, \quad A^0_\perp = \Id.
\end{equation}
The joint process $(y_t, A^t_\perp)$ evolves on the state space $H_I \times \mathrm{GL}(H_I^\perp)$. However, from the standpoint of hypoellipticity it is natural to consider the {\em volume normalized linearization}
\[
\bar{A}_t := \frac{A^t_\perp}{\det(A^t_\perp)^{1/|T|}} \in \mathrm{SL}(H_I^\perp) \quad \text{for } t\geq 0.
\]
Its dynamics can be described by an SDE
\[
\dee (y_t, \bar{A}_t) = Z_0(y_t, \bar{A}_t) \dt + \sum_{k \in I} Z_k(y_t, \bar{A}_t) \circ \dee W^k_t,
\]
where the vector fields $Z_k$ (noise) and $Z_0$ (drift) are defined on this product space as:
\begin{align*}
 Z_k(y, A) &= (\sigma_k e_k, 0), \\
 Z_0(y, A) &= (-\eps y, DB(y)A).
\end{align*}
Here, the second component of $Z_0$ represents a right-invariant vector field on $\mathrm{SL}(H_I^\perp)$.

As it turns out, since the right invariant vector field $A \mapsto DB(y)A$ is linear in $y$, we can analyze the hypoellipticity of the process $(y_t, \bar{A}_t)$ by analyzing the matrix Lie algebra associated to the following collection of matrices $\{M_k\}_{k\in I}$ in $\mathfrak{sl}(H_I^\perp)$:
\begin{equation}\label{eq:Mk}
M_k := DB(e_k)|_{H_I^\perp}, \quad (M_k)_{\ell, m} = B_\ell(e_k, e_m) + B_\ell(e_m, e_k), \quad \ell, m \in T.
\end{equation}

Crucially, we are able to show via computer assisted proof (CAP) that the Lie algebra generated by the matrices $\{M_k\}_{k \in I}$ is $\mathfrak{sl}(H_I^\perp)$ (the Lie algebra of traceless matrices on $H_I^\perp$) which is a sufficient condition for the hypoellipticity of the process $(y_t, A^t_\perp)$.
\begin{proposition}[Algebraic Generation, CAP Result]\label{prop:slT_generation}
    For $N=3K$ with $K \ge 3$ (i.e., $N \ge 9$), let $M_k$ be defined by \eqref{eq:Mk}, the following holds:
    \[
    \mathrm{Lie}(\{M_k : k \in I\}) = \mathfrak{sl}(H_I^\perp).
    \]
\end{proposition}
The proof, detailed in Appendix \ref{app:CAP}, involves analyzing the sparse structure and a shift-invariance structure of the matrices $M_k$. Using rigorous computer assistance we verify the generation of $\mathfrak{sl}(H_I^\perp)$ from key elementary matrices derived from Lie brackets.

A key consequence of this is:

\begin{proposition}[Hörmander Condition for Linearization]\label{prop:hormander_linear}
Assume $N=3K$ with $K \ge 3$. The volume normalized process $(y_t, \bar{A}_t)$ satisfies the restricted H\"ormander condition on $H_I \times \mathrm{SL}(H_I^\perp)$. That is,
\[
\mathrm{Lie}(\mathcal{S}_{1, \mathrm{lin}})(y,A) = T_{(y,A)}(H_I \times \mathrm{SL}(H_I^\perp))\quad \text{for all}\quad (y,A) \in H_I \times \mathrm{SL}(H_I^\perp),
\]
where $\mathcal{S}_{1, \mathrm{lin}} = \{Z_k, [Z_0, Z_k] : k \in I \}$.
\end{proposition}
\begin{proof}
A direct calculation easily shows that
\[
[Z_k,Z_0](y,A) = (-\eps \sigma_k e_k, \sigma_k M_k A) \,.
\]
The set $\mathcal{S}_{1, \mathrm{lin}}(y,A)$ contains $\{(\sigma_k e_k, 0)\}_{k \in I}$ and $\{(-\eps \sigma_k e_k, \sigma_k M_k A)\}_{k \in I}$.
Since $\sigma_k \neq 0$ for $k\in I$, it follows that the span of $\mathcal{S}_{1, \mathrm{lin}}(y, A)$ contains $(e_k, 0), k \in I$, spanning $T_y H_I \times \{0\}$, as well as the vectors $(0, M_k A)$ for all $k \in I$.

We need to show that the Lie algebra generated by the right-invariant vector fields $A \mapsto M_k A$ on $\mathrm{SL}(H_I^\perp)$ spans the full tangent space $T_A \mathrm{SL}(H_I^\perp) \simeq \mathfrak{sl}(H_I^\perp)$. The Lie algebra of these vector fields is isomorphic to the matrix Lie algebra generated by $\{M_k : k \in I\}$.
By Proposition \ref{prop:slT_generation}, $\mathrm{Lie}(\{M_k : k \in I\}) = \mathfrak{sl}(H_I^\perp)$. Therefore, the generated vector fields span $T_A \mathrm{SL}(H_I^\perp)$.
Combining the spans, we conclude $\mathrm{Lie}(\mathcal{S}_{1, \mathrm{lin}})(y,A)$ spans $T_y H_I \times T_A \mathrm{SL}(H_I^\perp)$.
\end{proof}

Proposition \ref{prop:hormander_linear} establishes that the generator $\mathcal{L}^{\mathrm{lin}}$ of the $(y_t, \bar{A}^t_\perp)$ process is hypoelliptic. While not strictly necessary for the proofs that follow, the algebraic generation result underpinning this proposition is also key to analyzing the transverse projective process and full process on $H\backslash H_I$.

\subsection{H\"ormander Condition for the Projective Process}\label{sec:hypo_projective}

We now apply the H\"ormander framework to the transverse projective process $z_t = (y_t, v_t)$ introduced in Definition \ref{def:TransverseLinearProcess}.
The state space is the product manifold $\S^\perp H_I = H_I \times \mathbb{S}_I^\perp$, where $\mathbb{S}_I^\perp = \{v \in H_I^\perp : |v|=1\}$ is the unit sphere in the transverse space. The process evolves according to the Stratonovich SDE:
\begin{equation}\label{eq:proj_SDE}
\dee z_t = \tilde{X}_0^\perp(z_t) \dt + \sum_{k \in I} \tilde{X}_k(z_t) \circ \dee W_t^k,
\end{equation}
where the vector fields $\tilde{X}_k \in \mathfrak{X}(\S^\perp H_I)$ are constant lifts from $H_I$:
\[
\tilde{X}_k(y,v) = (\sigma_k e_k, 0), \quad k \in I.
\]
The drift vector field $\tilde{X}_0^\perp \in \mathfrak{X}(\S^\perp H_I)$ is given by
\[
\tilde{X}_0^\perp(y,v) = \left( -\eps y, DB(y)v - \eps v - \langle v, DB(y)v \rangle v \right).
\]
Since the diffusion vector fields $\tilde{X}_k$ are constant, the Itô-to-Stratonovich correction term is zero, and the Itô and Stratonovich forms of the SDE coincide.

\begin{proposition}[Hörmander Condition for the Projective Process]\label{prop:hormander_projective}
The transverse projective process $z_t = (y_t, v_t)$ defined by the SDE \eqref{eq:proj_SDE} satisfies the restricted Hörmander condition
\[
\mathrm{Lie}(\tilde{\mathcal{S}}_1)(y,v) = T_{(y,v)} \S^\perp H_I \quad \text{for all } (y,v) \in \S^\perp H_I,
\]
where $\tilde{\mathcal{S}}_1 = \{\tilde{X}_k, [\tilde{X}_0^\perp, \tilde{X}_k] : k \in I \}$. Consequently, it also satisfies the full parabolic Hörmander condition.
\end{proposition}

\begin{proof}
The proof requires showing that $\mathrm{Lie}(\tilde{\mathcal{S}}_1)(y,v)$ spans $T_{(y,v)}(\S^\perp H_I)$, where $\tilde{\mathcal{S}}_1 = \{\tilde{X}_k, [\tilde{X}_0^\perp, \tilde{X}_k] : k \in I\}$. Clearly the $\tilde{X}_k$ vector fields span the $T_y H_I$ component. Moreover, the Lie algebra contains the vector fields $[\tilde{X_k},\tilde{X}_0^\perp]$ whose projection onto $T_v \mathbb{S}^\perp_I$ are of the form
\[
V_{M_k}(v) = M_k v - \langle v, M_k v \rangle v,
\]
where $M_k$ is defined in \eqref{eq:Mk}. The vector field $V_{M_k}(v)$ is the infinitesimal generator of the action of the Lie group $\mathrm{SL}(H_I^\perp)$ on $\mathbb{S}^\perp_I$ induced by the linearization of the bilinear form $B$ at $v$. Thus, spanning the full space reduces to showing that the Lie algebra generated by the vector fields $\{V_{M_k}(v) : k \in I\}$ spans the tangent space $T_v \mathbb{S}^\perp$ for all $v \in \mathbb{S}_I^\perp$.

To establish that $\mathrm{Lie}(\{V_{M_k}(v) : k \in I\})$ spans $T_v \mathbb{S}^\perp_I$, we first recall the standard action of the Lie group $G = \mathrm{SL}(H_I^\perp)$ on the manifold $M' = H_I^\perp \setminus \{0\}$. This action is known to be transitive\footnote{An action of a group $G$ on a set $M$ is transitive if for any two points $x, y \in M$, there exists an element $g \in G$ such that $g \cdot x = y$.}. The projection $\pi: M' \to \mathbb{S}^\perp_I$, defined by $\pi(w) = w/|w|$, is a surjective submersion. Consequently, the transitive action of $G$ on $M'$ induces a transitive action on the sphere $\mathbb{S}^\perp_I$.
A fundamental result from the theory of Lie groups (see, e.g., \cite{Warner-Foundations-2010o} Chapter 3) states that if a Lie group $G$ acts transitively on a manifold $\mathcal{M}$, then the Lie algebra formed by the infinitesimal generators (which are vector fields) of this action spans the tangent space $T_x \mathcal{M}$ at every point $x \in \mathcal{M}$.
In our context, the infinitesimal generators of the $G = \mathrm{SL}(H_I^\perp)$ action, when projected onto $\mathbb{S}_I^\perp$, are precisely the vector fields $V_M(v) = Mv - \langle v, Mv \rangle v$, where $M \in \mathfrak{sl}(H_I^\perp)$. Since the action of $G$ on $\mathbb{S}^\perp_I$ is transitive, it follows that the Lie algebra $\mathrm{Lie}(\{ V_M(v) : M \in \mathfrak{sl}(H_I^\perp) \})$ spans the tangent space $T_v \mathbb{S}_I^\perp$ for any $v \in \mathbb{S}_I^\perp$.

The next step is to connect this spanning property to the specific generators $V_{M_k}(v)$. The mapping $M \mapsto -V_M$ defines a Lie algebra homomorphism from $\mathfrak{sl}(H_I^\perp)$ to $\mathfrak{X}(\mathbb{S}_I^\perp)$. This means that if a set of matrices $\{M_k\}$ generates $\mathfrak{sl}(H_I^\perp)$, then the corresponding vector fields $\{V_{M_k}\}$ will generate the Lie algebra $\mathrm{Lie}(\{V_M(v) : M \in \mathfrak{sl}(H_I^\perp)\})$. By Proposition \ref{prop:slT_generation}, we know that $\mathfrak{sl}(H_I^\perp)$ is indeed generated by the set $\{M_k : k \in I\}$. Therefore, the Lie algebra $\mathrm{Lie}(\{V_{M_k}(v) : k \in I\})$ is precisely $\mathrm{Lie}(\{V_M(v) : M \in \mathfrak{sl}(H_I^\perp)\})$, which we have just shown spans $T_v \mathbb{S}^\perp_I$. This completes the proof.
\end{proof}

\begin{corollary}[Properties of Projective Process]\label{cor:projective_properties}
The transverse projective process $z_t = (y_t, v_t)$ on $\mathbb{S}^\perp H_I$ has the following properties:
\begin{enumerate}
    \item It is topologically irreducible. That is, for any open set $U \subset \mathbb{S}^\perp H_I$, there exists $t > 0$ such that $\widehat{P}_t^\perp(\mathbf{1}_U)(y_0,v_0) > 0$ for all initial conditions $(y_0, v_0) \in \S^\perp H_I$ (see Appendix \ref{app:control_theory}).
    \item Its transition semigroup $\widehat{P}_t^\perp$ is strong Feller (maps bounded measurable functions to continuous functions).
    \item It admits a unique stationary measure $\nu_\eps$, which has a smooth, strictly positive density $f_\eps$ with respect to the volume measure on $\mathbb{S}^\perp H_I$.
\end{enumerate}
\end{corollary}
\begin{proof}
Property (1) follows from the Proposition \ref{prop:hormander_projective} via Proposition \ref{prop:irreducibility-general} if we verify the cancellation condition
\[
[\tilde{X}_k,[\tilde{X}_k,\tilde{X}_0^\perp]]=0.
\]
The above cancellation property follows from the analogous one on $H$, namely $B(e_k,e_k)=0$. Indeed, since the full projective lift $X_0 \mapsto \tilde{X}_0$ of the vector field $X_0$ to the full projective bundle $\S H$
\[
	\tilde{X}_0(u,v) = (X_0(u),\nabla X_0(u)v - \langle v, \nabla X_0(u)v\rangle)
\]
is tangent to the sub-bundle $\S^\perp H_I \subseteq \S H$ and therefore the vector field $\tilde{X}_0^\perp$ on $\S^\perp H_I$ is just given by the restriction of $\tilde{X}_0$ to $\S^\perp H$. The cancellation condition now easily follows by fact that the projective lift is a Lie algebra homomorphism (e.g. Lemma C2 in \cite{BBPS20}) and therefore
\[
	[\tilde{X}_k,[\tilde{X}_k,\tilde{X}_0]] = [X_k,[X_k,X_0]]\,\widetilde{\,\,} = B(e_k,e_k)\,\widetilde{\,} = 0
\]
due to the properties of the bilinearity $B$. (2) follows from the Hörmander condition (standard result, see e.g., \cite{Hormander67}). Existence of a stationary measure $\nu^\eps$ is evident from compactness of $\S^\perp_I$, while the rest of item (3) follows from the fact that time-$t$ transition kernels of $z_t$ have smooth, strictly positive densities -- properties which are consequences of topological irreducibility and H\"ormander's theorem.
\end{proof}

\subsection{Application to the Full Process \texorpdfstring{$u_t$ on $H \setminus H_I$}{ut on H HI}}\label{sec:hypo_full}

We now leverage the algebraic condition (Proposition \ref{prop:slT_generation}) to establish hypoellipticity and irreducibility for the full Lorenz-96 process $u_t$ on the state space $H \setminus H_I$.

\begin{proposition}[Bracket Spanning for Full Process]\label{prop:bracket_spanning_full} Let $u_t$ be the full Lorenz-96 process defined by \eqref{eq: L96-SDE}, then the restricted H\"ormander condition holds on $H\setminus H_I$. That is, for any $u \in H \setminus H_I$,
    \[
    \mathrm{Lie}(\mathcal{S}_{1})(u) = T_u H,
    \]
    where $\mathcal{S}_{1} = \{X_k, [X_0, X_k] : k \in I \}$.
\end{proposition}

\begin{proof}
Note $\mathrm{Lie}(\mathcal{S}_1)(u)$ contains $X_k = \sigma_k e_k$ for $k \in I$, these span the $H_I$ directions by our assumption that $\sigma_k \neq 0$ for all $k \in I$. Thus, it remains to show that $\operatorname{Lie}(\mathcal{S}_1)(u) \supset H_I^\perp$ for all $u \in H \setminus H_I$.

Fix $u \in H \setminus H_I$ and write $u = y + w$ for $y = \Pi u \in H_I$ and $w = \Pi^\perp u \in H_I^\perp$. For $k \in I$, consider the projection
\[
\Pi^\perp [X_0, X_k](u) = -\Pi^\perp( \sigma_kDB(e_k)u - \ep\sigma_ke_k) = -\sigma_k M_k w \,.
\]
Here, $M_k$ is defined by \eqref{eq:Mk}, and we have used that $\mathrm{Ker}(DB(e_k)) = H_I$ and $\mathrm{Ran}(DB(e_k)) = H_I^\perp$ (Lemma \ref{lem:TransverseInvariance}). Since $H_I \subset \mathrm{Lie}(\mathcal{S}_1)(u)$ for all $u$, it follows that $\Pi^\perp[X_0, X_k](u) \in \operatorname{Lie}(\mathcal{S}_1)(u)$ for all $k \in I$.

Writing
\[
Y_{M}(u) := M w,
\quad M \in \mathfrak{sl}(H_I^\perp) \, ,
\]
it follows that $Y_{M_k} \in \operatorname{Lie}(\mathcal{S}_1)$ for all $k \in I$.
Since the mapping $M \mapsto -Y_M$ is Lie algebra homomorphism
and since $\mathrm{Lie}(\{M_k : k \in I\}) = \mathfrak{sl}(H_I^\perp)$ by Proposition \ref{prop:slT_generation}, it follows that $\mathrm{Lie}(\mathcal{S}_{1})$ contains all  vector fields of the form $Y_M$ for $M \in \mathfrak{sl}(H_I^\perp)$.

To complete the proof, fix $j \in T$ such that $u_j \neq 0$, using that $u \in H \setminus H_I$. Fix $\ell \in T \setminus \{ j \}$ and let $E^{\ell, j} \in \mathfrak{sl}(H_I^\perp)$ be the
elementary matrix at row $\ell$ and column $j$. Then $Y_{E^{\ell, j}}(u) = u_j e_\ell$, and since $u_j \neq 0$, it follows that $e_\ell \in \mathrm{Lie}(\mathcal{S}_1)(u)$ for all $\ell \neq j$.

Spanning along $e_j$ follows similarly if $u_{j'} \neq 0$ for some $j' \in T \setminus \{ j \}$. If no other index is available, i.e., $u_{j'} = 0$ for all $j' \in T \setminus \{ j\}$, then spanning along $e_j$ follows on noting that, for any $j' \in T \setminus \{ j\} $, for $M = E^{j'j'} - E^{jj} \in \mathfrak{sl}(H^\perp_I)$, one has in this case that
\[Y_M(u) = u_{j'} e_{j'} - u_j e_j = - u_j e_j \, .  \]
This completes the proof in all cases.

\end{proof}

\begin{corollary}[Properties of Full Process]\label{cor:full_process_properties}
The full Lorenz-96 process $(u_t)$ on $H \setminus H_I$ satisfies the following properties.
\begin{enumerate}
    \item The process $(u_t)$ is topologically irreducible on $H \setminus H_I$.
    \item The Markov semigroup $P_t$ for $(u_t)$ has the strong Feller property.
    \item Any stationary measure $\mu$ on $H \setminus H_I$ must have a smooth, strictly positive density with respect to the volume measure on $H \setminus H_I$. In particular, any such stationary $\mu$ is unique.
\end{enumerate}
\end{corollary}
\begin{proof}
(1) follows from Proposition \ref{prop:bracket_spanning_full} and the cancellation property $B(e_k,e_k)=0$ via Proposition \ref{prop:irreducibility-general}. (2) follows from the Hörmander condition. For (3), the smoothness and positivity of the density follow from hypoellipticity (Hörmander's theorem) and irreducibility. Uniqueness of the stationary measure follows on recalling that the topological supports of distinct ergodic stationary measures of strong Feller processes are disjoint.
\end{proof}

\begin{remark}[Failure of Hörmander Condition on $H_I$]
    It is instructive to contrast Proposition \ref{prop:bracket_spanning_full} with bracket-spanning along the invariant subspace $H_I$. The drift on $H_I$ restricts to $X_0(y) = -\eps y$, with noise fields are $X_k = \sigma_k e_k$, hence $[X_0, X_k] = \eps X_k$. Since all $X_k$ are constant vector fields within $H_I$, the Lie algebra generated by $\{X_k, [X_0, X_k], \dots \}_{k \in I}$ evaluated at any point $y \in H_I$ can only span the subspace $H_I$ itself, not the full tangent space $T_y H \simeq H$.
    \end{remark}

\subsection{Geometric Ergodicity Results}\label{sec:GeoErgodicity_consolidated}

We collect here the key geometric ergodicity results for both the projective and the full processes, which rely on the hypoellipticity/irreducibility properties established above and suitable drift conditions. A central tool is Harris's Ergodic Theorem.

\subsubsection*{Harris's Theorem}

We state a version of Harris's Ergodic Theorem adapted for establishing geometric convergence in weighted norms (see e.g., \cite{meyn2012markov, hairer2011yet}).

Let $(X_t)$ be a Markov process on a Polish space $Z$ with transition semigroup $Q_t$. Let $V: Z \to [1, \infty)$ be a given weight function, and define the weighted supremum norm $\|\varphi\|_{C_V} := \sup_{z \in Z} |\varphi(z)|/V(z)$.

\begin{definition}[Irreducibility and Small Sets] \
\begin{itemize}
    \item The process $(X_t)$ is \emph{$\psi$-irreducible} if there exists a measure $\psi$ on $Z$ such that for any set $A$ with $\psi(A) > 0$, and any $z \in Z$, there exists $t > 0$ such that $Q_t(z, A) > 0$. (For processes satisfying the Hörmander condition, topological irreducibility typically implies $\psi$-irreducibility for $\psi$ being the volume measure).
    \item A set $C \subset Z$ is \emph{small} if there exist $T>0$, $\delta>0$, and a probability measure $\nu$ such that $P_T(z, \cdot) \ge \delta \nu(\cdot)$ for all $z \in C$. (For strong Feller processes, compact sets often satisfy the small set condition).
\end{itemize}
\end{definition}

\begin{theorem}[Harris's Theorem - Geometric Convergence in $C_V$] \label{thm:Harris_consolidated}
Assume the following conditions hold:
\begin{enumerate}
    \item $(X_t)$ is $\psi$-irreducible for some measure $\psi$.
    \item The transition semigroup $Q_t$ is strong Feller (maps bounded measurable functions to continuous functions).
    \item There exists a function $V: S \to [1, \infty)$ such that (i) the sublevel sets $\{ V \leq R\}$ are small sets for any $R \geq 1$; and (ii) there are constants $\lambda > 0$, $b < \infty$, such that
    \begin{equation}
            Q^t V(x) \leq e^{- \lambda t} V(x) + b \quad \text{for all } x \in S.
    \end{equation}
\end{enumerate}
Then there exists a unique stationary measure $\mu$ satisfying $\int V d\mu < \infty$. Furthermore, the process converges geometrically fast to $\mu$ in the weighted norm $\|\cdot\|_{C_V}$: there exist $r > 0$ such that for any function $\varphi$ with $\|\varphi\|_{C_V} < \infty$,
\[
\|P_t \varphi - \mu(\varphi) \mathbf{1}\|_{C_V} \le e^{-rt} \|\varphi\|_{C_V},
\]
where $\mu(\varphi) = \int \varphi d\mu$. This implies that $P_t$ has a spectral gap on the space $C_V = \{\varphi : \|\varphi\|_{C_V} < \infty\}$.
\end{theorem}

\subsubsection*{Geometric Ergodicity of the Projective Process}

The geometric ergodicity of the projective process $z_t = (y_t, v_t)$ is a key ingredient for constructing the Lyapunov function $\psi_p$ used in Section \ref{sec:Drift}. Recall the weighted space continuous functions $C_{V_\eta}$ with the associated weighted norm $\|\cdot\|_{C_{V_\eta}}$ defined in \eqref{eq:C_V}.

\begin{lemma}[Geometric Ergodicity of Projective Process]\label{lem:GeoErgoPP_hypo_consolidated}
For all $\eps > 0$, the transverse projective process $z_t = (y_t, v_t)$ is geometrically ergodic in the weighted space $C_{V_\eta}$ for $0 < \eta < \eta_\ast$. That is, there exists a unique stationary measure $\nu_\eps \in \mathcal{P}(\S^\perp H_I)$ and constants $r > 0, C \ge 1$ such that for all $\varphi \in C_{V_\eta}$,
\[
\|\widehat{P}_t^\perp \varphi - \nu_\eps(\varphi) \mathbf{1}\|_{C_{V_\eta}} \le e^{-rt} \|\varphi\|_{C_{V_\eta}},
\]
where $\nu_\eps(\varphi) = \int \varphi \, \mathrm{d}\nu_\eps$ is the expectation with respect to the stationary measure $\nu_\eps$. Equivalently, $\widehat{P}_t^\perp$ has a spectral gap in the weighted space $C_{V_\eta}$.
\end{lemma}
\begin{proof}
This follows from standard application of Harris's Ergodic Theorem \ref{thm:Harris_consolidated}. We need:
\begin{enumerate}
    \item Irreducibility: Established in Corollary \ref{cor:projective_properties}(2). We need $\psi$-irreducibility for Harris' theorem, which follows from topological irreducibility and the strong Feller property.
    \item Strong Feller property: Established in Corollary \ref{cor:projective_properties}(3).
    \item A drift condition: The super-Lyapunov property for $V_\eta(y)$ (Lemma \ref{lem:twisty} or \eqref{ineq:LVeta}) provides the necessary drift for the $y_t$ component. Since $\mathbb{S}^\perp_I$ is compact, $V_\eta(y)$ serves as a Lyapunov function for the joint process $z_t = (y_t, v_t)$ on $\mathbb{S}^\perp H_I = H_I \times \S_I^\perp$, satisfying $\mathcal{L}^z V_\eta \le -c V_\eta + C$ for some $c, C > 0$. This implies the drift condition required by Harris' theorem towards the compact sets $\{y : |y| \le R\} \times \mathbb{S}^\perp_I$, which are small sets due to the strong Feller property.
\end{enumerate}
\end{proof}

Since it is used elsewhere, we record the following consequence of the argument for Lemma \ref{lem:GeoErgoPP_hypo_consolidated}.

\begin{corollary}\label{cor:smoothPosDensityProj4}
    For any $\eps > 0$, the density $f_\eps$ of $\nu_\eps$ with respect to Lebesgue measure on $\S^\perp H_I$ is $C^\infty$ and strictly positive.
\end{corollary}

\begin{proof}[Proof sketch] H\"ormander's condition immediately implies $\nu_\eps$ admits a $C^\infty$ density $f_\eps$, while topological irreducibility implies $f_\eps > 0$ pointwise.
\end{proof}

\subsubsection*{Geometric Ergodicity of the Full Process}

Having established the existence and uniqueness of the second stationary measure $\mu$ supported on $H \setminus H_I$ (see Corollary \ref{cor:SecondMeasure_L96_merged}) by constructing an appropriate Lyapunov function $\mathcal{V}$ (see Section \ref{sec:Drift}), we now show that the convergence towards this measure is geometrically fast in the weighted norm defined by the Lyapunov function $\mathcal{V}$.

\begin{theorem}[Geometric Ergodicity of Full Process] \label{thm:GeoErgoFull_consolidated}
Let $\mu$ be the unique stationary measure on $H \setminus H_I$ satisfying $\int \mathcal{V} d\mu < \infty$ (existence established in Corollary \ref{cor:SecondMeasure_L96_merged}), where $\mathcal{V}$ is the Lyapunov function from \eqref{eq:LyapFunc-candidate}. Then the process $(u_t)$ is geometrically ergodic with respect to $\mu$ in the weighted norm $\|\cdot\|_{C_{\mathcal{V}}}$. That is, there exist $\gamma > 0$ such that for any $\varphi \in C_{\mathcal{V}}(H \setminus H_I)$ with $\int \varphi d\mu = 0$,
\begin{align*}
|\mathbb{E}_u \varphi(u_t)| \le K e^{-\gamma t} \mathcal{V}(u) \|\varphi\|_{C_{\mathcal{V}}}.
\end{align*}
\end{theorem}
\begin{proof}[Proof Sketch]
This again follows from Harris's Theorem \ref{thm:Harris_consolidated}.
\begin{enumerate}
    \item Irreducibility: Established for $H \setminus H_I$ in Corollary \ref{cor:full_process_properties}(2).
    \item Strong Feller property: Established for $H \setminus H_I$ in Corollary \ref{cor:full_process_properties}(3).
    \item Drift condition: The Lyapunov function $\mathcal{V}$ constructed in Section \ref{sec:Drift} satisfies the drift condition $\mathcal{L}\mathcal{V} \le -\lambda \mathcal{V} + C_0$ (Lemma \ref{lem:drift_L96_merged}). This provides the required drift towards level sets of $\mathcal{V}$, which serve as small sets.
\end{enumerate}
\end{proof}

\subsection{Quantitative Hypoellipticity}\label{sec:quant_hypo}

For certain arguments, particularly those involving the convergence regularity of $\psi_p$ (Lemma \ref{lem:psipC1} in Section \ref{sec:psip}) and the analysis of the Fisher information as $\eps \to 0$ (Lemma \ref{lem:FIT}), quantitative versions of the hypoelliptic estimates are needed. These estimates provide bounds on Sobolev norms that are uniform in certain parameters or depend polynomially on the location in the state space.

To do this, we need to define a uniform version of the H\"ormander condition.
\begin{definition}[Uniform Parabolic Hörmander Condition {\cite[Definition 2.1]{BBPS20}}] \label{def:UniformHormander}
Let $\cM$ be a smooth manifold, and let $\set{Z_0^\eps, Z_1^\eps, \dots, Z_r^\eps} \subset \mathfrak{X}(\cM)$ be a set of smooth vector fields parameterized by $\eps \in (0,1]$. For each $k \in \N$, define
\[
\mathcal{X}_k = \set{\mathrm{ad}_{Z_{i_1}^\eps}\mathrm{ad}_{Z_{i_2}^\eps} \cdots \mathrm{ad}_{Z_{i_{k}}^\eps} Z_j^\eps \mid 0 \le i_1, i_2, \dots, i_k \le r, 1 \le j \le r}.
\]
We say the family $\set{Z_0^\eps, \dots, Z_r^\eps}$ satisfies the \emph{uniform parabolic H\"ormander condition} on $\cM$ if there exists $k \in \N$ such that for any open, bounded set $U \subseteq \cM$, there exist constants $\set{K_n}_{n=0}^\infty$, such that for all $\eps \in (0,1]$ and all $x \in U$, there is a finite subset $V(x) \subset \mathcal{X}_k$ such that for all $\xi \in T_x \cM$, the following two conditions hold:
\begin{align}
\label{eq:unif_horm_span} &\abs{\xi} \leq K_0 \sum_{Z \in V(x)} \abs{Z(x) \cdot \xi}, \\
\label{eq:unif_horm_bound} &\sum_{Z \in V(x)} \norm{Z}_{C^n(U)} \leq K_n \quad \text{for all } n \ge 0.
\end{align}
Here $\norm{\cdot}_{C^n(U)}$ denotes a suitable $C^n$ norm on the set $U$.
\end{definition}
\begin{corollary}[Hörmander Conditions are Uniform]\label{cor:hormander_conditions_are_uniform}
The vector fields for the volume normalized linear process $(y_t, \bar{A}_t)$ (Proposition \ref{prop:hormander_linear}), the transverse projective process $(y_t, v_t)$ (Proposition \ref{prop:hormander_projective}), and the full process $u_t$ on $H \setminus H_I$ (Proposition \ref{prop:bracket_spanning_full}) all satisfy the Uniform Parabolic Hörmander Condition (Definition \ref{def:UniformHormander}).
\end{corollary}
\begin{proof}
    This is because the vector fields used to generate the tangent space in each case (specifically, the sets $\mathcal{S}_{1, \mathrm{lin}}$, $\tilde{\mathcal{S}}_1$, and $\mathcal{S}_1$ respectively) rely on the noise vector fields $X_k$ (or $\tilde{X}_k$) and brackets involving $M_k = DB(e_k)|_{H_I^\perp}$, which are independent of $\eps$. Consequently, the spanning condition \eqref{eq:unif_horm_span} and the $C^n$ bounds \eqref{eq:unif_horm_bound} can be satisfied uniformly for $\eps \in (0, 1]$.
\end{proof}

We define the H\"ormander norm pair (see discussions in e.g. \cite{Hormander67,Albritton2024-eh,bedrossian2021quantitative} for motivations). For a function $w \in C^\infty_c(S^\perp H_I)$ we define
\begin{align*}
\norm{w}_{\mathcal{H}} & := \norm{w}_{L^2} + \sum_{k \in I} \norm{\tilde{X}_k w}_{L^2}, \\
\norm{w}_{\mathcal{H}^\ast} & := \sup_{\varphi: \norm{\varphi}_{\mathcal{H}} \leq 1} \abs{ \int_{\S^\perp H_I} (\tilde{X}^\perp_0 \varphi) w\, \dee y \dee v}.
\end{align*}

The following quantitative H\"ormander estimate is essentially the same as [Lemma B.2; \cite{BBPS20}], adapted to our context.

\begin{lemma}[Quantitative Hörmander inequality for projective process]\label{lem:HineqBall_hypo}
Suppose the lifted vector fields $\{\tilde{X}_0^\perp, \tilde{X}_k\}_{k \in I}$ satisfy the uniform parabolic H\"ormander condition on $B(0,2R) \times \mathbb{S}^\perp_I$ for some $R \ge 1$. Then there exist $s > 0$ and $q > 0$, independent of $R$ and $\eps$, such that for any $w \in C_c^\infty(B(0,R) \times \mathbb{S}^\perp_I)$ and all $\eps \in (0,1)$,
\[
\|w\|_{H^s} \lesssim R^{q} (\norm{w}_{\mathcal{H}} + \norm{w}_{\mathcal{H}^\ast}),
\]
where the implicit constant is independent of $\eps$ and $R$. The fractional Sobolev norm $H^s$ on the product manifold $H_I \times \mathbb{S}_I^\perp$ (with dimension $m = |I| + |T| - 1$) is defined for $w \in C_c^\infty(B(0,R) \times \mathbb{S}_I^\perp)$ as:
\begin{align}
\norm{w}_{H^s}^2 & := \norm{w}_{L^2}^2 + \int_{\mathbb{S}^\perp H_I} \int_{h \in T_z(\S^\perp H_I ), |h| \le 1} \frac{|w(\exp_z h) - w(z)|^2}{|h|^{m + 2s}} \dee \sigma(h) \dee z, \label{def:Hs}
\end{align}
where $\exp_z$ is the exponential map at $z=(y,v)$ and $\dee \sigma(h)$ is a measure on the tangent space.
\end{lemma}

\section{Positivity of transverse Lyapunov exponents}\label{sec:FI}

We turn now to the remaining ingredients in the proof of Lemma \ref{lem:PosLyapL96} concerning positivity of $\lambda^\perp_\eps$. After some preliminaries (Section \ref{subsec:prelims5}), we sketch the proof of the Fisher information identity (Lemma \ref{lem:FIIdentity3}, restated below as Lemma \ref{lem:fisherInfo}) in Section \ref{subsec:FI5}, and finish in Section \ref{subsec:nonexistDensity5} with a proof of nonexistence of an invariant density for $z_t = (y_t, v_t)$ at $\eps = 0$ (Lemma \ref{lem:singular_limit}).

\subsection{Preliminaries}\label{subsec:prelims5}

Below, $\eps > 0$ is fixed. For $y \in H_I$ and $\omega \in \Omega$, recall the notation $A^t_\perp = A^t_{\perp, y, \omega} \in GL(H_I^\perp)$ for the full linearization of the transverse linear process $(w_t)$ on $H_I^\perp$ as in \eqref{eq:LinearizedODE_perp_recap}. Recall that  $\varphi^t = \varphi^t_\omega$ denotes the stochastic flow of $(u_t)$, which restricts to that of $(y_t)$ on $H_I$ as in \eqref{eq:OU-process} due to almost-sure invariance of $H_I$.

\begin{proposition}\label{prop:existLE5}
    For any $y \in H_I$, the following hold.
    \begin{itemize}
        \item[(a)] There is a constant $\lambda^\perp_\eps \in \R$ with the property that for any $v \in S^\perp_I$, it holds that
        \begin{align}
            \lambda^\perp_\eps = \lim_{t \to \infty} \frac1t \log \abs{A^t_{\perp} v}\quad \P-\text{a.s.}
        \end{align}
        \item[(b)]
        It holds that
        \begin{align}
            \lim_{t \to \infty} \frac1t \log \abs{\det A^t_\perp} = - \eps |T| \quad \P-\text{a.s.}
        \end{align}
        \item[(c)]
        It holds that
        \begin{align}
            \lim_{t \to \infty} \frac1t \log \abs{\det D_y \varphi^t_\omega}= - \eps |I|\quad \P-\text{a.s.}
        \end{align}
    \end{itemize}
\end{proposition}

\begin{proof}
    Item (a) follows from the ``nonrandom'' version of the Multiplicative Ergodic Theorem \cite[Theorem III.1.2]{kifer2012ergodic} and uniqueness of the stationary measure $\nu_\eps$ for the $(y_t, v_t)$ process (Lemma \ref{lem:GeoErgoPP_hypo_consolidated}). This step uses log-integrability of $\| A^t_\perp\|$ and $\| (A^t_\perp)^{-1}\|$ as in Corollary \ref{cor:twistyFC}.
    
    For (b), since $DB(y)$ is trace-free and leaves $H_I^\perp$ invariant (Lemma \ref{lem:TransverseInvariance}), it holds that
    \[\frac{\dee}{\dt} \log \abs{\det A^t_\perp} = - \eps \operatorname{Tr}(\operatorname{Id}|_{H_I^\perp} ) =  - \eps |T| \, . \]
    That $\lim_t \frac1t \log \abs{\det A^t_\perp} = - \eps |T|$ is now immediate. Similarly, for (c) one computes
    \[\frac{\dee}{\dt} \log \abs{\det D_y \varphi^t_\omega|_{H_I}} = - \eps \operatorname{Tr}(\operatorname{Id}|_{H_I} ) = - \eps |I| \, . \qedhere \]

\begin{remark}
    It is standard that the Birkhoff averages of time-$t$ determinants appearing in Proposition \ref{prop:existLE5}(b) and (c) coincide with the sum of the corresponding Lyapunov exponents counted with multiplicity. Reflecting this, we will write
    \[\lambda^\perp_{\Sigma, \eps} := - \eps |T| \, , \qquad \lambda^I_{\Sigma, \eps} := - \eps |I| \, \]
    for these quantities.
\end{remark}

\end{proof}

{Next, we record the following form of the Furstenberg-Khasminskii formula in our setting, c.f. \cite[Section 6.2.2]{arnold1998random}.

\begin{lemma}\label{lem:furstKhasmAPP}
The transverse Lyapunov exponent $\lambda^\perp_\eps$ is given by
    \[\lambda^\perp_\eps = \int \langle DB(y) v , v \rangle \dee \nu_\eps (y, v) - \eps \, .\]
\end{lemma}
\begin{proof}
    This is immediate from the Birkhoff ergodic theorem and ergodicity of $\nu_\eps$, on observing that for the process $(w_t)$ on $H_I^\perp$ one has
    \[\frac{\dee}{\dt} \log |w_t| = \langle DB(y_t) v_t, v_t\rangle - \eps \, . \qedhere \]
\end{proof}
  }

Finally, we record an additional regularity estimate on $f_\eps$, the density of $\nu_\eps$ with respect to Lebesgue measure.
\begin{lemma}\label{lem:momentEsts}
For all $\eps > 0$ and $J > 0$, there holds
 $$\int_{\S^\perp H_I} \brak{y}^J f_\eps \log f_\eps \, \dee y \dee v < \infty.$$
\end{lemma}
We caution that the above estimate is not necessarily uniform in $\eps$.

\begin{proof}
This proceeds exactly as in in \cite[Theorem B.1]{BBPS20} using the quantitative hypoelliptic regularity estimate \cite[Lemma B.2]{bedrossian2021quantitative} interpolated against the moment bound from the Lyapunov function, $V_\eta \in L^1(\nu_\eps)$
\end{proof}

\subsection{Fisher information identity for $\lambda_\eps^\perp$}\label{subsec:FI5}

We now turn to the proof of the Fisher information identity (Lemma \ref{lem:FIIdentity3}), restated below for convenience.
Below, given a smooth  $\varphi : \S^\perp H_I \to (0,\infty)$ we write
\[FI(\varphi) :=  \sum_{j\in I} \frac{1}{2}\int_{\mathbb S^\perp H_I } \frac{\abs{ \sigma_{j} \partial_{y_{j}} \varphi}^2}{\varphi} \dee y \dee v\]
for its \emph{Fisher information}.

\begin{lemma}\label{lem:fisherInfo}
For all $\eps > 0$, the following formula holds:
\begin{align*}
\eps FI(f_\eps) =  |T| \lambda^\perp_\eps - \lambda_{\Sigma,\eps}^\perp - \lambda_{\Sigma,\eps}^I = \frac{2 N}{3} \lambda^\perp_\eps + \eps N
\end{align*}
\end{lemma}

\begin{proof}[Proof sketch of Lemma \ref{lem:fisherInfo}]
    Recall that the density $f_\eps$ is a solution to the forward Kolmogorov equation

\begin{align}\label{eq:FKapp}
    (\Lc^z)^* f_\eps = 0 \, ,
\end{align}
 where
\[\Lc^z = \tilde X_0^\perp + \frac{\eps}{2} \sum_{k \in I} \tilde{X}_k^2\]
is the generator in H\"ormander form, notation as in Section \ref{sec:hypo_projective},
and $(\Lc^z)^*$ is the formal $L^2$ dual, given here by
\[(\Lc^z)^* = (\tilde X_0^\perp)^* + \frac{\eps}{2} \sum_{k \in I} \tilde X_k^2\]
with
\begin{align} (\tilde X_0^\perp)^* &= - \tilde X_0^\perp - \Div \tilde{X}_0^\perp  \\
& = - \tilde X_0^\perp + \eps |I| + |T| \langle DB(y) v, v \rangle \,.
\end{align}

    Multiplying the left- and right-hand sides of \eqref{eq:FKapp} by $\log f_\eps$ and integrating jointly in $y$ and $v$ gives
    \begin{align}\label{eq:FKlogAPP}
        \int_{\S^\perp H_I} \log f_\eps \,\, (\tilde X_0^\perp)^* f_\eps \, \dee y \dee v =
        - \frac{\eps}{2} \sum_{k \in I} \int_{\S^\perp H_I} \log f_\eps \, \, (\tilde X_k)^2 f_\eps \, \dee y \dee v \,.
    \end{align}
    In what follows, we apply integration by parts to both sides of \eqref{eq:FKlogAPP}. That these manipulations are valid follows from Lemma \ref{lem:momentEsts} and a straightforward adaptation of the proof of \cite[Proposition 3.2]{BBPS20}, to which we refer the reader for further details.

    Proceeding with the formal computation, we have that the LHS of \eqref{eq:FKlogAPP} develops as
\begin{align}
    \int_{\S^\perp H_I} f_\eps \tilde X_0^\perp (\log f_\eps ) \; \dee y \dee v &  =  \int_{\S^\perp H_I} \tilde X_0^\perp f_\eps \, \dee y \dee v \\
    & = -\int_{\S^\perp H_I} \Div \tilde X_0^\perp \, d \nu_\eps \, ,
\end{align}
using that $\int_{\S^\perp H_I} (\tilde X_0^\perp)^* f \dee y \dee v = 0$. Plugging in the form of $\Div \tilde X_0^\perp$, we conclude that the LHS of \eqref{eq:FKlogAPP} is given by
\[ \eps |I| + |T| \int \langle DB(y) v, v \rangle d \nu_\eps(y, v) = \eps |I| + |T| (\lambda^\perp_\eps + \eps)  = |T| \lambda^\perp_\eps - \lambda^I_{\Sigma, \eps} - \lambda^\perp_{\Sigma, \eps}\,. \]
Meanwhile, integrating by parts in the RHS of \eqref{eq:FKlogAPP} gives
\[
    \frac{\eps}{2} \sum_{k \in I} \int_{\S^\perp H_I} \frac{|\tilde X_k f_\eps|^2 }{f_\eps} \dee y \dee v = \eps FI(f_\eps) \, . \qedhere
\]
\end{proof}

\subsection{Nonexistence of invariant density for $z_t$ at $\eps = 0$}\label{subsec:nonexistDensity5}

The argument presented in Section \ref{subsec:3tverseLE} demonstrates that, in pursuit of a contradiction, $\liminf_{\eps \to 0} \eps^{-1} \lambda^\perp_\eps < \infty$ implies the existence of an invariant density $f_0$ for the \emph{deterministic} ($\eps = 0$) process $z_t = (y_t, w_t)$ determined by the ordinary differential equation
\[\begin{cases}
  \dot y = 0 \, , \\
  \dot v = DB(y) v - v \langle v, DB(y) v \rangle \, ,
\end{cases}
\]
solved for fixed initial $(y_0, v_0)$ by
\[y_t = y_0 \, , \quad v_t = \frac{A_\perp^t v_0}{|A^t_\perp v_0|}\]
with
\[A_\perp^t = e^{t DB(y)|_{H^\perp_I}} \,. \]
Since $f_0$ is an $L^1$-limit of densities $f_\eps$, we note that $f_0$ on $\S^\perp H_I \cong H_I \times \S^\perp_I$ projects to the Gaussian density $\rho$ on $H_I$ from \eqref{eq:Gaussian_Density}.

It remains to prove that no such density can exist, as we show below.
\begin{lemma}\label{lem:singular_limit}
Let $\nu$ be any invariant probability measure for the $\eps = 0$ (deterministic) transverse projective process  with the property that $\nu(A \times \mathbb S^\perp_I) = \mu^I(A)$. Then, there is a nonempty open set $U \subset H_I$ with the property that $\nu|_{U \times \S^\perp_I}$  is singular with respect to Lebesgue measure on $U \times \mathbb S^\perp_I$.
\end{lemma}

Before proving Lemma \ref{lem:singular_limit}, we first identify a (zero Lebesgue-measure) set of $y \in H_I$ for which $DB(y)|_{H^\perp_I}$ is linearly unstable.

\begin{lemma}\label{lem:Eigenvalue-unstable}
Fix $a, b \in \R$ such that $a (b - a) > 0$, and let
\begin{align*}
y = y_{a,b} := a e_0 + b e_3 \, .
\end{align*}
Then, $DB(y)|_{H^\perp_I}$ admits an eigenvalue with positive real part.
\end{lemma}
\begin{proof}[Proof of Lemma \ref{lem:Eigenvalue-unstable}]

For $y = y_{a,b}$, the space
\[V = \operatorname{Span}\{ e_1, e_2, e_4, e_5\} \subset H_I^\perp\]
is left invariant by $DB(y)$. With respect to the basis $\{ e_1, e_2, e_4, e_5\}$, the matrix of $DB(y) = DB(y_{a,b})$ restricted to $V$ takes the form
\[\mathfrak{M} = \begin{pmatrix}
0     & a & 0  & 0 \\
b - a & 0 & 0  & 0 \\
0    & -b & 0  & b \\
0     & 0 & -b & 0
\end{pmatrix}\]

The characteristic polynomial $p(t) = \det(\mathfrak{M} - t \operatorname{Id} )$ is given by
\[p(t) = a^2 b^2+a^2 t^2-a b^3-a b t^2+b^2 t^2+t^4 = \left(b^2+t^2\right) \left(a^2-a b+t^2\right)\, , \]
and has roots
\[t = \pm i b \, , \pm \sqrt{a (b - a)} \,.   \]
Since $V$ is invariant and since $DB(y)|_V$ is linearly unstable, it follows that $DB(y)|_{H_I^\perp}$ is likewise linearly unstable for all such $y = y_{a,b}$.
\end{proof}

\begin{proof}[Proof of Lemma \ref{lem:singular_limit}]
    Lemma \ref{lem:Eigenvalue-unstable} and standard facts about continuity of spectra of finite matrices imply the existence of an open set $U \subset H_I$ for which $DB(y)|_{H^\perp_I}$ admits an eigenvalue $\lambda$ with positive real part. For each $y \in U$, let $E(y)$ denote the direct sum of all generalized eigenspaces corresponding to eigenvalues with positive real part, and let $F(y)$ denote the complementary direct sum of generalized eigenspaces corresponding to eigenvalues with nonpositive real part. Again by spectral continuity and on shrinking $U$, we can assume (i) $\dim E(y)$ is constant along $U$, and (ii) $y \mapsto E(y)$ varies continuously.
    
    Let
    \[\mathcal{S} = \{ (y, v) : y \in U, v \in \S^\perp_I \setminus F(y) \} \, ,\]
    which is measurable, and observe that for all $(y, v) \in \mathcal{S}$,
    \begin{align} \label{eq:angleConverge5}\angle(A^t_\perp v, E(y)) \to 0 \quad \text{ as } t \to \infty \, , \end{align}
    where $\angle(\cdot, \cdot)$ denotes the minimal angle between a vector and a subspace.

    Let now $\nu$ be an invariant measure with $\nu(A \times \S^\perp_I) = \mu^I(A)$ for all Borel $A \subset H_I$.
    Note that $\nu(U \times \S^\perp_I) > 0$. We will show that $\nu(\mathcal{S}) = 0$, which implies singularity w.r.t. Lebesgue on $U \times \S^\perp_I$ since $\mathcal{S}$ has full Lebesgue measure in $U \times \S^\perp_I$.
    
    To show this, observe that if $\nu(\mathcal{S}) > 0$, then $\angle(v, E(y)) > 0$ for a positive $\nu$-measure set of $(v, y) \in U \times \S^\perp_I$.
    The Poincar\'e Recurrence Theorem implies that for $\nu$-almost every $(v, y) \in U \times \S^\perp_I$, it holds that $\angle(v_{t_k}, E(y)) \geq \frac{99}{100} \angle (v, E(y))$ for a sequence $t_k \to \infty$. This is incompatible with \eqref{eq:angleConverge5}. We conclude that $\nu(\mathcal{S}) = 0$, as desired.
\end{proof}

\section{Construction of the Lyapunov Function} \label{sec:Drift}


Throughout, $\eps > 0$ is fixed\footnote{In particular, we are no longer concerned in Section \ref{sec:Drift} with $\epsilon$-uniform estimates.} once and for all so that $\lambda^\perp_\eps > 0$. Let $\eta \in (0, \eta_*)$ be fixed, with $\eta_*$ as in Lemma \ref{lem:twisty}, and  write $V = V_\eta$ for short.

For $p \in \R$, recall that the \emph{twisted semigroup} $\widehat{P}^{\perp, p}_t$ acts on an observable $\varphi : \S^\perp H_I \to \R$ by
\[\widehat{P}^{\perp, p}_t \varphi (y, v) = \EE_z \left[ \frac{1}{|A^t_\perp v|^p} \varphi(y_t, v_t) \right] \]
for $z = (y, v)$, whenever the RHS expectation is defined. As indicated in Section \ref{subsubsec:3LFconstruct}, the  Lyapunov function $\mathcal{V}$ for the $(u_t)$ process on $H \setminus H_I$ is constructed from the dominant eigenfunction $\psi_p$ for $\widehat{P}^{\perp, p}_t$, which will satisfy the eigenfunction relation
\[\widehat{P}^{\perp, p}_t \psi_p = e^{- \Lambda(p) t} \psi_p\]
where $\Lambda(p) \in \R$ is the associated moment Lyapunov exponent.

In this section we will flesh out the following necessary ingredients originally presented in Section \ref{subsubsec:3LFconstruct}:

\begin{itemize}
  \item[(i)] The semigroup $\widehat{P}^{\perp, p}_t$ admits a spectral gap on the weighted space $C_{V}$.
  \item[(ii)] The asymptotic
  \[\Lambda(p) = p \lambda^\perp_\eps + o(p)\]
  holds for $|p| \ll 1$. In particular, in view of positivity of $\lambda^\perp_\eps$ it holds that $\Lambda(p) > 0$ for $0 < p \ll 1$.
  \item[(iii)] For $p > 0$ small and fixed, it holds that $\psi_p > 0$ pointwise and belongs to the higher regularity space $C^1_{V_\eta}$ (notation as in Section \ref{subsubsec:3LFconstruct}).
  
\end{itemize}

Items (i), (ii) and positivity of $\psi_p$ are treated below in Section \ref{sec:psip}. Higher regularity of $\psi_p$ as in item (iii) will be treated in Section \ref{sec:QuantEst}.

\subsection{Spectral picture of $\widehat{P}^{\perp, p}_t$}\label{sec:psip}

In this section, we construct the dominant eigenfunction $\psi_p$ of the twisted semigroup $\widehat{P}^{\perp, p}_t$ via spectral perturbation theory. The key idea is that $\widehat{P}^{\perp, p}_t$ is a norm-continuous perturbation of the semigroup $\widehat{P}^{\perp}_t$ associated to the transverse projective process, allowing us to apply standard spectral perturbation results.

We begin by defining the generator of the twisted semigroup $\widehat{P}^{\perp, p}_t$. Let $\mathcal{L}^z$ denote the generator of the transverse projective process $z_t = (y_t, v_t)$ on $\S^\perp H_I$. By the Feynman-Kac formula, the twisted semigroup $\widehat{P}^{\perp, p}_t$ has generator
\begin{align}\label{eq:Lp_definition}
\mathcal{L}_p  := \mathcal{L}^z - p\widetilde{H},
\end{align}
where the perturbation potential $\widetilde{H}(y,v) = \langle v, (DB(y)-\eps\Id) v \rangle$ corresponds to the multiplicative factor in the Feynman-Kac representation.

We begin by confirming the spectral gap of $\widehat{P}^{\perp, p}_t$ for $p$ sufficiently small. This will be derived from the following norm-continuity estimate.

\begin{lemma}\label{lem:convergTwisty6}
For all $p \in \mathbb R$, the operator $\hat{P}_t^{\perp,p}$ is a $C^0$-semigroup $C_V \to C_V$ for all $ 0 < \eta < \eta_\ast$.
Moreover,  $\forall T > 0$ there holds
\begin{align*}
\lim_{p \to 0} \sup_{t \in [0,T]} \norm{\widehat{P}^{\perp,p}_t - \widehat{P}^{\perp}_t}_{C_V} = 0.
\end{align*}
\end{lemma}
\begin{proof}
Both statements follow quickly from Lemma \ref{lem:twisty} and, in particular, Corollary \ref{cor:twistyFC}.
\end{proof}

\begin{corollary} \
There exists $p_0 > 0$ such that for all $p \in [-p_0, p_0]$, the operator $\widehat{P}^{\perp, p}_1$ admits a spectral gap with real dominant eigenvalue $> 0$.
\end{corollary}
\begin{proof}
  This is immediate from standard spectral perturbation theory for discrete spectrum and the spectral gap for $\widehat{P}^\perp_1$ established in Lemma \ref{lem:GeoErgoPP_hypo_consolidated}.
\end{proof}

Write $e^{- \Lambda(p)}, \Lambda(p) \in \R$ for the dominant eigenvalue of $\widehat{P}^{\perp, p}_1$, and $\pi_p$ for the spectral projector to the (one-dimensional) dominant eigenspace.
\begin{align*}
& \lim_{p \to 0}\norm{\pi_p - \pi_0}_{C_V \to C_V} = 0.
\end{align*}
It now follows that the limit
\begin{align}\label{eq:limDefinPsiP6}
\psi_p := \lim_{n \to \infty} e^{n \Lambda(p)} \widehat{P}^{\perp, p}_n {\bf 1}
\end{align}
exists in the $C_V$ norm, where ${\bf 1}$ is the constant function identically equal to 1, and $\psi_p$ satisfies
\begin{align}\label{eq:eigenRelationPsip6}
  \widehat{P}^{\perp, p}_t \psi_p = e^{- t\Lambda(p)} \psi_p \qquad \text{ for } t = 1, 2, \dots \,.  \end{align}
The spectral mapping theorem for point spectrum \cite[Section A-III]{arendt1986one} now implies \eqref{eq:eigenRelationPsip6} for all real $t \geq 0$.

Moving on, we check now basic properties of $\psi_p$ and $\Lambda(p)$.

\begin{lemma} \
  \begin{itemize}
  \item[(a)] $\psi_p$ is (i) $C^\infty$ smooth and (ii) strictly positive pointwise on $\S^\perp H_I$.
  \item[(b)] The function $p \mapsto \Lambda(p)$ is differentiable at $p = 0$, and satisfies
  \[\frac{\dee}{\dee p}\bigg|_{p = 0} \Lambda(p)  =\lambda^\perp_\eps \,. \]
  In particular, $\Lambda(p) > 0$ for all $p > 0$ sufficiently small.
  \end{itemize}
  
\end{lemma}
\begin{proof}
  For (a)(i): Since $\psi_p$ belongs to the range of $\widehat{P}^{\perp, p}_t$ for all $t > 0$, hypoellipticity of the $z_t = (y_t, v_t)$ process (Proposition \ref{prop:hormander_projective}) implies smoothness of $\psi_p$.

  For positivity, observe that $\widehat{P}^{\perp, p}_t$ sends nonnegative functions to nonnegative functions, and so that $\psi_p \geq 0$ is immediate from \eqref{eq:limDefinPsiP6}.
  Moreover, since $\psi_p$ is continuous and not identically zero, $U_p := \set{z \in \S^\perp H_I : \psi_p > 0}$ is non-empty and open.
  By topological irreducibility (Corollary \ref{cor:projective_properties}), it follows that $\PP_z (z_t \in U_p) > 0$ for all $z \in \S^\perp H_I$. Therefore,
  \[\psi_p(z) \geq \P_z(z_t \in U_p) \EE_z \left(|A^t_\perp v|^{-p} \psi_p(z_t) | z_t \in U_p\right) > 0 \,.\]

Item (b) is standard and follows on checking \begin{align}\label{eq:limitBold16}
    \Lambda(p) = - \lim_{t \to \infty} \frac{1}{t} \log \widehat{P}^{\perp, p}_t {\bf 1}
  \end{align}
  pointwise in $z = (y, v) \in \S^\perp H_I$; see, e.g., \cite[Lemma 5.10]{bedrossian2022almost} for a review of how \eqref{eq:limitBold16} impiles differentiability of $\Lambda(p)$ at $p = 0$. In turn, \eqref{eq:limitBold16} follows from positivity of $\psi_p$ and the spectral gap from Lemma \ref{lem:convergTwisty6}.
\end{proof}

The following lemma establishes that $\psi_p$ is an eigenfunction of both the semigroup and its generator.
\begin{lemma} \label{lem:psipeigen}
The function $\psi_p$ is a smooth eigenfunction of $\mathcal{L}_p$ with eigenvalue $-\Lambda(p)$. Specifically, $\psi_p$ solves the eigenfunction PDE
\begin{align}
\mathcal{L}_p\psi_p = -\Lambda(p)\psi_p \label{eq:psip_eigenPDE}
\end{align}
in the classical sense.
\end{lemma}
\begin{proof}
We establish the eigenfunction property by showing that $\psi_p$ solves the PDE \eqref{eq:psip_eigenPDE} in the classical sense.

Choose a cut-off function $\chi \in C^\infty_c(H_I \times \mathbb S^{\abs{T}-1})$ such that $\chi$ depends only on the $H_I$ variables, that $\chi \equiv 1$ on a ball of radius $M$ and that $\abs{\grad^\ell \chi} \lesssim M^{-\ell}$.
By the Feynman-Kac formula, we have that $u = \widehat{P}_t^{\perp,p} \chi \psi_p$ is smooth for positive times (by H\"ormander's theorem) and solves the Kolmogorov equation
\begin{align}
\partial_t u & = \mathcal{L}_p u \label{eq:kolmogorov} \\
u(0,x,v) & = \chi(x) \psi_p(x,v). \nonumber
\end{align}
Let $\eta_0 > \eta' > \eta$ be arbitrary.
It is straightforward to check from the definition of $\widehat{P}^t_{\perp,p}$ and Lemma \ref{lem:twisty} that
\begin{align}
\lim_{M \to \infty} \sup_{t \in (0,1)} \norm{u(t) - \psi_p}_{C_{V_{\eta'}}} = 0. \label{eq:cutoff_convergence}
\end{align}
It follows that $\psi_p$ solves \eqref{eq:psip_eigenPDE} in the sense of distributions.
Since $\psi_p \in C^\infty$, it is hence also a classical solution, establishing \eqref{eq:psip_eigenPDE}.

\end{proof}

\subsection{Higher regularity for $\psi_p$}\label{sec:QuantEst}
Now we are ready to upgrade the regularity of $\psi_p$ to $C^1_V$ by iterating local hypoelliptic regularizations. By Lemma \ref{lem:psipeigen}, $\psi_p$ satisfies the eigenfunction equation \eqref{eq:psip_eigenPDE}
where $\mathcal{L}_p$ is the generator \eqref{eq:Lp_definition} of the twisted semigroup. Expanding this using the definition of $\mathcal{L}_p$, we have
\begin{align}
\mathcal{L}^z\psi_p = (p\widetilde{H} -\Lambda(p))\psi_p.  \label{def:psipPDE_expanded}
\end{align}

\begin{lemma}\label{lem:psipC1}
For all $0 < \eta < \eta_0$, there holds $\psi_p \in C^1_{V_\eta}$.
\end{lemma}
\begin{proof}
We will use direct PDE estimates to prove that $\exists q' > 0$ such that $\forall y \in H_I$ with $\abs{y} \geq 1/2$, there holds
\begin{align}
\norm{\grad \psi_p}_{L^\infty(B(y,1) \times \mathbb S^\perp_I)} \lesssim \abs{y}^{q'} \norm{\psi_p}_{L^2(B(y,2) \times \mathbb S^\perp_I)}. \label{ineq:Gradpsip}
\end{align}
Notice that this implies for any $0 < \eta' < \eta_0$,
\begin{align*}
\norm{\grad \psi_p}_{L^\infty(B(y,1) \times \S^\perp_I )} \lesssim_{\eta'} V_{\eta'}(y) \norm{\psi_p}_{C_{V_{\eta'}}},
\end{align*}
which implies the desired $C_{V_\eta}^1$ bound for all $\eta \leq \eta'$. Hence, we need only prove \eqref{ineq:Gradpsip}.

For $y = 0$, we can fix a finite atlas for the set $B(0,2) \times \mathbb S^{\abs{T}-1}$ which provides local diffeomorphisms that map the PDE \eqref{eq:psip_eigenPDE} into a similar Kolmogorov equation in $\R^{m}$ with smooth coefficients. Moreover, these transformed Kolmogorov equations also satisfy the parabolic H\"ormander condition uniformly in $\eps$.
By translation invariance of the geometry, this finite atlas induces a corresponding finite atlas on all sets of the form $B(y,2) \times \mathbb S^{\abs{T}-1}$ and local diffeomorphisms which only depend on $y$ through a translation.
For definiteness, we set $J$ to be the number of charts.
Denote $\set{\chi_{j,y}}_{j=1}^J$ a set of smooth cutoff functions adapted to the charts and associated diffeomorphisms $\Phi_{j,u}:\textup{supp} \chi_{j,u} \to \R^{m}$.
These charts reduce \eqref{ineq:Gradpsip} to problems posed locally on $\R^{m}$. However, while the geometry is bounded, the coefficients of the Kolmogorov equations and the conditioning of the parabolic H\"ormander condition (specifically the constants in Definition \ref{def:UniformHormander}) are \emph{not} translation invariant. Indeed, the conditioning degenerates as an inverse power of $y$ and the coefficients of the Kolmogorov equation grow polynomially as $y \to \infty$. This is the source of the $q' > 0$ in \eqref{ineq:Gradpsip}.

For any $R,y_0$ fixed with $R = \abs{y_0}$, we consider one of the $J$ transformed Kolmogorov equations.
That is, if we let $\psi^j = \psi_p \circ \Phi_{j,y_0}^{-1}$ we obtain a transformed Kolmogorov equation
\begin{align*}
\mathcal{L}^{(y_0,j)}_p \psi^{j} := \mathcal{L}^{z,(y_0,j)} \psi^{j} - p \tilde{H}^{(j)} \psi^j = -\Lambda(p) \psi^j,
\end{align*}
where $\mathcal{L}^{z,(y_0,j)} = \tilde{X}_0^{(y_0,j)} + \frac{\eps}{2} \sum_{k=1}^r (\tilde{X_k}^{(y_0,j)})^2$ is the transformed generator of the transverse projective process and $\tilde{H}^{(j)}$ is the transformed perturbation term. Rearranging,
\begin{align*}
\mathcal{L}^{z,(y_0,j)} \psi^{j} = -\Lambda(p) \psi^j + p \tilde{H}^{(j)} \psi^j,
\end{align*}
which holds only in some open ball near the origin, which we denote $B(0,\delta) \subset \mathbb R^m$.
Due to the open overlap from one chart to another, there is some other $0 < \delta' < \delta $ on which it suffices to prove the further localized and transformed version of \eqref{ineq:Gradpsip},
\begin{align}
\norm{\grad \psi^j}_{L^\infty(B(0,\delta') \times \S^\perp_I )} \lesssim R^{q'}\norm{\psi^j}_{L^2(B(0,\delta) \times \S^\perp_I )}. \label{ineq:LocGoal}
\end{align}
Due to the bounded geometry, Lemma \ref{lem:HineqBall_hypo} applies to the family $\{\tilde{X}_k^{(y_0,j)}\}_{k=0}^r$ (where the constants depend only on $R$, not on $y_0$ or $j$), i.e. if we define the norms $\mathcal{H}$ and $\mathcal{H}^\ast$, Lemma \ref{lem:HineqBall_hypo} holds with constants independent independent of $y_0,j$ except through $R = \abs{y_0}$.
This is analogous to the observations used in [Lemma B.2; \cite{BBPS20}] and [Lemma 2.3; \cite{bedrossian2021quantitative}].

We will use Sobolev embedding to obtain $C^1$ regularity estimates.
Let $n_\ast$ be an integer with $n_\ast > \frac{1}{2s}m + \frac{1}{s}$.
Let $\set{r_k}_{k=0}^{n_\ast+1}$ be such that $\delta' < r_{n_\ast + 1} < ... < r_k < r_{k-1} < ... < r_0 < \delta$.
Next, consider the collection of concentric balls in $\mathbb R^{m}$, $B_k := B(0, r_k)$ with adapted cutoff functions $\chi_k \in C^\infty_c(B_{k})$ with $\chi_k \equiv 1$ on $B_{k+1}$ and $\abs{\grad^\ell \chi_k} \lesssim K^{\ell}$ for some constant $K$ depending on $\delta,\delta'$ and $n_\ast$.

Applying the cutoff to the localized Kolmogorov equation gives
\begin{align}
\mathcal{L}^{z,(y_0,j)} (\chi_1\psi^j) & = -[\mathcal{L}^{z,(y_0,j)},\chi_1]\psi^j - \Lambda(p) \chi_1 \psi^j + p \tilde{H}^{(j)} \chi_1 \psi^j. \label{eq:LocKol}
\end{align}
Note that
\begin{align*}
[\mathcal{L}^{z,(y_0,j)},\chi_1]\psi^j =2\sum_{k =1}^r \tilde{X}_k \chi_1 \tilde{X}_k \psi^j + (\sum_{k =1}^r \tilde{X}_k^2 \chi_1) \psi^j + (X_0\chi_1) \psi^j.
\end{align*}
Pairing the equation with $\chi_1\psi^j$ and integrating by parts gives
\begin{align*}
\norm{\chi_1\psi^j}_{\mathcal{H}} \lesssim R^2 \norm{\chi_0\psi^j}_{L^2(B(0,\delta))}.
\end{align*}
Pairing the equation with a test function $\varphi$ and integrating by parts again, gives the matching hypoelliptic estimate
\begin{align*}
\norm{\chi_1 \psi^j}_{\mathcal{H}^\ast} \lesssim R^2 \norm{\chi_1\psi^j}_{\mathcal{H}} \lesssim R^4 \norm{\chi_0\psi^j}_{L^2(B(0,\delta))}.
\end{align*}
Therefore, by Lemma \ref{lem:HineqBall_hypo}, we obtain
\begin{align*}
\norm{\chi_1 \psi^j}_{H^s} \lesssim R^{4+q} \norm{\chi_0\psi^j}_{L^2(B(0,\delta))}.
\end{align*}
Denote by $\brak{\grad}^s$ the Fourier multiplier given by
\begin{align*}
\widehat{\brak{\grad}^s f}(\xi) = (1 + \abs{\xi}^2)^{s/2} \hat{f}(\xi),
\end{align*}
where $\hat{g}(\xi) = \frac{1}{(2\pi)^{m/2}} \int_{\R^m} e^{-i y \cdot \xi} g(y) \dee y$ denotes the Fourier transform.
Note that it is classical that for any Schwartz class function $g$ on $\mathbb R^m$
\begin{align*}
\norm{g}_{H^s} \approx \norm{\brak{\grad}^s g}_{L^2}.
\end{align*}
To iterate, we will follow H\"ormander's approach in \cite{Hormander67} and apply the $\chi_{2}\brak{\grad}^{s}$ operator to both sides of \eqref{eq:LocKol} to obtain
\begin{align*}
\mathcal{L}^{z,(y_0,j)} \chi_2 \brak{\grad}^s (\chi_1\psi^j) & = - [\mathcal{L}^{z,(y_0,j)},\chi_2 \brak{\grad}^s] \chi_1\psi^j +  \chi_2 \brak{\grad}^s \left(pH^{(j)} \chi_1 \psi^j\right) \\ & \quad -\chi_2 \brak{\grad}^s [\mathcal{L}^{z,(y_0,j)},\chi_1]\psi^j + \chi_2 \brak{\grad}^s \left(pH^{(j)} \chi_1 \psi^j\right).
\end{align*}
Using standard commutator estimates for fractional derivatives and the quantitative hypoelliptic estimate Lemma \ref{lem:HineqBall_hypo}, we have
\begin{align*}
\norm{\chi_2 \brak{\grad}^s \chi_1 \psi^j}_{H^s} \lesssim R^{2+q+2s} \left(\norm{\chi_1\psi^j}_{H^s} + \norm{\chi_0\psi^j}_{L^2(B(0,\delta))} \right).
\end{align*}
By iterating this argument further, we have for some $q' > 0$,
\begin{align*}
\norm{\chi_{n_\ast} \brak{\grad}^s .... \brak{\grad}^s \chi_1 \psi^j}_{L^2} \lesssim_{n_\ast} R^{q'}\norm{\chi_0\psi^j}_{L^2(B(0,\delta))}.
\end{align*}
Next, recall the standard Sobolev space interpolation $\norm{\brak{\grad}^{sm} f}_{L^2} \lesssim \norm{\brak{\grad}^{s n_\ast} f}^\theta \norm{f}_{L^2}^{1-\theta}$ for any $m < n_\ast$ and some $\theta \in (0,1)$ depending on $s$, $m$, and $n_\ast$.
Applying this interpolation estimate and the standard commutator estimates for fractional derivatives, we therefore deduce
\begin{align*}
\norm{\brak{\grad}^{sn_\ast} (\chi_{n_\ast +1}\psi^j)}_{L^2} \lesssim_{n_\ast} R^{q'}\norm{\chi_0\psi^j}_{L^2(B(0,\delta))}.
\end{align*}
Therefore, by Sobolev embedding (using again that the geometry is uniformly bounded) and the choice of $n_\ast$, the desired estimate \eqref{ineq:LocGoal} holds.
\end{proof}

\appendix

\section{Super-Lyapunov estimate} \label{app:SuperLyap}

A key ingredient in establishing geometric ergodicity is controlling the behavior of the process $u_t$. This is achieved using the function $V_\eta(u) = e^{\eta |u|^2}$ for sufficiently small $\eta > 0$. This function acts as a Lyapunov function, ensuring the process does not escape to infinity. Specifically, it satisfies a \emph{drift condition} (sometimes referred to as a super-Lyapunov property due to its exponential form): there exists $\eta_0 > 0$ such that for all $\eta \in (0, \eta_0)$ and for all $\kappa > 0$, there exists a constant $C_\kappa \ge 0$ such that
\begin{align}
\mathcal{L} V_\eta \leq -\kappa V_\eta + C_\kappa, \label{ineq:LVeta_u}
\end{align}
where $\mathcal{L}$ is the generator of the Markov process $(u_t)$ given in \eqref{eq: L96-SDE}.
This property essentially guarantees that the process $u_t$ is positive recurrent and possesses moments related to $V_\eta$. Lemma \ref{lem:twisty} leverages this property, combined with It\^o calculus, to obtain strong estimates (related to large deviations) on the growth of quantities related to $V_\eta(u_t)$.

\begin{proof}[Proof of Lemma \ref{lem:twisty}]
The generator $\mathcal{L}$ for $u_t$ is
\[ \mathcal{L} \phi(u) = (B(u,u) - \eps u) \cdot \nabla \phi(u) + \frac{\eps}{2} \sum_{j \in I} \sigma_j^2 \partial^2_{u_j u_j} \phi(u). \]
Using the property $\langle B(u,u), u \rangle = 0$, a direct calculation yields
\begin{align*}
\mathcal{L} |u|^2 &= 2\langle B(u,u) - \eps u, u\rangle + \eps\sum_{j \in I} \sigma_j^2 \\
&= -2\eps |u|^2 + \eps C_\sigma,
\end{align*}
where $C_\sigma = \sum_{j \in I} |\sigma_j|^2$.

Let $\eta, \gamma \geq 0$, to be chosen later. By It\^o's formula applied to $X_t := \eta e^{\gamma t}|u|^2 = e^{\gamma t}\log V_\eta(u_t)$, the process
\begin{equation} \label{eq:Mt_def_L96_app}
	\begin{aligned}
M_t &:= X_t - X_0 - \int_0^t \eta e^{\gamma s}\left(\gamma|u_s|^2  + \mathcal{L}|u_s|^2\right)\ds\\
&= X_t - X_0 - (\gamma - 2\eps) \int_0^t X_s\ds - \eps \eta C_\sigma \frac{e^{\gamma t}-1}{\gamma}
	\end{aligned}
\end{equation}
is a continuous local martingale with $M_0=0$ and quadratic variation
\begin{align*}
\brak{M}_t
&= 4 \eta^2 \eps \int_0^t e^{2\gamma s} \sum_{j \in I} \sigma_j^2 (u_s)_j^2 \ds \\
&\leq 4 \eta^2 \eps \left(\max_{k \in I} |\sigma_k|^2\right) \int_0^t e^{2\gamma s} |u_s|^2 \ds = \mathcal{Q}_* \eta \eps \int_0^t e^{\gamma s} X_s \ds,
\end{align*}
where $\mathcal{Q}_* = 4 (\max_{k \in I} |\sigma_k|^2)$.

We use the exponential martingale inequality (see, e.g., \cite[Chapter IV, Corollary 3.4]{RevuzYor-Continuous-2010a}): for a continuous local martingale $M_t$ with $M_0=0$,
\begin{equation}\label{eq:exp_mart_ineq_L96_app}
\EE\left[\exp\left( \sup_{0 \le t \le T} (M_t - \brak{M}_t)\right)\right] \le 2.
\end{equation}
From \eqref{eq:Mt_def_L96_app}, using the bound on $\brak{M}_t$, we have for $t\in [0,T]$
\begin{align*}
M_t - \brak{M}_t &\geq X_t - X_0 - (\gamma - 2\eps) \int_0^t X_s\ds - \eps \eta C_\sigma \frac{e^{\gamma t}-1}{\gamma} - \mathcal{Q}_* \eta \eps \int_0^t e^{\gamma s} X_s \ds \\
&\geq X_t - X_0 + (2\eps- \gamma - \mathcal{Q}_* \eta \eps e^{\gamma T}) \int_0^t X_s\ds - \eps \eta C_\sigma \frac{e^{\gamma T}-1}{\gamma}.
\end{align*}
Choosing $\gamma_* = \eps$ and $\eta_* = \frac{1}{2\mathcal{Q}_*}$ ensures that $2\eps - \gamma - \mathcal{Q}_* \eta \eps e^{\gamma T} > \eps/2$ for $\eta$ and $\gamma$ satisfying $0\leq \gamma <\gamma_*$ and $0 \leq \eta e^{\gamma T} \leq \eta_*$.

Applying the exponential martingale inequality \eqref{eq:exp_mart_ineq_L96_app}, we have
\begin{equation}\label{eq:Martingale-estimate}
\EE\exp\left( \sup_{0 < t < T} \left(X_t + \frac{\eps}{2}\int_0^t X_s\ds\right)\right) \lesssim_{T,\gamma,\eta} \exp(X_0),
\end{equation}
or equivalently,
\[
\EE\left[\exp\left(\frac{\eta\eps}{2}\int_0^Te^{\gamma s} |u_s|^2 \ds\right) \sup_{0<t<T}V_{e^{\gamma t}\eta}(u_t)\right] \lesssim_{T,\gamma, \eta} V_\eta(u_0).
\]
The above estimate is in fact stronger than what is stated in Lemma~\ref{lem:twisty}. For any $c>0$, using Young's inequality, there exists a $C' = C'(c,\eta,\eps) >0$ such that
\[
c|u_t| \leq \frac{\eta\eps}{2}|u_t|^2 + C',
\]
which allows us to extract the desired $e^{c \int_0^T |u_s| \ds}$ factor in the lemma statement.

The uniform in $\epsilon$ estimate follows from the observation that the constant on the right hand side of \eqref{eq:Martingale-estimate} is of the form $C = \exp(\eta C_\sigma (e^{\eps T}- 1)) \leq \exp(\eta C_\sigma (e^T - 1))$, when $\epsilon \in (0,1]$.
\end{proof}
The above estimates provide moment estimates on the transverse matrix process.
\begin{corollary} \label{cor:twistyFC}
For any $T, p > 0$ and $\eta \in [0,\eta_0)$ there holds
\begin{align*}
\EE \sup_{t \in [0,T]} \left(\abs{A^t_\perp}^{-p} + \abs{A^{t}_\perp}^p\right) V(y_t) \lesssim_{p, T,\eta} V(y).
\end{align*}
and in particular
\begin{align}
\EE \sup_{t \in (0,1)} \int_{H_{I}} \log \abs{A^{t}_\perp} \dee \mu^I(u) + \EE \sup_{t \in (0,1)} \int_{H_{I}} \log \abs{(A^{t}_\perp)^{-1}} \dee \mu^I(u) <\infty. \label{ineq:Aintegrable}
\end{align}
\end{corollary}
\begin{proof}
It is easy to show by a Gr\"{o}nwall argument that the norm of the linearization $A_y^t$ can be bounded crudely by an exponential involving the integral of $|y_s|$:
\begin{align*}
	e^{-c \int_0^t \abs{y_s} \dee s}\leq \abs{A_\perp^{t}} \leq e^{c \int_0^t \abs{y_s} \dee s},
\end{align*}
where $c> 0$ is some constant. The estimate in Lemma \ref{lem:twisty} then yields the desired moment bounds on $\sup_{t \in [0,T]} (\abs{A^t}^{-p} + \abs{A^{t}}^p) V_\eta(y_t)$. The integrability condition \eqref{ineq:Aintegrable} follows similarly considering $\log |A^t|$ and integrating against $\mu^I$, using Lemma \ref{lem:twisty} leveraging the moment bounds provided by $V_\eta$ for $\mu^I$.
\end{proof}

\section{Appendix: Control Theory and Irreducibility}\label{app:control_theory}

This appendix establishes a sufficient condition for the topological irreducibility of the Stratonovich SDE
\begin{align}\label{eq:sde-appendix}
  \mathrm{d} x_t = X_0(x_t)\,\mathrm{d}t + \sum_{k=1}^r X_k(x_t)\circ \mathrm{d} W_t^k
\end{align}
on an analytic, connected manifold $M$. We assume $\{X_0, \dots, X_r\} \subset \mathfrak{X}(M)$ are analytic and complete vector fields, and that the SDE \eqref{eq:sde-appendix} admits a global flow. Specifically the goal of this section is to prove Proposition~\ref{prop:irreducibility-general}, which states that if the vector fields satisfy the restricted parabolic Hörmander condition and a certain cancellation condition, then the SDE is topologically irreducible.

This result is likely known among experts in SDEs and geometric control theory, and shares a lot of similarities to the setting of polynomial drifts (see e.g.~\cite{Jurdjevic1985-cr},~\cite{HerzogMattingly-Practical-2015y}). Nonetheless, we could not find a proof in the literature and so we provide a proof here for completeness.

The proof relies on connecting the SDE's properties to the controllability of an associated deterministic control system via the Stroock-Varadhan support theorem~\cite{Stroock1972-nc}. We follow the framework of geometric control theory, primarily based on the monograph by Jurdjevic~\cite{JurdGCT}.

\subsection{Control System and Controllability}

Consider the affine control system associated with \eqref{eq:sde-appendix}{}:
\begin{align}\label{eq:control-sys-appendix}
    \dot{x}_t = X_0(x_t) + \sum_{k=1}^r X_k(x_t) u^k_t,
\end{align}
where $u = (u^1, \dots, u^r)$ is a control function. The dynamics can be viewed as being generated by the specific family of vector fields
\[
\mathcal{F}_0 := \{X_0 + X \mid X \in \mathcal{X}\}, \quad \text{where } \mathcal{X} = \mathrm{span}\{X_1,\ldots,X_r\}.
\]
A trajectory of this system using piecewise constant controls is a curve obtained by concatenating integral curves $t \mapsto e^{tY}x$ for fields $Y \in \mathcal{F}_0$. The Lie algebra generated by $\mathcal{F}_0$ is denoted $\mathrm{Lie}(\mathcal{F}_0)$.

\begin{definition}[Accessibility and Controllability]
Let $x \in M$ and $t > 0$.
\begin{enumerate}
    \item The \textit{time-$t$ accessible set} from $x$ for the system $\mathcal{F}_0$ is
    \[
    \mathcal{A}_x^t(\mathcal{F}_0) := \left\{e^{t_nY_n}\cdots e^{t_1Y_1}x \mid Y_i \in \mathcal{F}_0, t_i>0, \sum t_i=t, n \ge 1\right\}.
    \]
    \item The \textit{set reachable by time $t$} from $x$ for $\mathcal{F}_0$ is $\mathcal{A}_x^{\leq t}(\mathcal{F}_0) := \bigcup_{0 < s \leq t} \mathcal{A}_x^s(\mathcal{F}_0)$.
    \item The system $\mathcal{F}_0$ is \textit{strongly controllable} if $\mathcal{A}_x^{\leq t}(\mathcal{F}_0) = M$ for all $x \in M, t > 0$.
    \item The system $\mathcal{F}_0$ is \textit{exactly controllable} if $\mathcal{A}_x^t(\mathcal{F}_0) = M$ for all $x \in M, t>0$.
\end{enumerate}
\end{definition}

The link between the SDE's support and the control system's reachability is given by the Support Theorem:

\begin{theorem}[Support Theorem~\cite{Stroock1972-nc}]\label{thm:support}
If the family $\mathcal{F}_0$ associated with the SDE \eqref{eq:sde-appendix}{} (with analytic vector fields and global flow) is exactly controllable, then the process $(x_t)$ is topologically irreducible, meaning $P(x_t \in O \mid x_0 = x) > 0$ for all $x \in M$, $t>0$, and any non-empty open set $O \subset M$.
\end{theorem}

Our strategy is thus to find conditions ensuring exact controllability of $\mathcal{F}_0$.

\subsection{Lie Saturate and Strong Controllability}

The concept of the Lie saturate is crucial for analyzing strong controllability. Below, $\mathrm{cl}(K)$ denotes the topological closure of a set $K \subset M$.

\begin{definition}[Lie Saturate]
Two families $\mathcal{F}, \mathcal{G} \subseteq \mathfrak{X}(M)$ are equivalent, denoted $\mathcal{F} \sim \mathcal{G}$, if $\mathrm{cl}(\mathcal{A}_x^{\leq t}(\mathcal{F})) = \mathrm{cl}(\mathcal{A}_x^{\leq t}(\mathcal{G}))$ for all $x \in M, t > 0$. The (strong) \textit{Lie saturate} of a family $\mathcal{F}$, denoted $\mathrm{LS}(\mathcal{F})$, is the largest subset of $\mathrm{Lie}(\mathcal{F})$ equivalent to $\mathcal{F}$.
\end{definition}

\begin{theorem}[\cite{JurdGCT}, Ch~3, Thm~12]\label{thm:LS-strong-control}
The family $\mathcal{F}_0$ is strongly controllable if and only if its Lie saturate spans the tangent space everywhere: $\mathrm{LS}(\mathcal{F}_0)(x) := \{Y(x) \mid Y\in \mathrm{LS}(\mathcal{F}_0)\} = T_xM$ for all $x \in M$.
\end{theorem}

For analytic vector fields, the Lie saturate has important structural properties:
\begin{proposition}[Properties of LS~\cite{JurdGCT}, Ch~3]\label{prop:LS-properties}
Let $\mathcal{F} \subseteq \mathfrak{X}(M)$ be a family generated by analytic vector fields. Then $\mathrm{LS}(\mathcal{F})$ satisfies:
\begin{enumerate}
    \item {\em Convexity and Closure:} $\mathrm{LS}(\mathcal{F})$ is a closed convex cone in the $C^\infty(M)$ topology.
    \item {\em Lie Subalgebra Generation:} If $\mathcal{V} \subseteq \mathrm{LS}(\mathcal{F})$ is a vector subspace, then $\mathrm{Lie}(\mathcal{V}) \subseteq \mathrm{LS}(\mathcal{F})$.
    \item {\em Invariance under Flows:} If $\pm X \in \mathrm{LS}(\mathcal{F})$ and $Y \in \mathrm{LS}(\mathcal{F})$, then the pushforward $(e^{\alpha X})_\sharp Y \in \mathrm{LS}(\mathcal{F})$ for all $\alpha \in \mathbb{R}$.
\end{enumerate}
\end{proposition}
\textit{Note:} Above, for a diffeomorphism $\phi : M \to M$ we have written $\phi_\sharp$ for the \emph{pushforward} by $\phi$, given at $y \in M$ by $(\phi_\sharp Y)(y) := d\phi_{\phi^{-1}(y)}(Y(\phi^{-1}(y)))$.

\subsection{Zero-Time Ideal}

While the Lie saturate helps characterize strong controllability (Theorem~\ref{thm:LS-strong-control}). Theorem~\ref{thm:support} requires exact controllability to establish topological irreducibility. The concept of the zero-time ideal provides a link between these two notions of controllability. For analytic vector fields, the structure of the tangent space directions reachable in arbitrarily small time is captured by the zero-time ideal. In essence, while the Lie saturate captures the directions reachable in finite time, the zero-time ideal captures the directions reachable in infinitesimal time. Heuristically, this allows us first reach a certain point before time $t$ and then to ``dither'' in place until time $t$.

\begin{definition}[Derived Algebra and Zero-Time Ideal,~\cite{JurdGCT} Ch~2, Def~11--12]\label{def:zero-time-ideal}
Let $\mathcal{F} \subseteq \mathfrak{X}(M)$ be a family of analytic vector fields.
\begin{enumerate}
    \item The \textit{derived algebra} of $\mathcal{F}$, denoted $D(\mathcal{F})$, is the ideal of $\mathrm{Lie}(\mathcal{F})$ generated by all iterated brackets of elements of $\mathcal{F}$ (but not the elements of $\mathcal{F}$ themselves). That is,
    \[
	D(\mathcal{F}) = \mathrm{span}\left\{\mathrm{ad}_{Z_{1}}\mathrm{ad}_{Z_2}\ldots \mathrm{ad}_{Z_{k-1}}Z_k \mid k \ge 2, Z_i \in \mathcal{F}\right\}.
	\]
    \item The \textit{zero-time ideal} of $\mathcal{F}$, denoted $I(\mathcal{F})$, is the linear span of $D(\mathcal{F})$ and all differences $X-Y$ where $X, Y \in \mathcal{F}$. That is,
    \[ I(\mathcal{F}) = \mathrm{span}\{D(\mathcal{F}) \cup \{X-Y \mid X, Y \in \mathcal{F}\}\}. \]
\end{enumerate}
 We note that the zero-time ideal $I(\mathcal{F})$ is a Lie ideal\footnote{Let $\mathfrak{g}$ be a real Lie algebra and let $\mathfrak{h} \subset \mathfrak{g}$ be a sub-Lie algebra. We say that $\mathfrak{h}$ is a \emph{Lie algebra ideal}, or \emph{Lie ideal}  for short, if $[\mathfrak{h}, \mathfrak{g}] \subset \mathfrak{h}$.  } of $\mathrm{Lie}(\mathcal{F})$. We denote its evaluation at $x$ by $I(\mathcal{F})(x) = \{Y(x) \mid Y \in I(\mathcal{F})\}$.
\end{definition}

For our specific system $\mathcal{F}_0 = X_0+\mathrm{span}\{X_i \mid i=1,\ldots r\}$, the differences span $\mathcal{X} = \mathrm{span}\{X_i \mid i=1,\ldots r\}$. This provides a convenient identification of $I(\mathcal{F}_0)$ with the Lie algebra generated by $\mathcal{S} = \{ \mathrm{ad}_{X_0}^k X_j | 1 \leq j \leq r, k \geq 0\}$ appearing in the parabolic H\"ormander condition (Definition \ref{def:parabolic_hormander}).
\begin{lemma}\label{lem:zero-ideal-relation} Let $\mathcal{F}_0 = X_0+\mathrm{span}\{X_i \mid i=1,\ldots r\}$.
Then,
\[
I(\mathcal{F}_0) = \mathrm{Lie}(\mathcal{S}) \, .
\]
\end{lemma}
\begin{proof}
We show the equality $I(\mathcal{F}_0) = \mathrm{Lie}(\mathcal{S})$ by demonstrating both inclusions. Note that $I(\mathcal{F}_0) = \mathrm{span}(D(\mathcal{F}_0) \cup \mathcal{X})$ and $\mathrm{Lie}(\mathcal{S})$ is the ideal generated by $\mathcal{X}$ in $\mathrm{Lie}(\mathcal{F}_0) = \mathrm{Lie}(X_0, \mathcal{X})$.

First we show $\mathrm{Lie}(\mathcal{S}) \subseteq I(\mathcal{F}_0)$: Since $I(\mathcal{F}_0)$ is an ideal containing $\mathcal{X}$, it must contain the smallest ideal containing $\mathcal{X}$, which is $\mathrm{Lie}(\mathcal{S})$.

Next we show $I(\mathcal{F}_0) \subseteq \mathrm{Lie}(\mathcal{S})$: We need to show $\mathcal{X} \subseteq \mathrm{Lie}(\mathcal{S})$ and $D(\mathcal{F}_0) \subseteq \mathrm{Lie}(\mathcal{S})$. Clearly, $\mathcal{X} \subseteq \mathrm{Lie}(\mathcal{S})$,  while, $D(\mathcal{F}_0)$ is generated by brackets $[Z_1, Z_2]$ where $Z_i = X_0+Y_i$ with $Y_i \in \mathcal{X}$. We have
\[
[Z_1, Z_2] = [X_0+Y_1, X_0+Y_2] = [X_0, Y_2] - [X_0, Y_1] + [Y_1, Y_2].
\]
Since $Y_1, Y_2 \in \mathcal{X} \subseteq \mathrm{Lie}(\mathcal{S})$ and $\mathrm{Lie}(\mathcal{S})$ is an ideal in $\mathrm{Lie}(X_0, \mathcal{X})$, all three terms $[X_0, Y_2]$, $[X_0, Y_1]$, and $[Y_1, Y_2]$ belong to $\mathrm{Lie}(\mathcal{S})$. Thus, $[Z_1, Z_2] \in \mathrm{Lie}(\mathcal{S})$. Since $D(\mathcal{F}_0)$ is the ideal generated by such brackets, $D(\mathcal{F}_0) \subseteq \mathrm{Lie}(\mathcal{S})$. Therefore, $I(\mathcal{F}_0) = \mathrm{span}(\mathcal{X} \cup D(\mathcal{F}_0)) \subseteq \mathrm{Lie}(\mathcal{S})$.
\end{proof}

The zero-time ideal plays a key role in connecting strong and exact controllability for analytic systems.

\begin{theorem}[\cite{JurdGCT} Ch~3, Thm~13b]\label{thm:strong-implies-exact-via-ideal}
Let $\mathcal{F}$ be a family of analytic vector fields on $M$. If $\mathcal{F}$ is strongly controllable and $I(\mathcal{F})(x) = \mathrm{Lie}(\mathcal{F})(x)$ for all $x \in M$, then $\mathcal{F}$ is exactly controllable.
\end{theorem}

We combine these results into a practical criterion for our system $\mathcal{F}_0$. Let $\mathcal{S}_1 = \{X_k, [X_0, X_k] : 1 \le k \le r \}$. Recall from Section~\ref{sec:Hypoellipticity} (specifically Definition~\ref{def:parabolic_hormander}) that the restricted parabolic Hörmander condition requires $\mathrm{Lie}(\mathcal{S}_1)(x) = T_xM$. Note that $\mathrm{Lie}(\mathcal{S}_1) = \mathrm{Lie}(\mathcal{X}, [X_0, \mathcal{X}])$.

\begin{corollary}\label{cor:sufficient-control}
Assume the vector fields $\{X_0, \dots, X_r\}$ are analytic. If
\begin{enumerate}
    \item the restricted parabolic Hörmander condition holds: $\mathrm{Lie}(\mathcal{X}, [X_0, \mathcal{X}])(x) = T_xM$ for all $x \in M$ and
    \item $\mathrm{Lie}(\mathcal{X}, [X_0, \mathcal{X}]) \subseteq \mathrm{LS}(\mathcal{F}_0)$,
\end{enumerate}
then the system $\mathcal{F}_0$ is exactly controllable.
\end{corollary}
\begin{proof}
Assumption~(1) implies $\mathrm{Lie}(\mathcal{S}_1)(x) = T_x M$. Since $\mathrm{Lie}(\mathcal{S}_1) \subseteq I(\mathcal{F}_0) \subseteq \mathrm{Lie}(\mathcal{F}_0)$ we have that $I(\mathcal{F}_0)(x) = \mathrm{Lie}(\mathcal{F}_0)(x) = T_x M$.
Therefore assumption~(2) combined with Assumption~(1) implies $\mathrm{LS}(\mathcal{F}_0)(x) = T_x M$. By Theorem~\ref{thm:LS-strong-control}, the system $\mathcal{F}_0$ is strongly controllable.
Since $\mathcal{F}_0$ consists of analytic fields, is strongly controllable, and satisfies $I(\mathcal{F}_0)(x) = \mathrm{Lie}(\mathcal{F}_0)(x) = \mathrm{Lie}(\mathcal{S}_1)(x) = T_xM$, Theorem~\ref{thm:strong-implies-exact-via-ideal} implies that $\mathcal{F}_0$ is exactly controllable.
\end{proof}

\subsection{Irreducibility under Cancellation Condition}

We now state and prove the main result, providing sufficient conditions for the irreducibility of the SDE \eqref{eq:sde-appendix}{}.

\begin{proposition}\label{prop:irreducibility-general}
Let $\{X_0, \dots, X_r\}$ be analytic, complete vector fields on $M$ such that the SDE \eqref{eq:sde-appendix}{} has a global flow. Assume:
\begin{enumerate}
    \item The restricted parabolic Hörmander condition holds: $\mathrm{Lie}(\mathcal{X}, [X_0, \mathcal{X}])(x) = T_xM$ for all $x \in M$ (Definition~\ref{def:parabolic_hormander} in Section~\ref{sec:Hypoellipticity}).
    \item The cancellation condition\footnote{The cancellation condition in assumption~(2) can be slightly generalized. It is sufficient for the condition $\mathrm{ad}(Y_k)^2 X_0 = 0$ to hold for \textit{some} set of vector fields $\{Y_1, \dots, Y_r\}$ that spans $\mathcal{X} = \mathrm{span}\{X_1, \dots, X_r\}$. The proof proceeds by showing $\pm [Y_k, X_0] \in \mathrm{LS}(\mathcal{F}_0)$ for this spanning set, which implies $[X_0, \mathcal{X}] \subseteq \mathrm{LS}(\mathcal{F}_0)$ as required.} holds: $\mathrm{ad}(X_k)^2 X_0 = [X_k, [X_k, X_0]] = 0$ for all $k=1,\dots,r$.
\end{enumerate}
Then the SDE \eqref{eq:sde-appendix}{} is topologically irreducible.

\end{proposition}
\begin{proof}
By Theorem~\ref{thm:support} and Corollary~\ref{cor:sufficient-control}, the proof reduces to demonstrating the inclusion
\(
\mathrm{Lie}(\mathcal{X},[X_0, \mathcal{X}]) \subseteq \mathrm{LS}(\mathcal{F}_0)
\)
under the given assumptions. We use the properties of $\mathrm{LS}(\mathcal{F}_0)$ listed in Proposition~\ref{prop:LS-properties}.

\textit{Step 1: Show $\mathcal{X} \subseteq \mathrm{LS}(\mathcal{F}_0)$.}
For any $X_k \in \mathcal{X}$ and $\alpha \in (0, 1]$, the vector field $Y_\alpha = X_0+\alpha^{-1}X_k$ belongs to $\mathcal{F}_0$, and thus $Y_\alpha \in \mathrm{LS}(\mathcal{F}_0)$. By convexity (Proposition~\ref{prop:LS-properties}.1), the scaled field $\alpha Y_\alpha = \alpha X_0 + X_k$ also lies in $\mathrm{LS}(\mathcal{F}_0)$. Since $\mathrm{LS}(\mathcal{F}_0)$ is closed (Proposition~\ref{prop:LS-properties}.1), we can take the limit as $\alpha \to 0^+$:
\begin{align*}
X_k = \lim_{\alpha\to 0^+} (\alpha X_0 + X_k) \in \mathrm{LS}(\mathcal{F}_0).
\end{align*}
As $\mathrm{LS}(\mathcal{F}_0)$ is convex and closed, it contains the linear span $\mathcal{X} = \mathrm{span}\{X_1, \dots, X_r\}$. We can ensure $\pm X_k \in \mathrm{LS}(\mathcal{F}_0)$, hence the vector space $\mathcal{X} \subseteq \mathrm{LS}(\mathcal{F}_0)$.

\textit{Step 2: Show $[X_0, \mathcal{X}] \subseteq \mathrm{LS}(\mathcal{F}_0)$.}
Let $k \in \{1, \dots, r\}$. From Step~1, we know $\pm X_k \in \mathrm{LS}(\mathcal{F}_0)$. Since $X_0 \in \mathcal{F}_0$, we have $X_0 \in \mathrm{LS}(\mathcal{F}_0)$. By Proposition~\ref{prop:LS-properties}.3 (invariance under flows), the pushforward $(e^{t X_k})_\sharp X_0$ belongs to $\mathrm{LS}(\mathcal{F}_0)$ for all $t \in \mathbb{R}$.
The Baker-Campbell-Hausdorff formula for the pushforward, truncated by the cancellation condition $\mathrm{ad}(X_k)^2 X_0 = 0$, gives:
\begin{align*}
Y(t) := (e^{t X_k})_\sharp X_0 = X_0 + t [X_k, X_0].
\end{align*}
Thus, $Y(t) \in \mathrm{LS}(\mathcal{F}_0)$ for all $t \in \mathbb{R}$.
Since $X_k \in \mathrm{LS}(\mathcal{F}_0)$ and $-X_k \in \mathrm{LS}(\mathcal{F}_0)$, their convex combination $0 = \frac{1}{2}X_k + \frac{1}{2}(-X_k)$ is in $\mathrm{LS}(\mathcal{F}_0)$ by Proposition~\ref{prop:LS-properties}.1 (Convexity).
Now, for any $t \ne 0$ and any $\alpha \in (0, 1]$, consider
\[ Z_\alpha(t) := \alpha Y(t) = \alpha X_0 + \alpha t [X_k, X_0]. \]
This $Z_\alpha(t)$ belongs to $\mathrm{LS}(\mathcal{F}_0)$ by the cone property of $\mathrm{LS}(\mathcal{F}_0)$ (Proposition~\ref{prop:LS-properties}.1).
Let $t=1/\alpha$. Then for $\alpha \in (0, 1]$,
\[ Z_\alpha(1/\alpha) = \alpha X_0 + [X_k, X_0] \in \mathrm{LS}(\mathcal{F}_0). \]
Since $\mathrm{LS}(\mathcal{F}_0)$ is closed (Proposition~\ref{prop:LS-properties}.1), we can take the limit as $\alpha \to 0^+$:
\[ \lim_{\alpha \to 0^+} Z_\alpha(1/\alpha) = \lim_{\alpha \to 0^+} (\alpha X_0 + [X_k, X_0]) = [X_k, X_0] \in \mathrm{LS}(\mathcal{F}_0). \]
Similarly, starting with $Y(-t) = X_0 - t [X_k, X_0] \in \mathrm{LS}(\mathcal{F}_0)$ for $t>0$, we can conclude by the same argument that  $-[X_k, X_0] \in \mathrm{LS}(\mathcal{F}_0)$.

Thus, we have shown that $\pm [X_k, X_0] \in \mathrm{LS}(\mathcal{F}_0)$ for each $k=1, \dots, r$.
As argued in Step~1, since $\mathrm{LS}(\mathcal{F}_0)$ is convex and closed and contains $\pm [X_k, X_0]$, it must contain the vector space $\mathrm{span}\{[X_k, X_0]\}$.
Therefore, the vector space generated by all such brackets, $[X_0, \mathcal{X}] = \mathrm{span}\{[X_0, X_k] \mid k=1,\dots,r\}$, is contained in $\mathrm{LS}(\mathcal{F}_0)$.

\textit{Step 3: Conclude $\mathrm{Lie}(\mathcal{X},[X_0, \mathcal{X}]) \subseteq \mathrm{LS}(\mathcal{F}_0)$.}
From Step~1, the vector space $\mathcal{X} \subseteq \mathrm{LS}(\mathcal{F}_0)$. From Step~2, the vector space $[X_0, \mathcal{X}] \subseteq \mathrm{LS}(\mathcal{F}_0)$. Let $V = \mathcal{X} \oplus [X_0, \mathcal{X}]$. Since $V$ is a vector subspace contained in $\mathrm{LS}(\mathcal{F}_0)$, Proposition~\ref{prop:LS-properties}.2 implies that the Lie algebra generated by $V$ is also contained in the saturate:
\[
\mathrm{Lie}(V) = \mathrm{Lie}(\mathcal{X}, [X_0, \mathcal{X}]) \subseteq \mathrm{LS}(\mathcal{F}_0).
\]
This establishes Assumption~(2) of Corollary~\ref{cor:sufficient-control} and therefore completes the proof.
\end{proof}

\section{Computer Assisted Proof for Algebraic Generation}\label{app:CAP}
This appendix provides the detailed proof that the Lie algebra generated by the matrices $\{M_k : k \in I\}$ is $\mathfrak{sl}(H_I^\perp)$ for $N=3K$ with $K \ge 3$. Recall $I = \{3, 6, \dots, N\}$ and $T = \Z_N \setminus I$. The dimension of $H_I^\perp$ is $n = |T| = N - K = 2K$.

\subsection{Matrix Representation and Re-indexing}

Let $M_k = DB(e_k)|_{H_I^\perp}$ be the restriction of the linearization $DB(e_k)$ to the transverse space $H_I^\perp$. We compute its matrix elements $(M_k)_{\ell, m}$ with respect to the standard basis $\{e_j\}_{j \in T}$. Recall $(DB(u)v)_\ell = (v_{\ell+1} - v_{\ell-2})u_{\ell-1} + (u_{\ell+1} - u_{\ell-2})v_{\ell-1}$. Setting $u=e_k$ ($k \in I$) and $v=e_m$ ($m \in T$), the $\ell$-th component is
\[
(DB(e_k)e_m)_\ell = ((e_m)_{\ell+1} - (e_m)_{\ell-2})(e_k)_{\ell-1} + ((e_k)_{\ell+1} - (e_k)_{\ell-2})(e_m)_{\ell-1}.
\]
 The matrix element $(M_k)_{\ell, m}$ is the coefficient of $e_\ell$ ($\ell \in T$) in $DB(e_k)e_m$:
\begin{align*}
(M_k)_{\ell, m} &= {(\delta_{m, \ell+1} - \delta_{m, \ell-2})}\delta_{k, \ell-1} + {(\delta_{k, \ell+1} - \delta_{k, \ell-2})}\delta_{m, \ell-1} \\
&= \delta_{k, \ell-1}\delta_{m, \ell+1} - \delta_{k, \ell-1}\delta_{m, \ell-2} + \delta_{k, \ell+1}\delta_{m, \ell-1} - \delta_{k, \ell-2}\delta_{m, \ell-1}.
\end{align*}
Here, all indices are modulo $N$. Let $E_{\ell, m}$ denote the elementary matrix with a 1 at position $(\ell, m)$ and 0 elsewhere. Then $M_k$ is a sum of at most four elementary matrices:
\begin{equation}\label{eq:Mk_elementary_app_restructured}
M_k = E_{k+1, k+2} - E_{k+1, k-1} + E_{k-1, k-2} - E_{k+2, k+1}.
\end{equation}
(Indices $\ell, m$ must be in $T$. Since $k \in I$, $k\pm 1, k\pm 2 \in T$, ensuring these $E_{\ell,m}$ are well-defined within the matrix space for $H_I^\perp$).

For computational convenience and to simplify the description of symmetries, we re-index the transverse basis ${\{e_j\}}_{j \in T}$ using the natural ordering:
\[
T = \{1, 2, 4, 5, \ldots, N-2, N-1\} \quad \rightarrow \quad \{1, 2, 3, 4, \ldots, 2K-1, 2K\}.
\]
From now on, we identify $H_I^\perp \simeq \R^{2K}$ and view $M_k$ as a $2K \times 2K$ matrix acting on $\R^{2K}$. We denote the elementary matrices in this re-indexed space also by $E_{i,j}$ for $i,j \in \{1, \dots, 2K\}$.

\subsection{Shift Invariance and Finite Truncation}

\begin{lemma}[Shift Invariance]\label{lem:shift_invariance_cap}
Let $P$ be the permutation matrix implementing the index shift $j \mapsto j-2 \pmod{2K}$ in the re-indexed space $\R^{2K}$. Then, for $k \in I$, we have $M_{k+3} = P M_k P^{-1}$. Consequently, the Lie algebra $\mathfrak{g} = \mathrm{Lie}(\{M_k : k \in I\})$ is invariant under conjugation by $P$: if $A \in \mathfrak{g}$, then $PAP^{-1} \in \mathfrak{g}$.
\end{lemma}
\begin{proof}[Proof Sketch]
This follows from the structure of the bilinear form $B(u,v)_j = (u_{j+1} - u_{j-2})v_{j-1}$ and the definition of $M_k$. A shift $k \mapsto k+3$ in the first argument of $DB(e_k)e_m$ corresponds to shifting all indices in the calculation by 3. When restricted to the transverse indices $T$ and re-indexed to $\{1, \dots, 2K\}$, this corresponds to the permutation $P$ acting by conjugation.
\end{proof}

The matrix $M_k$ has a finite support window relative to the index $k$. Iterated brackets $[M_k, M_l]$ have growing but bounded support windows for fixed bracket depth. This means that generating specific elementary matrices $E_{i,j}$ with small indices $i,j$ using brackets of $M_3, M_6, M_9$ depends only on a local block of indices. Provided $N$ is large enough -- for our purposes, $N \geq 15, 2K \ge 10$ suffices --  this calculation is independent of the exact value of $N$.

\subsection{Computer-Assisted Generation Result}

Using symbolic computation (Sympy) with exact rational arithmetic for $N=15$, we compute iterated Lie brackets of $M_3, M_6, M_9$ up to depth $5$. Let $\mathcal{B}$ be the set of all generated matrices. We then form a matrix $S$ whose rows are vectorized versions of matrices in $\mathcal{B}$. Computing the reduced row echelon form of $S$ allows us to identify the elementary matrices in the span of $\mathcal{B}$.
The computation verifies the following:
\begin{proposition}[CAP Result]\label{prop:cap_result}
For $N=15$ ($2K=10$), the Lie algebra $\mathfrak{g} = \mathrm{Lie}(\{M_k : k \in I\})$ contains the elementary matrices $E_{3,2}$, $E_{4,3}$, and $E_{5,4}$ (using the re-indexed basis $\{1, \dots, {\color{blue} 9,10  } \}$).
\end{proposition}
The verification code is available as an IPython notebook in the public GitHub repository \cite{L96CAPGithub}. It can easily be run in a local environment with the required packages installed or directly online via Google Colaboratory \cite{L96CAPColab}. Due to the finite truncation argument, the CAP argument holds for all $N \ge 15$.
\subsection{Generation of \texorpdfstring{$\mathfrak{sl}(H_I^\perp)$}{sl(H_I^perp}}

We now combine the CAP result with the shift invariance to show that $\mathfrak{g}$ contains a known generating set for $\mathfrak{sl}_{2K}(\mathbb{R})$.

\begin{lemma}[Generation of $\mathcal{G}$]\label{lem:gen_G_cap}
Let $\mathfrak{g} = \mathrm{Lie}(\{M_k : k \in I\})$ for $N \ge 9$. Then $\mathfrak{g}$ contains the set
\[
\mathcal{G} = \{E_{j+1, j} : j=1, \dots, 2K-1\} \cup \{E_{1, 2K}\}.
\]
\end{lemma}
\begin{proof}
The cases $N = 9, 12$ can be treated by direct computation, either computer-assisted or by hand. These cases are omitted, and from here on we assume $N \geq 15$.

By Lemma \ref{lem:shift_invariance_cap}, the Lie algebra $\mathfrak{g} = \mathrm{Lie}(\{M_k : k \in I\})$ is invariant under conjugation by the shift permutation $P$ (index map $j \mapsto j-2 \pmod{2K}$) and its inverse $P^{-1}$ (index map $j \mapsto j+2 \pmod{2K}$). That is, if $A \in \mathfrak{g}$, then $PAP^{-1} \in \mathfrak{g}$ and $P^{-1}AP \in \mathfrak{g}$.
The action of these conjugations on an elementary matrix $E_{i,j}$ is given by:
\begin{align}
P E_{i,j} P^{-1} &= E_{i-2, j-2} \label{eq:conj_P} \\
P^{-1} E_{i,j} P &= E_{i+2, j+2} \label{eq:conj_Pinv}
\end{align}
where all indices are interpreted modulo $2K$.

We assume the result from the computer-assisted proof (Proposition \ref{prop:cap_result}), which states that for $N \ge 15$, the set $\{E_{3,2}, E_{4,3}, E_{5,4}\}$ is contained in $\mathfrak{g}$.

We first generate the ``wrap-around'' elements $E_{2,1}$ and $E_{1, 2K}$. Applying~\eqref{eq:conj_P}:
\begin{align*}
E_{2,1} &= P E_{4,3} P^{-1} \in \mathfrak{g} \\
E_{1, 2K} &= P E_{3,2} P^{-1} \in \mathfrak{g}
\end{align*}
Since $E_{5,4} \in \mathfrak{g}$, we can generate the remaining sub-diagonal elements $E_{j+1, j}$ by repeatedly applying conjugation by $P^{-1}$ according to~\eqref{eq:conj_Pinv}:
\begin{align*}
E_{7,6} &= P^{-1} E_{5,4} P \in \mathfrak{g} \\
E_{9,8} &= P^{-1} E_{7,6} P = {(P^{-1})}^2 E_{5,4} P^2 \in \mathfrak{g} \\
&\vdots \\
E_{j+1, j} &= {(P^{-1})}^{(j-5)/2} E_{5,4} P^{(j-5)/2} \in \mathfrak{g} \quad \text{for odd } j \ge 5.
\end{align*}
Similarly, starting from $E_{4,3} \in \mathfrak{g}$:
\begin{align*}
E_{6,5} &= P^{-1} E_{4,3} P \in \mathfrak{g} \\
E_{8,7} &= P^{-1} E_{6,5} P = {(P^{-1})}^2 E_{4,3} P^2 \in \mathfrak{g} \\
&\vdots \\
E_{j+1, j} &= {(P^{-1})}^{(j-4)/2} E_{4,3} P^{(j-4)/2} \in \mathfrak{g} \quad \text{for even } j \ge 4.
\end{align*}
Combining the known elements $\{E_{2,1}, E_{3,2}, E_{4,3}\}$ with those generated above for $j \ge 4$, we have shown that all $E_{j+1, j}$ for $j=1, \dots, 2K-1$ are in $\mathfrak{g}$.
Since we also showed $E_{1, 2K} \in \mathfrak{g}$, the entire set
\[
\mathcal{G} = \{E_{j+1, j} : j=1, \dots, 2K-1\} \cup \{E_{1, 2K}\}
\]
is contained in $\mathfrak{g}$.
\end{proof}

\begin{proposition}[Standard Generating Set for $\mathfrak{sl}_{n}$]\label{prop:sl_gen_set}
The set $\mathcal{G} = \{E_{j+1, j} : j=1, \dots, n-1\} \cup \{E_{1, n}\}$ is a generating set for the special linear Lie algebra $\mathfrak{sl}_{n}(\mathbb{R})$ for $n \ge 2$.
\end{proposition}
\begin{proof}
This is a standard result in the theory of Lie algebras, see e.g.,~\cite[Chapter VIII, Section 4, Theorem 9]{Jacobson1962lie}.
\end{proof}

\begin{proof}[Proof of Proposition~\ref{prop:slT_generation}]
Let $\mathfrak{g} = \mathrm{Lie}(\{M_k : k \in I\})$. By Proposition \ref{prop:cap_result} and Lemma \ref{lem:gen_G_cap}, we know that $\mathcal{G} \subset \mathfrak{g}$, where $\mathcal{G}$ is the generating set defined in Lemma \ref{lem:gen_G_cap} with $n=2K$. By Proposition \ref{prop:sl_gen_set}, the Lie algebra generated by $\mathcal{G}$ is $\mathfrak{sl}_{2K}(\mathbb{R})$. Therefore, $\mathfrak{sl}_{2K}(\mathbb{R}) = \mathrm{Lie}(\mathcal{G}) \subseteq \mathfrak{g}$.
Furthermore, each $M_k$ is traceless (as can be verified from \eqref{eq:Mk_elementary_app_restructured}), so the generated Lie algebra $\mathfrak{g}$ must be a subalgebra of $\mathfrak{sl}_{2K}(\mathbb{R})$.
Combining $\mathfrak{sl}_{2K}(\mathbb{R}) \subseteq \mathfrak{g}$ and $\mathfrak{g} \subseteq \mathfrak{sl}_{2K}(\mathbb{R})$, we conclude that $\mathfrak{g} = \mathfrak{sl}_{2K}(\mathbb{R}) \simeq \mathfrak{sl}(H_I^\perp)$.
\end{proof}

\phantomsection
\addcontentsline{toc}{section}{References}
\bibliographystyle{abbrv}
\bibliography{bibliography}

\end{document}